\documentclass{article}[12pt]

\usepackage{amsthm}
\usepackage{mathrsfs}
\usepackage{amsmath}
\usepackage{amsfonts}
\usepackage{enumerate}
\usepackage{amssymb}
\usepackage{dsfont}
\usepackage{color}
\usepackage{mathrsfs}
\usepackage{systeme}
\usepackage[natbibapa]{apacite}
\usepackage{float}

\usepackage{emptypage}

\usepackage{hyperref}                                                                                                                 
 \hypersetup{colorlinks,citecolor=blue,filecolor=blue,linkcolor=blue,urlcolor=blue}
\usepackage[pdftex]{graphicx}
\usepackage{float}
\usepackage{enumitem}
\allowdisplaybreaks

\def\P{\mathbb{P}}
\def\R{\mathbb{R}}
\def\F{\mathcal{F}}
\def\E{\mathbb{E}}

\def \I{\mathbb{I}}

\def\e{\mathbf{e} }
\def \dd{\textup{d}}

\def \sign{\text{sign}}

\newtheorem{deff}{Definition}[section]
\newtheorem{thm}[deff]{Theorem}

\newtheorem{prop}[deff]{Proposition}
\newtheorem{lemma}[deff]{Lemma}
\newtheorem{cor}[deff]{Corollary}
\newtheorem{rem}[deff]{Remark}
\newtheorem{example}[deff]{Example}

\title{On the last zero process with an application in corporate bankruptcy}
\author{Erik J. Baurdoux\footnote{Department of Statistics, London School of Economics and Political Science. Houghton Street, {\sc London, WC2A 2AE, United Kingdom.} E-mail: e.j.baurdoux@lse.ac.uk} \quad \& \quad Jos\'e M. Pedraza\footnote{School of Mathematics, The University of Manchester. Oxford Road, {\sc Manchester, M13 9PL,  United Kingdom.} E-mail: jose.pedrazaramirez@manchester.ac.uk} }

\begin{document}
\maketitle
\begin{abstract}
\noindent
For a spectrally negative L\'evy process $X$, consider $g_t$, the last time $X$ is below the level zero before time $t\geq 0$. We use a perturbation method for L\'evy processes to derive an It\^o formula for the three-dimensional process $\{(g_t,t, X_t), t\geq 0 \}$ and its infinitesimal generator. Moreover, with $U_t:=t-g_t$, the length of a current positive excursion, we derive a general formula that allows us to calculate a functional of the whole path of $ (U, X)=\{(U_t, X_t),t\geq 0\}$ in terms of the positive and negative excursions of the process $X$. As a corollary, we find the joint Laplace transform of $(U_{\mathbf{e}_q}, X_{\mathbf{e}_q})$, where $\mathbf{e}_q$ is an independent exponential time, and the q-potential measure of the process $(U, X)$. Furthermore, using the results mentioned above, we find a solution to a general optimal stopping problem depending on $(U, X)$ with an application in corporate bankruptcy. Lastly, we establish a link between the optimal prediction of $g_{\infty}$ and optimal stopping problems in terms of $(U, X)$ as per \cite{baurdoux2020lp}. 
\end{abstract}

\noindent
{\footnotesize Keywords: L{\'{e}}vy processes, last zero, positive excursions, It\^o formula, optimal stopping, corporate bankruptcy.}

\noindent
{\footnotesize Mathematics Subject Classification (2020): 60G40, 60J45, 60G51, 91G50}

\section{Introduction}

Last passage times have received considerable attention in the recent literature. For instance, in the classic ruin theory (which describes the capital of an insurance company), the moment of ruin is considered as the first time the process is below level zero. However, in more recent literature, the last passage time below zero is treated as the moment of ruin, and the Cram\'er--Lundberg process has been generalised to spectrally negative L\'evy processes (see e.g. \cite{chiu2005passage}). Moreover, in \cite{paroissin2013first}, spectrally positive L\'evy processes are used for degradation models, and the last passage time above a fixed boundary is considered the failure time.\\

Let $X=\{X_t,t\geq 0 \}$ be a spectrally negative L\'evy process. For any $t\geq 0$ and $x\in \R$, we define $g_t^{(x)}$ as the last time that the process is below $x$ before time $t$, i.e.,

\begin{align*}
g_t^{(x)}=\sup\{0\leq s\leq t: X_s \leq x \},
\end{align*}
with the convention $\sup \emptyset=0$. We simply denote $g_t:=g_t^{(0)}$ for all $t\geq 0$. 

 A similar version of this random time is studied in \cite{revuz2004continuous} (see Chapter XII.3), namely the last hitting time of zero, before any time $t\geq 0$, to describe excursions straddling a given time. It is also shown that this random time at time $t=1$ follows the arcsine distribution. The last-hitting time of zero plays an essential role in the study of Az\'ema's martingale (see \cite{AzemaYor1989}). In \cite{SALMINEN1988}, the distribution of the last hitting time of a moving boundary is found. \\
 
It is well known that spectrally negative L\'evy processes are often used to model the surplus of an insurance company (see e.g. \cite{huzak2004ruin}, \cite{huzak2004ruinb}, \cite{chan2004some}, \cite{kluppelberg2004ruin}, among many others). The random variable $g_t$ provides essential information regarding the insurance company's solvency. For instance, a large value of $U_t:=t-g_t$ (the time of the current positive excursion away from zero) indicates that the insurance company's capital has not fallen below zero for a considerable amount of time, suggesting that the company is currently able to meet debts and financial obligations.  \\
 
L\'evy processes are also widely used in financial modelling. For instance, there is considerable work in the literature that adopts markets driven by L\'evy processes (see e.g. \cite{schoutens2003levy}, \cite{tankov2003financial}, \cite{kyprianou2006exotic}, among many others). Assume that a stock price is given by $Y_t=\exp(X_t)$, then it is of interest for an investor to know when is the last time, before the time $t\geq 0$, that the stock price is below a certain level $y^*>0$. That is, the investor is interested in knowing the value of $g_t^{(\log(y^*))}$.\\

In \cite{leland1994corporate} and \cite{manso2010performance}, it is assumed that equity holders endogenously choose the time of the bankruptcy of a firm. They suppose that the performance measure of the firm can be modelled by a time-homogeneous diffusion $Y=\{Y_t,t\geq 0 \}$. Then, the time of the bankruptcy is determined by the optimal stopping problem
\begin{align*}
\sup_{\tau \in \mathcal{T}} \E\left(\int_0^{\tau }e^{-rt}[\delta(Y_t)-c(Y_t)]\dd t\bigg| Y_t=y \right),
\end{align*}
where $c(y)$ is the coupon rate that the firm must pay to the debt holders, and $\delta(y)$ is the payout rate received by the firm. The performance $Y$ measures the ability of the firm to serve its future debt obligations and can be taken to be financial ratios, stock prices or credit ratings. Note that given a certain level $k\geq 0$, the current positive excursion above the level $k$, given by $V_t^{(k)}=t-\sup\{0\leq s\leq t: Y_s\geq k \}$, also provides information about the performance of the firm. Indeed, large values of $V_t^{(k)}$ suggest that the firm has been able to meet its obligations for a long time without a negative dividend rate. Hence, the default time of the firm can be generalised to consider the process $(V,Y)=\{(V_t,Y_t), t\geq 0 \}$ as its performance measure, where $Y$ can be taken to be an exponential L\'evy process.  \\

On the other hand, when the pricing of American-type options is necessary to solve optimal stopping problems (see e.g. \cite{jacka1991optimal}, \cite{mordecki1999optimal}, \cite{mordecki2002optimal} and \cite{kyprianou2006exotic}), and it is known that they are intimately related to free-boundary problems (see e.g. Chapter III in \cite{peskir2006optimal}). Then, their solution often requires techniques that involve a Markovian approach and applications of It\^o formula. Hence, an explicit expression of the infinitesimal generator of the process is needed. Moreover, in more recent literature, the development of fluctuation identities allowed the use of the ``guess and verify'' approach to solving optimal stopping problems driven by L\'evy processes (see, for example, \cite{avram2004exit}, \cite{alili2005some} and \cite{kyprianou2005novikov}). Hence, given the importance of the random time $g_t$, it is relevant to be able to solve optimal stopping problems of the form,
\begin{align*}
\sup_{\tau \in \mathcal{T}} \E\left( e^{-r\tau}f(g_{\tau},\tau,X_{\tau})+\int_0^{\tau} e^{-r\tau } G(g_s,s,X_s)\dd s\right).
\end{align*}
Indeed, in \cite{baurdoux2020lp}, an optimal stopping of the form above arises when predicting $g_{\infty}$ with stopping times in an $L_p$ sense. In Section \ref{subsec:optimalstoppingproblems}, we also propose an optimal stopping problem that generalises the work of \cite{leland1994corporate} and \cite{manso2010performance} on corporate bankruptcy. Hence, it is relevant to derive path properties of the process $\{ (g_t,t,X_t), t\geq 0\}$.\\

The process $\{ g_t,t\geq 0\}$ is non-decreasing and hence is a process of finite variation, implying that it belongs to the class of semi-martingales. Then It\^o formula for the process $\{ (g_t,t,X_t), t\geq 0 \}$ is well known (see e.g., \cite{protter2005}, Theorem 7.33 on pp. 81-82 or \cite{cohen2015stochastic}, Theorem 14.2.4 on p. 345) and is given for any function $F:\R^3 \mapsto \R$ in $C^{2,2,2}$ by 
\begin{align*}
&F(g_{t}, t,X_t)\\
&=F(g_0,0,X_0)  +\int_{0}^{t} \frac{\partial F}{\partial \gamma} (g_{s-},s,X_{s-}) \dd g_s+ \int_{0}^{t } \frac{\partial F }{\partial t} (g_{s},s,X_{s}) \dd s\\
&\qquad\hspace{-0.5cm}+\int_{0}^{t} \frac{\partial F }{\partial x} (g_{s-},s,X_{s-})\dd X_s+\frac{1}{2} \sigma^2 \int_{0}^{t } \frac{\partial^2 F }{\partial x^2} (g_{s},s,X_{s})\dd s\\
& \qquad\hspace{-0.5cm} +\sum_{0<s\leq t} \left( F(g_{s},s,X_{s}) -F(g_{s-},s,X_{s-})-\frac{\partial }{\partial \gamma} F(g_{s-},s,X_{s-}) \Delta  g_s-\frac{\partial F}{\partial x} (g_{s-},s,X_{s-}) \Delta  X_s \right).
\end{align*}
%Here $E_g$ denotes the state space $\{ (\gamma,t,x)  : 0\leq \gamma < t \text{ and } x>0\}\cup \{ (\gamma,t,x) : 0 \leq \gamma=t \text{ and } x\leq 0\}$.
Note that using the local time-space formula given in \cite{peskir2007change} (see Theorem 3.2), we can assume that $F \in C^{1,1,i}$, where $i=2$ if $X$ is of infinite variation, and $i=1$ otherwise. Moreover, the formula above is given in terms of the jumps of the processes $\{ g_t,t\geq 0\}$ and $X$, and it does not reflect the dependence between $g_t$ and $X_t$. Indeed, some of the jumps of $\{ g_t,t\geq 0\}$ occur when $X$ jumps to $(-\infty,0)$ from the positive half line. Moreover, when a Brownian motion component is included in the dynamics of $X$, the stochastic process $\{g_t,t\geq 0 \}$ has infinitely many (small) jumps due to creeping. These facts imply that, to obtain a more explicit version of It\^o formula, a careful study of the trajectory of $t\mapsto g_t$ is required in terms of the excursions of $X$ away from zero. \\

On the other hand, we show that $\{ (g_t,t, X_t), t\geq 0 \}$ belongs to the family of strong Markov processes (see Proposition \ref{prop:Markovproperty}), and it is of interest to find its infinitesimal generator. It turns out that for Feller processes a general form of their infinitesimal generator is known. For instance, from \cite{revuz2004continuous} (see Theorem VII.1.13 on p. 289) we know that if $Z$ is a Feller process in $\R^d$, with $d$ a positive integer, and if $B\subset \R^d$ is any relative compact set, there exist functions $\sigma_{ij}$, $b_i$ and $c$ on $B$ and a kernel $\nu$ such that for any function $F\in C^{2}$ with compact support and $z\in B$, 
\begin{align*}
\mathcal{A}_Z F (z) 
&=c(z)F(z)+\sum_{i=1}^d  b_i(z) \frac{\partial }{\partial z_i} F(z) +\sum_{i,j=1}^d \sigma_{ij}(z)\frac{\partial^2}{\partial z_i \partial z_j}F(z)\\
& \qquad + \int_{\R^d \setminus \{0\}} \left( F(y) -F(z)- \sum_{i=1}^d (y_i-z_i) \frac{\partial }{\partial z_i} F(z) \right)\nu(z,\dd y).
\end{align*}
However, more explicit expressions for an It\^o formula and the infinitesimal generator are required in applications (for example, in optimal stopping and free boundary problems). In this work (see Theorem \ref{thm:ItoformulaforgttXt} and Corollary \ref{cor:infinitesimalgeneratorofgtX}), we give an expression for It\^o formula and the infinitesimal generator of the process $\{ (g_t,t,X_t), t\geq 0\}$ in terms of the dynamics of $X$ only.\\

We also consider, for any $t\geq 0$, the random variable $U_t=t-g_t$, the time of the current positive excursion away from zero. Then, having in mind the derivation of expressions for the potential measure of $(U,X)=\{(U_t,X_t),t\geq 0 \}$ and its  joint Laplace transform at an exponential time, we also derive an explicit formula, in terms of the positive and negative excursions of $X$, for functionals of the process $(U,X)$ of the form 
\begin{align*}
\E_{u,x}\left(\int_0^{\infty} e^{-qr} K(U_r,X_r)\dd r\right),
\end{align*}
for some function $K$ satisfying some conditions (see Theorem \ref{thm:integralofgr}), where $q\geq 0$ and $\P_{u,x}$ is the measure for which $(U_0,X_0)=(u,x)$ in view of the Markov property of $(U,X)$. The reader can find applications of these results in \cite{baurdoux2020lp}, which concerns the optimal prediction of the last zero of a spectrally negative L\'evy process and where the solution is given in terms of the process $(U,X)$. We also apply these results in Section \ref{sec:applicationtoOSandOPproblems} to solve a general optimal stopping problem.  \\

This paper is organised as follows. In Section \ref{sec:preliminaries}, we collect some fluctuation identities of spectrally negative L\'evy processes. Section \ref{sec:thelastzeroprocess} is dedicated to defining the last zero process, for which its basic properties are shown. Moreover, a derivation of It\^o formula, infinitesimal generator and formula for the expectation of a functional of $(U,X)$ are the main results of this section (see Theorems \ref{thm:ItoformulaforgttXt} and \ref{thm:integralofgr} and Corollary \ref{cor:infinitesimalgeneratorofgtX}). Then, the results mentioned above are applied to find formulas for the joint Laplace transform of $(U,X)$ at an exponential time, and a density of its $q$-potential measure is found. In Section \ref{sec:applicationtoOSandOPproblems}, we solve an optimal stopping problem (see Theorem \ref{thm:solutiontooptimalstopping}) driven by $(U,X)$. In particular, in Example \ref{ex:whentosellastock}, we propose an optimal stopping problem applied to corporate bankruptcy that depends on the trajectory of $(U,X)$. We also describe some optimal prediction problems of the last zero of the process. In this section, we emphasise the importance of the results developed in Section \ref{sec:thelastzeroprocess}. Lastly, in Section \ref{sec:proofs}, we include the main proofs of the paper. 
\section{Preliminaries}
\label{sec:preliminaries}

Let $(\Omega,\F, \mathbb{F}, \P)$ be a filtered probability space, where $\mathbb{F}=\{\F_t,t\geq 0 \}$ is a filtration which is naturally enlarged (see Definition 1.3.38 of \cite{bichteler2002stochastic}). A L\'evy process $X=\{X_t,t\geq 0 \}$ is an almost surely c\`adl\`ag process that has independent and stationary increments such that $\P(X_0=0)=1$. From the stationary and independent increments property, the law of $X$ is characterised by the distribution of $X_1$. We hence define the characteristic exponent of $X$, $\Psi(\theta):=-\log(\E(e^{i\theta X_1}))$. The L\'evy--Khintchine formula guarantees the existence of constants, $\mu \in \R$, $\sigma\geq 0$ and a measure $\Pi$ concentrated on $\R\setminus \{0\}$ with the property that $\int_{\R} (1\wedge x^2) \Pi(\dd x)<\infty$ (called the L\'evy measure) such that
\begin{align*}
\Psi(\theta)= i \mu\theta +\frac{1}{2}\sigma^2 \theta^2-\int_{\R} (e^{i \theta y}-1-i \theta y\I_{\{|y|<1 \}})\Pi(\dd y).
\end{align*}
Moreover, from the L\'evy--It\^o decomposition we can write 
\begin{align*}
%\label{eq:LevyItodecomposition}
X_t=\sigma B_t-\mu t+\int_{[0,t]}\int_{(-\infty,-1)\cup (1,\infty)} x N(\dd s\times \dd x)+\int_{[0,t]}\int_{(-1,1)} x( N(\dd s\times \dd x)-\dd s\Pi(\dd x)),
\end{align*}
where $N$ is a Poisson random measure on $\R^+\times \R$ with intensity $\dd t \times \Pi(\dd x)$ and $B=\{ B_t,t\geq 0 \}$ is an independent standard Brownian motion. We now state some properties and facts about L\'evy processes. The reader can refer, for example, to \cite{bertoin1998levy}, \cite{sato1999levy} and \cite{kyprianou2014fluctuations} for more details. Every L\'evy process $X$ is also a strong Markov $\mathbb{F}$-adapted process. Moreover, L\'evy processes satisfy a stronger property. Indeed, for any stopping time $\tau$ we define, on the event $\{ \tau<\infty\}$, the quantity $\widetilde{X}_t= X_{\tau+t}-X_t$, for each $t\geq 0$. Then, on the event $\{ \tau<\infty\}$, the process $\widetilde{X}=\{\widetilde{X}_t,t\geq 0 \}$ is independent of $\F_{\tau}$, has the same law as $X$ and is a L\'evy process. For all $x\in \R$, denote $\P_x$ as the law of $X$ when started at the point $x\in \R$, that is, $\E_x(\cdot)=\E(\cdot|X_0=x)$. Due to the spatial homogeneity of L\'evy processes, the law of $X$ under $\P_x$ is the same as that of $X+x$ under $\P$.\\

The process $X$ is a spectrally negative L\'evy process if it has no positive jumps ($\Pi(0,\infty)=0$) with no monotone paths. We state now some important properties and fluctuation identities of spectrally negative L\'evy processes, which will be useful in later sections, see \cite{bertoin1998levy}, Chapter VII or  Chapter 8 in \cite{kyprianou2014fluctuations} for details.\\

Due to the absence of positive jumps, we can define the Laplace transform of $X_1$. We denote $\psi(\beta)$ as the Laplace exponent of the process, that is, $\psi(\beta)=\log(\E(e^{\beta X_1}))$. Then for all $\beta \geq 0$,

\begin{align*}
\psi(\beta)=-\mu\beta +\frac{1}{2}\sigma^2 \beta^2+\int_{(-\infty,0)} (e^{\beta y}-1-\beta y\I_{\{y>-1 \}})\Pi(\dd y).
\end{align*}
It can be shown that $\psi$ is an infinitely differentiable and strictly convex function on $(0,\infty)$ that tends to infinity at infinity. In particular, $\psi'(0+)=\E(X_1)\in [-\infty,\infty)$ and determines the value of $X$ at infinity. When $\psi'(0+)>0$ the process $X$ drifts to infinity, i.e., $\lim_{t \rightarrow \infty} X_t=\infty$, when $\psi'(0+)<0$, $X$ drifts to minus infinity and the condition $\psi'(0+)=0$ implies that $X$ oscillates, that is, $\limsup_{t\rightarrow \infty} X_t =-\liminf_{t\rightarrow \infty} X_t=\infty$. We also define the right-inverse of $\psi$,

\begin{align*}
\Phi(q)=\sup\{\beta\geq 0: \psi(\beta)=q \}, \qquad q\geq 0.
\end{align*}
The process $X$ has paths of finite variation if and only if $\sigma=0$ and $\int_{(-1,0)} |x| \Pi(\dd x)<\infty$, otherwise $X$ has paths of infinite variation. In the latter case, we have that $X$ can be just written as a drift process minus a subordinator,

\begin{align}
\label{eq:Xhasfinitevariation}
X_t=d t+\int_{[0,t]}\int_{(-\infty,0)} x N(\dd s\times \dd x),
\end{align}
where \begin{align*}
d=-\mu -\int_{(-1,0)} x\Pi(\dd x).
\end{align*}
Since $X$ cannot have monotone paths, we necessarily have that $d>0$. Define $\tau_a^+$ as the first passage time above the level $a>0$,
\begin{align*}
\tau_a^+=\inf\{t>0: X_t>a \},
\end{align*}
where here and throughout the paper, we use the usual convention that $\inf\emptyset=\infty$. Then, for any $a>0$ and $q\geq 0$, the Laplace transform of $\tau_a^+$ is given by 
\begin{align}
\label{eq:laplacetransformtau0}
\E(e^{-q \tau_a^+}\I_{\{\tau_a^+<\infty \}})=e^{-\Phi(q)a}.
\end{align}
An essential family of functions for spectrally negative L\'evy processes are the scale functions, $W^{(q)}$. For all $q\geq 0$, the scale function $W^{(q)}:\R \mapsto \R_+$ is such that $W^{(q)}(x)=0$ for all $x<0$ and it is characterised on the interval $[0,\infty)$ as a strictly increasing and continuous function with Laplace transform given by
\begin{align}
\label{eq:laplacetransformofWq}
\int_0^{\infty} e^{-\beta x} W^{(q)}(x)\dd x=\frac{1}{\psi(\beta)-q}, \qquad \text{ for } \beta >\Phi(q).
\end{align} 
For the case $q=0$ we simply denote $W=W^{(0)}$. When $X$ has paths of infinite variation, $W^{(q)}$ is continuous on $\R$ and $W^{(q)}(0)=0$ for all $q\geq 0$, otherwise, we have $W^{(q)}(0)=1/d$, where $d>0$. The behaviour of $W^{(q)}$ at infinity is the following. For $q\geq 0$ we have, $\lim_{x\rightarrow \infty} e^{-\Phi(q)x}W^{(q)}(x)=\Phi'(q)$.\\

 There are some important fluctuation identities of L\'evy processes in terms of the scale functions. In particular, we list some that will be useful in later sections. Denote by $\tau_x^-$ as the first time $X$ is strictly below the level $x\leq 0$, i.e.,
\begin{align*}
\tau_x^-=\inf\{t>0: X_t<x \}.
\end{align*}
The Laplace transform of $\tau_a^+$, on the event of hitting the level $a>0$ before entering the set $(-\infty,0)$, is given by 
\begin{align}
\label{eq:laplacetransformoftaua+beforetau0-}
\E_x\left(e^{-q\tau_a^+} \I_{\{ \tau_a^+< \tau_0^-\}} \right)=\frac{ W^{(q)}(x)}{W^{(q)}(a)}
\end{align}
for any $x\leq a$. The joint Laplace transform of $\tau_0^-$ and $X_{\tau_0^-}$ is  
\begin{align}
\label{eq:jointlaplacetransformtau0-Xtau0-}
%\E_x(e^{-q \tau_0^-+\beta X_{\tau_0^-}} \I_{\{\tau_0^- <\infty \}})=e^{ \beta x }\left( 1+(q-\psi(\beta)) \int_0^x e^{- \beta y} W^{(q)}(y) \dd y - \frac{q-\psi(\beta )}{\Phi(q)-\beta} e^{-\beta x} W^{(q)}(x)\right)=e^{\beta x} \mathcal{I}^{(q,\beta)}(x)
\E_x(e^{-q \tau_0^-+\beta X_{\tau_0^-}} \I_{\{\tau_0^- <\infty \}})=e^{\beta x} \mathcal{I}^{(q,\beta)}(x)
\end{align}
for all $x>0$, $q\geq 0$ and $\beta\geq 0$, where the function $\mathcal{I}^{(q,\beta)}$ is given by
\begin{align}
\label{eq:functionI}
\mathcal{I}^{(q,\beta)}(x):=1+(q-\psi(\beta)) \int_0^x e^{-\beta y} W^{(q)}(y) \dd y-\frac{q-\psi(\beta)}{\Phi(q)-\beta} e^{-\beta x}W^{(q)}(x),	 \qquad x\in \R.
\end{align}
When $\beta=\Phi(q)$, for some $q\geq 0$, we understand the equation above in the limiting sense, i.e.,
\begin{align*}
\mathcal{I}^{(q,\Phi(q))}(x)=1-\psi'(\Phi(q)+) e^{-\Phi(q) x}W^{(q)}(x), \qquad x\in \R.
\end{align*}
Since $X$ has only negative jumps, we have that it only creeps upwards, that is,
\begin{align}
\label{eq:probabilityofcreepingupwards}
    \P(X_{\tau_x^+}=x |\tau_x^+<\infty )=1
\end{align}
for any $x>0$. Moreover, $X$ creeps downwards if and only if $\sigma>0$ and we have
\begin{align}
\label{eq:probabilityofcreepingdownwards}
    \P_x(X_{\tau_0^-}=0,\tau_0^-<\infty)=\frac{\sigma^2}{2}\left(W'(x)-\Phi(0)W(x) \right)
\end{align}
for any $x>0$. Denote by $\sigma_x^-$ the first time the process $X$ is below or equal to the level $x$, that is,
\begin{align}
\label{eq:definitionofsigmax-}
\sigma_x^-=\inf\{t>0: X_t \leq x \}.
\end{align}
For $t\geq 0$, let $\underline{X}_t=\inf_{0\leq s\leq t}X_s$ and let $\e_q$ be an exponential random variable (independent of $X$) with mean $1/q$, for $q\geq 0$. Since 
\begin{align*}
\E(e^{-q \sigma_{x}^-}\I_{\{\sigma_x^- <\infty\}})=\P(\e_q >\sigma_x^-)=\P(\underline{X}_{\e_q} \leq -x)
\end{align*}
for all $x \leq 0$, and the fact that the random variable $\underline{X}_{\e_q}$ is continuous on $(-\infty,0)$, we have that, for any $x>0$, the stopping times $\sigma_x^-$ and $\tau_x^-$ have the same distribution. When $X$ is of infinite variation, $X$ enters instantly to the set $(-\infty,0)$, whilst in the finite variation case, there is a positive time before the process enters it. That implies that in the infinite variation case, $\tau_0^-=\sigma_0^-=0$ almost surely. Note that in the finite variation case, since $0$ is irregular for $(-\infty,0]$ (see discussion in \cite{kyprianou2014fluctuations} on p. 157) and due to equation \eqref{eq:probabilityofcreepingdownwards}, we have that $\sigma_0^-=\tau_0^->0$ $\P$-a.s. \\
 
 Let $q>0$ and $a\in \R$. The $q$-potential measure of $X$ killed on exiting $(-\infty,a]$ is absolutely continuous with respect to Lebesgue measure with a density given by
\begin{align}
\label{eq:qpotentialdensitytkillingonexitinga}
 e^{-\Phi(q)(a-x)}W^{(q)}(a-y) -W^{(q)}(x-y), \qquad  x,y\leq a.
\end{align}
Similarly, the $q$-potential measure of $X$ killed on exiting $[0,\infty)$
\[\int_0^{\infty} e^{-qt} \P_x(X_t \in \dd y,t< \tau_0^-)\dd t\] is absolutely continuous with respect to Lebesgue measure, and it has a density given by
\begin{align}
\label{eq:qpotentialdensitytkillingonexiting0}
e^{-\Phi(q) y}W^{(q)}(x) -W^{(q)}(x-y) \qquad x,y\geq 0.
\end{align}
%For $\beta \geq 0$ the stochastic process $\{e^{\beta X_t-\psi(\beta) t}, t\geq 0 \}$ is a martingale. Then for each  $\beta\geq 0$, we can define a change of measure given by
%
%\begin{align}
%\label{eq:exponentialchangeofmeasure}
%\frac{\dd \P^{\beta}}{\dd \P}\bigg|_{\F_t}=e^{\beta X_t-\psi(\beta) t}.
%\end{align}
%Under the measure $\P^{\beta}$, $X$ is a L\'evy process with Laplace exponent given by $\psi_{\beta}(\lambda)=\psi(\lambda+\beta)-\psi(\beta)$ for $\lambda\geq -\beta$ and hence $\Phi_{\beta}(q):=\sup\{\lambda \geq -\beta: \psi_{\beta}(\lambda)=q \}=\Phi(q+\psi(\beta))-\beta$ for $q\geq -\psi(\beta)$. In the particular case when $\beta=\Phi(q)$ for $q\geq 0$ we have that $\psi_{\Phi(q)}(\lambda)=\psi(\lambda+\Phi(q))-q$. That implies that for any $q>0$, $\psi'_{\Phi(q)}(0+)=\psi'(\Phi(q))\geq 0$  and then the process $X$ drifts to infinity under the measure $\P^{\Phi(q)}$. Moreover, denote $W_{\Phi(q)}$ the $0$-scale function under the measure $\P^{\Phi(q)}$, we have that $W^{(q)}(x)=e^{\Phi(q)x} W_{\Phi(q)}(x)$ for all $x\in \R$ and $q>0$.
We have that the stochastic process 
\begin{align*}
\{ e^{-q(t\wedge \tau_0^-\wedge \tau_a^+)} W^{(q)}(X_{t\wedge \tau_0^-\wedge \tau_a^+}), t\geq 0\}
\end{align*}
is a martingale under $\P_x$, for any $a\in (0,\infty]$, $q\geq 0$ and $x\in \R$. \\

We close this section by stating a version of It\^o formula available in the literature applied to spectrally negative L\'evy processes that will be used in the sequel. Let $F:\R^2\mapsto \R$ and a set $D\subset \R^2$ with its closure denoted by $\bar{D}$. We say that $F$ is $C^{j_1,j_2}$ in $\bar{D}$, for $j_1,j_2 \in \{0,1,2,\ldots\}$, if $F$ restricted to $D$ coincides with a function $F_1:\R^2\mapsto \R$ which is $C^{j_1,j_2}$ on $\R_+\times \R$. Suppose that $X$ is any spectrally negative L\'evy process.
% and $b$ is a continuous function of bounded variation. We define the sets 
%\begin{align*}
%C&=\{(t,x) \in \R_+\times \R: x<b(t) \},\\
%D&=\{(t,x) \in \R_+\times \R: x>b(t) \}.
%\end{align*}
Let $i=1$ if $X$ is of finite variation and $i=2$, otherwise. From \cite{peskir2007change} (see Theorem 3.2 for the local time-space formula for general semimartingales) we deduce that if $F:\R^2 \mapsto \R$ is a continuous function such that, for some $b\in \R$, $F$ is $C^{1,i}$ on $[0,\infty)\times (-\infty,b]$ and on $[0,\infty)\times [b,\infty)$, the following change-of-variable formula holds:
\begin{align}
&F(t,X_t)\nonumber\\
&=F(0,X_0)+\int_0^{t} \frac{1}{2}\left( \frac{\partial F}{\partial t}(s,X_{s-}+)+\frac{\partial F}{\partial t}(s,X_{s-}-)  \right) \dd s \nonumber\\
&\qquad+\int_0^{t} \frac{1}{2}\left( \frac{\partial F}{\partial x}(s,X_{s-}+)+\frac{\partial F}{\partial x}(s,X_{s-}-)  \right) \dd  X_s+\frac{1}{2}\sigma^2 \int_0^{t} \frac{1}{2}\left( \frac{\partial^2 F}{\partial x^2}(s,X_{s-}+)+\frac{\partial^2 F}{\partial x^2}(s,X_{s-}-)  \right) \dd  s\nonumber\\
 &\qquad+ \int_{[0,t]}\int_{(-\infty,0)}\left[F(s,X_{s-}+y)-F(s,X_{s-})-\frac{1}{2}\left( \frac{\partial F}{\partial x}(s,X_{s-}+)+\frac{\partial F}{\partial x}(s,X_{s-}-)    \right)y \right]N(\dd s\times \dd y)\nonumber\\
 \label{eq:semimartingalelocaltimeItoformula}
 &\qquad+ \frac{1}{2} \int_0^t \left( \frac{\partial F}{\partial x}(s,X_{s-}+)-\frac{\partial F}{\partial x}(s,X_{s-}-) \I_{\{ X_{s-}=b, X_s=b\}}    \right) \dd L^b_s(X),
\end{align}
where $L_s^b(X)$ is the local time of $X$ at the level $b$ given in terms of the Tanaka formula, namely,
\begin{align}
|X_t-b|&=|X_0-b|+\int_0^t \sign(X_{s-}-b)\dd X_s+L_t^b(X) \nonumber\\
\label{eq:definitionofLocaltimebtanaka}
&\qquad +\int_{[0,t]}\int_{(-\infty,0)} \left[|X_{s-}+y-b|-|X_{s-}-b|-\sign(X_{s-}-b)y \right] N(\dd s \times \dd y).
\end{align}
%On the other hand, from \cite{kyprianou2007note} (see Theorem 7 for the general case) we know that if $X$ is of finite variation (see \eqref{eq:Xhasfinitevariation}) and $F:\R^2 \mapsto \R$  is $C^{1,1}$ on $[0,\infty)\times [b,\infty)$ and $F$ is $C^{1,1}$ on $[0,\infty)\times (-\infty,b]$, for some $b\in \R$, the following formula holds:
%\begin{align}
%F(t,X_t)&= F(0,X_0+)+\int_0^t \frac{\partial F}{\partial t}(s,X_{s-})\dd s+d\int_0^t \frac{\partial F}{\partial x} (s,X_{s-})\dd s\nonumber\\
%&\qquad+\int_{[0,t]} \int_{(-\infty,0)}[F(s,X_{s-}+y)-F(s,X_{s-}) ]N(\dd s\times \dd y)\nonumber\\
%\label{eq:Itoboundedvariation}
%&\qquad+\int_0^t[F(s,X_s)-F(s,X_{s-})]\dd n_s^b,
%\end{align}
%where $n_t^b=1+\sum_{0<s\leq t} \I_{\{X_s=b \}}$ is the number of times the process $X$ hits the level $b\in \R$. Note that in the spectrally negative case we have, from \eqref{eq:probabilityofcreepingupwards}-\eqref{eq:probabilityofcreepingdownwards} and the strong Markov property, that $X$ hits the level $b$ only by creeping upwards. Thus, the last term in the formula above vanishes if $F(s,b-)=F(s,b)$ for all $s\geq 0$. Having this in mind we have the following It\^o formula for the case when $X$ does not have a Gaussian component. 

\section{The last zero process}
\label{sec:thelastzeroprocess}
Let $X$ be a spectrally negative L\'evy process. Recall that $g_t^{(x)}$ is the last time that the process is below $x$ before time $t$, i.e.,
\begin{align*}
g_t^{(x)}=\sup\{0\leq s\leq t: X_s \leq x \},
\end{align*}
with the convention $\sup \emptyset=0$. We simply denote $g_t:=g_t^{(0)}$, for all $t\geq 0$. For any stopping time $\tau$, the random variable $g_{\tau}^{(x)}$ is $\F_{\tau}$ measurable. In particular we get that $\{g_t^{(x)}, t\geq 0 \}$ is adapted to the filtration $\{\F_t,t\geq 0 \}$. Moreover, it is easy to show that for a fixed $x\in \R$, the stochastic process $\{ g_t^{(x)},t\geq 0\}$ is non-decreasing, right-continuous with left limits. Similarly, for a fixed $t\geq 0$, the mapping $x\mapsto g_t^{(x)}$ is non-decreasing and almost surely right-continuous with left limits. \\

It can be easily seen that the process $\{g_t,t\geq 0\}$ is not a Markov process, particularly not a L\'evy process. However, the strong Markov property holds for the three-dimensional process $\{(g_t,t, X_t),t\geq 0 \}$. We delegate its proof to the Section \ref{sec:proofs}. 
\begin{prop}
\label{prop:Markovproperty}
The process $\{ (g_t,t,X_t),t\geq 0 \}$ is a strong Markov process with respect to the filtration $\{\F_t,t\geq 0 \}$ with state space given by $E_g=\{ (\gamma,t,x)  : 0\leq \gamma < t \text{ and } x>0\}\cup \{ (\gamma,t,x) : 0 \leq \gamma=t \text{ and } x\leq 0\}$. Moreover, we have that for any measurable positive function $h: E_g \mapsto \R$,  any $s\geq 0$ and any stopping time $\tau\geq 0$,
\begin{align*}
\E( h(g_{\tau+s},\tau+s,X_{\tau+s})|\F_{\tau})=f_s(g_{\tau},\tau,X_{\tau}),
\end{align*}
where for any $(\gamma,t,x)\in E_g$,
\begin{align}
\label{eq:measuretgammax}
f_s(\gamma,t ,x)&=\E_x(h(\gamma,t+s, X_s)\I_{\{ \sigma_0^->s\}})+\E_x(h(g_s+t,t+s, X_s)\I_{\{ \sigma_0^-\leq s\}}).
\end{align}
\end{prop}

In the spirit of the above Proposition, we define, for $(\gamma,t,x)\in E_g$, the probability measure $\P_{\gamma,t,x}$ in the following way: for every measurable and positive function $h$ we set
\begin{align*}
\E_{\gamma,t,x}(h(g_{t+s},t+s,X_{t+s})):=\E \left( h(g_{t+s},t+s,X_{t+s})\bigg|(g_t,t,X_t)=(\gamma,t,x) \right)=f_s(\gamma,t ,x),
\end{align*}
where $f_s$ is given in \eqref{eq:measuretgammax}. Then we can write $\P_{\gamma,t,x}$ in terms of $\P_x$ by
\begin{align}
\label{eq:initialmeausreP}
\E_{\gamma,t,x}(h(g_{t+s},t+s,X_{t+s}))=\E_x(h(\gamma,t+s, X_s)\I_{\{ \sigma_0^->s\}})+\E_x(h(g_s+t,t+s, X_s)\I_{\{ \sigma_0^-\leq s\}}).
\end{align}
Define $U_t=t-g_t$ as the length of the current excursion above the level zero. As a direct consequence of the form of $f_s$ given in \eqref{eq:measuretgammax}, we have that the process $\{(U_t,X_t),t\geq 0 \}$ is also a strong Markov process with state space given by $E=[(0,\infty)\times (0,\infty)]\cup \{ (0,x)\in \R^2 : x\leq 0\}$. Indeed, let $f:E\mapsto \R$ be any positive measurable function, taking $h(\gamma,t,x)=f(t-\gamma,x)$ in \eqref{eq:measuretgammax}, we obtain that 
\begin{align*}
f_{s}(\gamma,t,x)&=\E_x(h(\gamma,t+s, X_s)\I_{\{ \sigma_0^->s\}})+\E_x(h(g_s+t,t+s, X_s)\I_{\{ \sigma_0^-\leq s\}})\\
&=\E_x(f(t-\gamma+s, X_s)\I_{\{ \sigma_0^->s\}})+\E_x(f(U_s, X_s)\I_{\{ \sigma_0^-\leq s\}})\\
&=f_s(0,t-\gamma,x).
\end{align*}
From this, we deduce that
\begin{align*}
\E(f(U_{\tau+s},X_{\tau+s})|\F_{\tau})=f_s(g_{\tau},\tau,X_{\tau} )=f_s(0,U_{\tau},X_{\tau} )=\E(f(U_{\tau+s},X_{\tau+s})|\sigma(U_{\tau},X_{\tau})).
\end{align*}
We hence can define a probability measure $\P_{u,x}$, for all $(u,x)\in E$, by
\begin{align}
\label{eq:initialmeausrePforU}
\E_{u,x}(f(U_s,X_{s}))=\E_x(f(u+ s, X_s)\I_{\{ \sigma_0^->s\}})+\E_x(f(U_s, X_s)\I_{\{ \sigma_0^-\leq s\}}),
\end{align}
for any positive and measurable function $f$.

\begin{rem}
\label{rem:behaviourofgt}
We know that, for any $x\in \R$, the stochastic process $\{ g_t^{(x)}, t\geq 0 \}$ has non-decreasing paths. That directly implies that $g_t^{(x)}$ is a process of finite variation, and then it has a countable number of jumps. Moreover, by a close inspection to the definition of $g_t^{(x)}$, we notice that $g_t^{(x)}=t$ on the set $\{t\geq 0: X_t \leq x\}$, it is flat when $X$ is in the set $(x,\infty)$ and it has a jump when $X$ enters the set $(-\infty,x]$. Moreover, if $X$ is a process of infinite variation, we know that the set of times $X$ visits the level $x$ from above may be infinite with positive probability. That implies that when $X$ is of infinite variation, $t\mapsto g_t^{(x)}$ may have an infinite number of arbitrary small jumps with positive probability.
\end{rem}
%Note that the process $(g_t,t,X_t)$ is a semi-martingale so its It\^o formula is well known (see e.g. \cite{protter2005}, Theorem 33). 

In the following Theorem, we give a more explicit expression for the It\^o formula for the process $(g_t,t, X_t)$ in terms of the random measure $N$. Note that this formula will be helpful later in deriving the infinitesimal generator of $(g_t,t, X_t)$. The reader can find its proof in Section \ref{sub:proofofItoforumula}.
\begin{thm}[It\^o formula]
\label{thm:ItoformulaforgttXt}
Let $X$ be any spectrally negative L\'evy process, we define $i=1$ if $X$ is of finite variation and $i=2$, otherwise. Let $F:\bar{E}_g \mapsto \R$ be a continuous function that satisfies:

\begin{enumerate}
\item[i)] The mapping $(t,x)\mapsto F(t,t,x)$ is $C^{1,i}$ on $[0,\infty)\times(-\infty,0]$, for all $t\geq 0$;

\item[ii)] $F$ is $C^{0,1,i}$ on $\{(\gamma,t,x)\in \R_+^2\times \R: 0\leq \gamma\leq t \text{ and } x\geq 0  \}$.

\item[iii)] In the case that $\sigma>0$, $F$ is such that $\lim_{h\downarrow 0} F(\gamma,t,h)=F(t,t,0)$, for all $0\leq \gamma\leq t$, and 
\begin{align}
\label{eq:pastingatzeroItoformula}
\frac{\partial F}{\partial x}(t,t,0+)=\frac{\partial F }{\partial x}(t,t,0-)
\end{align}
for all $t\geq 0$. 
%\begin{align}
%\label{eq:pastingatzeroItoformula}
%(t,x)\mapsto \frac{\partial F}{\partial x}(t,t,x) \quad \text{is continuous in }	[0,\infty)\times \R.
%\end{align}
%
%
%We further assume that $(t,x)\mapsto \frac{\partial F}{\partial x}(t,t,x)$ is continuous in $[0,\infty)\times \R$.
\end{enumerate}
Then we have the following version of It\^o formula for the three dimensional process $\{(g_t,t,X_t),t\geq 0\}$.

\begin{align}
F(g_{t}, t,X_t)&=F(g_0,0,X_0)  +\int_{0}^{t}   \frac{\partial F_g }{\partial t } (s,X_{s-})\I_{\{g_{s-}=s \}}\dd s+ \int_{0}^{t} \frac{\partial F }{\partial t} (g_{s-},s,X_{s-})\I_{\{ g_{s-}<s\}} \dd s\nonumber\\
&\qquad+\int_{0}^{t} \frac{\partial F }{\partial x} (g_{s-},s,X_{s-})\dd X_s+\frac{1}{2} \sigma^2 \int_{0}^{t } \frac{\partial^2 F }{\partial x^2} (g_{s},s,X_{s})\dd s\nonumber\\
& \qquad +\int_{[0,t] } \int_{(-\infty,0)} \left[ F(g_{s},s,X_{s^-}+y) -F(g_{s-},s,X_{s-})-y\frac{\partial F}{\partial x} (g_{s-},s,X_{s-}) \right]N(\dd s\times \dd y)\nonumber\\
&=F(g_0,0,X_0)+ \int_{0}^{t}   \frac{\partial F_g }{\partial t } (s,X_{s-})\I_{\{g_{s-}=s \}}\dd s+ \int_{0}^{t} \frac{\partial F }{\partial t} (g_{s-},s,X_{s-})\I_{\{ g_{s-}<s\}} \dd s \nonumber\\
&\qquad+\int_{0}^{t}\frac{\partial F }{\partial x} (g_{s-},s,X_{s-})\dd X_s+\frac{1}{2} \sigma^2 \int_{0}^{t} \frac{\partial^2  F}{\partial x^2} (g_{s-},s,X_{s-}) \dd s\nonumber\\
& \qquad +\int_{[0,t] } \int_{(-\infty,0)} \left[ F(s,s,X_{s-}+y) -F(s,s,X_{s-})-y\frac{\partial F }{\partial x} (s,s,X_{s-}) \right]\I_{\{g_{s-}=s \}}N(\dd s\times \dd y)\nonumber\\
&\qquad+\int_{[0,t]}\int_{(-\infty,0)} \left[F(g_{s-},s,X_{s-}+y)-F(g_{s-},s,X_{s-})-y\frac{\partial F }{\partial x} (g_{s-},s,X_{s-})\right]\nonumber\\
&\qquad\qquad\qquad \times \I_{\{X_{s-}+y >0 \}}\I_{\{ g_{s-}<s\}} N(\dd s\times \dd y)\nonumber \\
&\qquad+ \int_{[0,t]}\int_{(-\infty,0)} \left[F(s,s,X_{s-}+y)-F(g_{s-},s,X_{s-})-y\frac{\partial F }{\partial x} (g_{s-},s,X_{s-}) \right]\nonumber\\
\label{eq:ItoformulaforgttXt}
&\qquad\qquad\qquad \times \I_{\{X_{s-}+y \leq 0 \}} \I_{\{ g_{s-}<s\}}N(\dd s\times \dd y),
\end{align}
where $ F_g(t,x):= F(t,t,x)$ for $t \geq 0$ and $x\leq 0$.

\end{thm}
\begin{rem}
\begin{itemize}
\item[i)]  When $\sigma>0$, the Brownian motion part of $X$ implies that $X$ can visit the interval $(-\infty,0]$ by creeping. That means that $t\mapsto g_t$ has two types of jumps: those as a consequence of $X$ jumping from the positive half line to $(-\infty,0)$ and those as a consequence of creeping. The limit condition imposed for $F$ and \eqref{eq:pastingatzeroItoformula} (when $\sigma>0$) ensure that the jumps due to the Brownian component vanish. Otherwise, a more careful analysis involving the local time needs to be done.
\item[ii)] Note that the limit condition imposed on $F$, when $\sigma>0$, comes naturally when we have functions that depend on expectations of the process $\{ (g_t,t,X_t), t\geq 0\}$. For example, if $f$ is a bounded continuous function, we have that
\begin{align*}
F(\gamma,t,x):=\E_{\gamma,t,x}(f(g_{t+s},t+s,X_{t+s}))=\E_x(f(\gamma,t+s, X_s)\I_{\{ \sigma_0^->s\}})+\E_x(f(g_s+t,t+s, X_s)\I_{\{ \sigma_0^-\leq s\}})
\end{align*}
satisfies that $\lim_{h\downarrow 0} F(\gamma,t,h)=F(t,t,0)$ for all $\gamma\leq t$ and $s>0$, when $\sigma>0$.

\item[iii)] Note that the proof relies on applying the appropriate version of It\^o formula to $F$ on the regions of $E_g$ where $x>0$ and $x<0$. So analogous results would be obtained if we relax the regularity conditions of $F$ and apply an appropriate version of It\^o formula (see e.g. Theorem IV.70 in \cite{protter2005}, Theorem 3.2 in \cite{peskir2007change}, Theorem 7 in \cite{kyprianou2007note}, etc.). For instance, it is possible to use Theorem 3.2 in \cite{peskir2007change} to obtain the same result as in Theorem \ref{thm:ItoformulaforgttXt} by replacing condition \eqref{eq:pastingatzeroItoformula} with the more restrictive condition: 
\begin{align*}
\frac{\partial F}{\partial x}(\gamma,t,0+)=\frac{\partial F }{\partial x}(t,t,0-) \qquad \text{for each }0\leq \gamma\leq t.
\end{align*}
It turns out that for some applications in optimal stopping condition \eqref{eq:pastingatzeroItoformula} is satisfied, whereas the condition above fails as it is found in Proposition \ref{prop:SolutiontocorporatebankrupcyOS} and Remark \ref{rem:PropexponentialOStimebankruptcy}. 

\end{itemize}
\end{rem}
Following the definition of \cite{revuz2004continuous} (see Definition VII.1.8 on p. 285), we say that $F$ belongs to the domain of the \emph{extended infinitesimal generator} of the process $Z_t=(g_t,t,X_t)$ if there exists a Borel function $f:E_g\mapsto \R$ such that $\int_0^{t}|f(X_s)| \dd s<\infty$, almost surely, for all $t>0$, and 
\begin{align*}
F(g_t,t,X_t)-F(g_0,0,X_0)-\int_0^{t} f(g_s,s,X_s)\dd s
\end{align*}
is a $\P_{\gamma,t,x}$-right-continuous martingale, for each $(\gamma,t,x)\in E_g$. If such a function exists, we write $\mathcal{A}_{Z}F(\gamma,t,x)=f(\gamma,t,x)$ and call it the \emph{extended infinitesimal generator}.

Now that we have a more explicit version of It\^o's formula for the three-dimensional process $(g_t,t, X_t)$ in terms of the Poisson random measure $N$, we are ready to state an explicit formula for its extended infinitesimal generator. The following Corollary follows directly from equation \eqref{eq:ItoformulaforgttXt} and standard arguments, so its proof is omitted (see, e.g., Proposition 2.4 in \cite{lamberton2008critical}, Proposition 8.16 in \cite{tankov2003financial} or the proof of Theorem 6.7.4 in \cite{applebaum_2009}).
\begin{cor}
\label{cor:infinitesimalgeneratorofgtX}
Suppose that $X$ and $F$ satisfy the conditions of Theorem \ref{thm:ItoformulaforgttXt} and further assume that $F$ and its derivatives are bounded functions. Then the infinitesimal generator $\mathcal{A}_Z$ of the process $Z_t=(g_t,t,X_t)$ is given by:

\begin{align}
\mathcal{A}_Z F (\gamma,t,x) &=\frac{\partial F_g }{\partial t} (t,x) \I_{\{x\leq 0 \}}+\frac{\partial F}{\partial	t} (\gamma,t,x)\I_{\{x> 0 \}}  -\mu \frac{\partial F }{\partial	x} (\gamma,t,x)+ \frac{1}{2} \sigma^2  \frac{\partial^2 F}{\partial x^2} (\gamma,t,x) \nonumber\\
& \qquad + \int_{(-\infty,0)} \left( F(\gamma,t,x+y) -F(\gamma,t,x)-y\I_{\{y>-1\}}\frac{\partial F}{\partial x} (\gamma,t,x) \right)\I_{\{x+y >0 \}}\Pi(\dd y)\nonumber\\
& \qquad + \int_{(-\infty,0)} \left( F(t,t,x+y) -F(t,t,x)-y\I_{\{y>-1\}}\frac{\partial F }{\partial x} (t,t,x) \right)\I_{\{x\leq 0 \}}\Pi(\dd y)\nonumber\\
& \qquad +  \int_{(-\infty,0)} \left( F(t,t,x+y)-F(\gamma,t,x)-y\I_{\{y>-1\}} \frac{\partial F }{\partial x} (\gamma,t,x) \right)\I_{\{x>0 \}}\I_{\{x+y <0 \}}\Pi(\dd y) \label{eq:infinitesimalgeneratorofgtX}
\end{align}
for all $(\gamma,t,x) \in E_g$.
\end{cor}

Recall from Remark \ref{rem:behaviourofgt} that the behaviour of $g_t$ (and then $U_t$) can be determined from the excursions of $X$ away from zero. The following theorem provides a formula to calculate an integral involving the process $\{(U_t, X_t), t\geq 0 \}$ with respect to time in terms of the excursions of $X$ above and below zero.

\begin{thm}
\label{thm:integralofgr} 
Let $q\geq 0$ and $X$ be a spectrally negative L\'evy process and $K:E\mapsto \R$ be a left-continuous function in each argument. Assume that there exists a non-negative function $C:\R_+ \times \R \mapsto \R$ such that $u\mapsto C(u,x)$ is a monotone function for all $x\in \R$, $|K(u,x)|\leq C(u,x)$ and $\E_{u,x}\left( \int_0^{\infty} e^{-qr } C(U_r,X_r+y) \dd r\right)<\infty$ for all $(u,x)\in E$ and $y\in \R$. Then we have, for any $(u,x)\in E$, that 
\begin{align}
\label{eq:calucationofintegraluptoinfinity2ndformula}
\E_{u,x}\left(\int_0^{\infty}e^{-qr} K(U_r,X_r) \dd r \right)&=
%&=K^+(u,x)+\lim_{\varepsilon \downarrow 0 }\E_x \left(  \I_{\{\tau_0^-<\infty \}} e^{-q\tau_0^-} K^-(X_{\tau_0^-}-\varepsilon) \right)\nonumber\\
%&\qquad +\mathcal{I}^{(q,\Phi(q))}(x) \lim_{\varepsilon \downarrow 0 } \frac{e^{-\Phi(q) \varepsilon}}{ 1-\mathcal{I}^{(q,\Phi(q))}(\varepsilon)} \left[ \E_{\varepsilon}\left(  \I_{\{ \tau_0^-<\infty\}} e^{-q \tau_0^-} K^-(X_{\tau_0^-}-\varepsilon) \right)+K^+(0,\varepsilon) \right]\\
%&=K^+(u,x)+\int_{-\infty}^{0}K(0,y)\left[e^{\Phi(q) x}\mathcal{I}^{(q,\Phi(q))}(x) W^{(q)}(-y) -W^{(q)}(x-y)+e^{-\Phi(q)y} W^{(q)}(x) \right] \dd y\\
%&\qquad +e^{\Phi(q)x}\mathcal{I}^{(q,\Phi(q))}(x)\int_{-\infty}^0 K(0,y)[\Phi'(q)e^{-\Phi(q) y}-W^{(q)}(-y)]\dd y ++e^{\Phi(q)x}\mathcal{I}^{(q,\Phi(q))}(x) \lim_{\varepsilon \downarrow 0 } \frac{e^{-\Phi(q) \varepsilon} K^+(0,\varepsilon)}{ 1-\mathcal{I}^{(q,\Phi(q))}(\varepsilon)} \\
%&=K^+(u,x)+\int_{-\infty}^{0}K(0,y)\left[e^{\Phi(q) x}\mathcal{I}^{(q,\Phi(q))}(x) \Phi'(q)e^{-\Phi(q) y} -W^{(q)}(x-y)+e^{-\Phi(q)y} W^{(q)}(x) \right] \dd y\\
%&\qquad +e^{\Phi(q)x}\mathcal{I}^{(q,\Phi(q))}(x) \lim_{\varepsilon \downarrow 0 } \frac{e^{-\Phi(q) \varepsilon} K^+(0,\varepsilon)}{ 1-\mathcal{I}^{(q,\Phi(q))}(\varepsilon)} \nonumber\\
K^+(u,x)+\int_{-\infty}^0 K(0,y) \left[\Phi'(q)e^{-\Phi(q)(y-x)} -W^{(q)}(x-y)\right] \dd y \nonumber\\
&\qquad+ [\Phi'(q)e^{\Phi(q)x}-W^{(q)}(x)]\lim_{\varepsilon \downarrow 0 } \frac{K^+(0,\varepsilon)}{ W^{(q)}(\varepsilon)},
\end{align}  
where $K^+$ is given by
\begin{align*}
K^+(u,x)=\displaystyle{\E_x\left( \int_0^{\tau_0^-}e^{-qr}K(u+r,X_r)\dd r \right)}, \qquad (u,x)\in E.
\end{align*}
In particular, when $u=x=0$ we have that 
\begin{align*}
\E\left(\int_0^{\infty}e^{-qr} K(U_r,X_r) \dd r \right)&=\int_{-\infty}^{0}K(0,y)\left[\Phi'(q)e^{-\Phi(q) y} -W^{(q)}(-y)\right] \dd y +\Phi'(q) \lim_{\varepsilon \downarrow 0 } \frac{ K^+(0,\varepsilon)}{ W^{(q)}(\varepsilon)}.
\end{align*}

\end{thm}
\begin{rem}
From the proof of Theorem \ref{thm:integralofgr}, we can find an alternative representation for formula \eqref{eq:calucationofintegraluptoinfinity2ndformula} as a limit in terms of excursions of $X$ above and below zero divided by a normalisation term. Indeed, for $(u,x)\in E$,
\begin{align*}
\E_{u,x}&\left(\int_0^{\infty}e^{-qr} K(U_r,X_r) \dd r \right)\\
&=K^+(u,x)+ \lim_{\varepsilon \downarrow 0 } \E_{x}\left(  \I_{\{ \tau_0^-<\infty\}} e^{-q \tau_0^-} K^-(X_{\tau_0^-}-\varepsilon) \right) \nonumber\\
&\qquad+ \left[\Phi'(q)e^{\Phi(q) x}- W^{(q)}(x) \right]  \lim_{\varepsilon \downarrow 0 }\frac{ \E_{\varepsilon}\left(  \I_{\{ \tau_0^-<\infty\}} e^{-q \tau_0^-} K^-(X_{\tau_0^-}-\varepsilon) \right)+K^+(0,\varepsilon) }{ W^{(q)}(\varepsilon)},
\end{align*}
where $K^-$ is given by
\begin{align}
\label{eq:definitionofK-}
K^-(x)=\displaystyle{\E_x\left( \int_0^{\tau_0^+}e^{-qr}K(0,X_r)\dd r \right)},
\end{align}
for all $x\in \R$.
\end{rem}

\subsection{Applications of Theorem \ref{thm:integralofgr}}
In this section, we consider applications of Theorem \ref{thm:integralofgr}. We first calculate the joint Laplace transform of $(U_{\e_q},X_{\e_q})$ where $\e_q$ is an exponential time with parameter $q>0$ independent of $X$.
\begin{cor}
Let $X$ be a spectrally negative L\'evy process. Let $q>0$ and $\alpha \in \R$, $\beta \geq 0$ such that $q>\psi(\beta) \vee (\psi(\beta)-\alpha)$. We have that for all $(u,x)\in E$,
\begin{align}
\label{eq:jointLaplaceUeqXeq}
\E_{u,x}&\left( e^{-\alpha U_{\e_q}+\beta X_{\e_q} } \right)\nonumber\\
%&=\frac{e^{-\alpha u}q W^{(q+\alpha)}(x)}{\Phi(q+\alpha)-\beta}-e^{-\alpha u}e^{\beta x} q\int_0^{x} e^{-\beta y}  W^{(q+\alpha)}(y) \dd y + \frac{qe^{\beta x}}{q-\psi(\beta)}+e^{\beta x}q \int_0^{x} e^{-\beta y}  W^{(q)}(y) \dd y \\
%&\qquad -\frac{q W^{(q)}(x)}{\Phi(q)-\beta} + \left[ e^{\Phi(q) x }\Phi'(q)-W^{(q)}(x)\right]\left[   \frac{q}{\Phi(q+\alpha)-\beta}-\frac{q}{\Phi(q)-\beta}\right]\\
&=\frac{qe^{\beta x}}{q-\psi(\beta)}+e^{\Phi(q) x }\Phi'(q)\left[   \frac{q}{\Phi(q+\alpha)-\beta}-\frac{q}{\Phi(q)-\beta}\right]+e^{\beta x}q \int_0^{x} e^{-\beta y}  [W^{(q)}(y)-e^{-\alpha u} W^{(q+\alpha)}(y) ]\dd y \nonumber\\
&\qquad + \frac{q}{\Phi(q+\alpha)-\beta} \left[e^{-\alpha u} W^{(q+\alpha)}(x)-W^{(q)}(x)\right].
\end{align}
\end{cor}
\begin{proof}
Consider the function $K(u,x)=e^{-\alpha u +\beta x}$ for all $(u,x)\in E$. We have that $K$ is a continuous function and $K(u,x)\leq e^{-(\alpha \wedge 0)  u+\beta x}$ for all $(u,x)\in E$. Take $q>0$ such that $q> \psi(\beta)\vee (\psi(\beta)-\alpha)=\psi(\beta)-(\alpha \wedge 0)$, we see that 

\begin{align*}
\E_x\left(\int_0^{\infty} e^{-qr} e^{-(\alpha \wedge 0) (u+r) +\beta X_r} \dd r \right)&= e^{\beta x-(\alpha \wedge 0)u}\int_0^{\infty}e^{-(q+(\alpha \wedge 0)-\psi(\beta))r} \dd r=\frac{e^{\beta x-(\alpha \wedge 0)u}}{ q+(\alpha \wedge 0)-\psi(\beta)}<\infty,
\end{align*}
for all $u\geq 0$ and $x\in \R$. Then for all $u> 0$ and $x>  0$ we have, by Fubini's theorem and from equation \eqref{eq:qpotentialdensitytkillingonexiting0}, that 
\begin{align*}
K^+(u,x)&=\E_x\left( \int_0^{\tau_0^-}e^{-qr}e^{-\alpha (u+r)+\beta X_r}\dd r \right)\\
%&=e^{-\alpha u} \E_x\left( \int_0^{\tau_0^-}e^{-(q+\alpha) r}e^{\beta X_r}\dd r \right) \\
%&=e^{-\alpha u} \int_0^{\infty} e^{-(q+\alpha) r}\E_x\left(e^{\beta X_r} \I_{\{r<\tau_0^- \}}\dd r \right) \\
&=e^{-\alpha u}  \int_{(0,\infty)}   e^{\beta y}\int_0^{\infty} e^{-(q+\alpha) r} \P_x(X_r\in \dd y ,r<\tau_0^- )\dd r  \\
&=e^{-\alpha u}  \int_{0}^{\infty}   e^{\beta y}\left[ e^{-\Phi(q+\alpha) y}W^{(q+\alpha)}(x)-W^{(q+\alpha)}(x-y) \right] \dd y \\
&=  \frac{e^{-\alpha u} W^{(q+\alpha)}(x)}{\Phi(q+\alpha)-\beta}-e^{-\alpha u}e^{\beta x} \int_0^{x} e^{-\beta y}  W^{(q+\alpha)}(y) \dd y.
%&=e^{-\alpha u+ \beta x} \frac{1-\mathcal{I}^{(q+\alpha,\beta)}(x)}{q+\alpha -\psi(\beta)}.
\end{align*}
Similarly, we calculate for any $x\in \R$,
\begin{align*}
\int_{-\infty}^{0}e^{\beta y}\left[e^{\Phi(q) (x-y)}\Phi'(q) -W^{(q)}(x-y)\right] \dd y
& = \Phi'(q)e^{\Phi(q)x} \int_0^{\infty}e^{-(\beta -\Phi(q))y }\dd y-e^{\beta x}\int_x^{\infty} e^{-\beta y}W^{(q)}(y)\dd y \\
& = \frac{\Phi'(q)e^{\Phi(q)x}}{\beta -\Phi(q)}-\frac{e^{\beta x}}{\psi(\beta)-q}+e^{\beta x}\int_0^x e^{-\beta y}W^{(q)}(y)\dd y,
\end{align*}
where the last equality follows from equation \eqref{eq:laplacetransformofWq} and the last integral is understood like $0$ when $x<0$. 
%
%
%
%Similary, from Fubini's theorem and equation \eqref{eq:qpotentialdensitytkillingonexitinga} we have that for all $x\leq 0$,
%\begin{align*}
%K^-(x)&=\E_x\left( \int_0^{\tau_0^+}e^{-qr}e^{\beta X_r}\dd r \right)\\
%%&=\int_0^{\infty}e^{-qr}  \E_x\left(e^{\beta X_r} \I_{\{r<\tau_0^+ \}} \right) \\
%&=\int_{(-\infty,0)}     e^{\beta y} \int_0^{\infty} e^{-qr} \P_x\left( X_r \in \dd y,r<\tau_0^+  \right) \dd r\\
%&=\int_{-\infty}^0     e^{\beta y} \left[ e^{\Phi(q)x} W^{(q)}(-y)-W^{(q)}(x-y) \right]\dd y \\
%&=\frac{e^{\Phi(q)x}-e^{\beta x}}{\psi(\beta)-q},
%\end{align*}
%where the last equality follows from \eqref{eq:laplacetransformofWq}. 
Then from \eqref{eq:calucationofintegraluptoinfinity2ndformula} we get that for all $(u,x)\in E$,

\begin{align*}
&\E_{u,x}\left(\int_0^{\infty}e^{-qr} e^{-\alpha U_{r}+\beta X_{r} }\dd r \right)\\
&=\frac{e^{-\alpha u} W^{(q+\alpha)}(x)}{\Phi(q+\alpha)-\beta}-e^{-\alpha u}e^{\beta x} \int_0^{x} e^{-\beta y}  W^{(q+\alpha)}(y) \dd y +\frac{\Phi'(q)e^{\Phi(q)x}}{\beta -\Phi(q)}-\frac{e^{\beta x}}{\psi(\beta)-q}+e^{\beta x}\int_0^x e^{-\beta y}W^{(q)}(y)\dd y \\
&\qquad +e^{\Phi(q)x}\mathcal{I}^{(q,\Phi(q))}(x)\lim_{\varepsilon \downarrow 0 } \frac{ 1}{\psi'(\Phi(q)+) W^{(q)}(\varepsilon)} \left[\frac{ W^{(q+\alpha)}(\varepsilon)}{\Phi(q+\alpha)-\beta}-e^{\beta \varepsilon} \int_0^{\varepsilon} e^{-\beta y}  W^{(q+\alpha)}(y) \dd y \right]\\
&=\frac{e^{-\alpha u} W^{(q+\alpha)}(x)}{\Phi(q+\alpha)-\beta}-e^{-\alpha u}e^{\beta x} \int_0^{x} e^{-\beta y}  W^{(q+\alpha)}(y) \dd y +\frac{\Phi'(q)e^{\Phi(q)x}}{\beta -\Phi(q)}-\frac{e^{\beta x}}{\psi(\beta)-q}+e^{\beta x}\int_0^x e^{-\beta y}W^{(q)}(y)\dd y \\
&\qquad +e^{\Phi(q)x} \left[ 1-\psi'(\Phi(q)+) e^{-\Phi(q) x}W^{(q)}(x) \right] \frac{\Phi'(q)}{\Phi(q+\alpha)-\beta},
\end{align*}
%
%
%
%\begin{align*}
%\E_{u,x}&\left(\int_0^{\infty}e^{-qr} e^{-\alpha U_{r}+\beta X_{r} }\dd r \right)\\
%&=e^{-\alpha u+ \beta x} \frac{1-\mathcal{I}^{(q+\alpha,\beta)}(x)}{q+\alpha -\psi(\beta)}+ \frac{e^{\Phi(q)x}\mathcal{I}^{(q,\Phi(q))}(x)-e^{\beta x}\mathcal{I}^{(q,\beta)}(x)}{\psi(\beta)-q}\\
%&\qquad+ e^{\Phi(q) x}\mathcal{I}^{(q,\Phi(q))}(x)  \lim_{\varepsilon \downarrow 0 } \frac{1}{ \psi'(\Phi(q)+) W^{(q)}(\varepsilon)}\left[\frac{ \mathcal{I}^{(q,\Phi(q))}(\varepsilon)     -\mathcal{I}^{(q,\beta)}(\varepsilon) }{\psi(\beta)-q} + e^{ \beta \varepsilon} \frac{1-\mathcal{I}^{(q+\alpha,\beta)}(\varepsilon)}{q+\alpha -\psi(\beta)} \right]\\
%&=e^{-\alpha u+ \beta x} \frac{1-\mathcal{I}^{(q+\alpha,\beta)}(x)}{q+\alpha -\psi(\beta)}- \frac{e^{\beta x}\mathcal{I}^{(q,\beta)}(x)}{\psi(\beta)-q}\\
%&\qquad+ e^{\Phi(q) x}\mathcal{I}^{(q,\Phi(q))}(x)  \lim_{\varepsilon \downarrow 0 } \frac{1}{ \psi'(\Phi(q)+) W^{(q)}(\varepsilon)}\left[\frac{ 1 -\mathcal{I}^{(q,\beta)}(\varepsilon) }{\psi(\beta)-q} + e^{ \beta \varepsilon} \frac{1-\mathcal{I}^{(q+\alpha,\beta)}(\varepsilon)}{q+\alpha -\psi(\beta)} \right]\\
%&=e^{-\alpha u+ \beta x} \frac{1-\mathcal{I}^{(q+\alpha,\beta)}(x)}{q+\alpha -\psi(\beta)}- \frac{e^{\beta x}\mathcal{I}^{(q,\beta)}(x)}{\psi(\beta)-q}+ e^{\Phi(q) x}\Phi'(q) \mathcal{I}^{(q,\Phi(q))}(x)  \left[   \frac{1}{\Phi(q+\alpha)-\beta}-\frac{1}{\Phi(q)-\beta}\right],
%\end{align*}
where in the last equality we used the fact that $\Phi'(q)=1/\psi'(\Phi(q)+)$, $W^{(q)}{(x)}$ is non-negative and strictly increasing on $[0,\infty)$, for all $q\geq 0$, and that
\begin{align*}
\lim_{\varepsilon \downarrow 0} \frac{W^{(q+\alpha)}(\varepsilon)}{W^{(q)}(\varepsilon)}=1.
\end{align*} 
The latter fact follows from the representation $W^{(q)}(x)=\sum_{k=0}^{\infty}q^k W^{*(k+1)}(x)$ and the estimate $W^{*(k+1)}(x)\leq x^k/k! W(x)^{k+1}$ (see equations (8.28) and (8.29) in \cite{kyprianou2014fluctuations}, pp 241-242). Rearranging the terms and using that
\begin{align*}
\E_{u,x}\left( e^{-\alpha U_{\e_q}+\beta X_{\e_q} } \right)=q\E_{u,x}\left(\int_0^{\infty}e^{-qr} e^{-\alpha U_{r}+\beta X_{r} }\dd r \right),
\end{align*}
for all $(u,x)\in E$, we obtain the desired result.
\end{proof}
\begin{rem}
\label{rem:LaplaceofXeq}
Note that from formula \eqref{eq:jointLaplaceUeqXeq}, we can recover some known expressions for spectrally negative L\'evy processes. If we take $\alpha=0$, we obtain for all $\beta \geq 0$, $q>\psi(\beta)\vee 0$ and $x\in \R$,
\begin{align*}
\E_x(e^{\beta X_{\e_q}})=\frac{q e^{\beta x}}{q-\psi(\beta)}.
\end{align*}
On the other hand, for any $\theta\geq 0$, $q\geq 0$ and $x\in \R$ we have that 
\begin{align*}
\E_x(e^{-\theta g_{\e_q}})=\int_0^{\infty}qe^{-qt} \E_x(e^{-\theta g_t})\dd t=\int_0^{\infty}qe^{-(q+\theta)t} \E_x(e^{\theta U_t})\dd t =\frac{q}{q+\theta} \E_x(e^{\theta U_{\e_{q+\theta}}}),
\end{align*}
where $\e_{q+\theta}$ is an exponential random variable with parameter $q+\theta$. The result coincides with the one found in \cite{baurdoux2009last} (see Theorem 2).
\end{rem}

Let $q>0$, we consider the $q$-potential measure of $(U,X)$ given by 
\begin{align*}
\int_0^{\infty} e^{-qr }\P_{u,x}(U_r \in \dd v, X_r \in \dd y) \dd r
\end{align*}
for $(u,x),(v,y)\in E$. From the fact $U_t=0$ if and only if $X_t\leq 0$, for any $t>0$, we have that for $(u,x)\in E$ and $y\leq 0$,
\begin{align*}
\int_0^{\infty} e^{-qr }\P_{u,x}(U_r =0, X_r \in \dd y) \dd r=\int_0^{\infty} e^{-qr }\P_{x}( X_r \in \dd y) \dd r.
\end{align*}
In the next corollary, we find an expression for a density when $v,x>0$.
\begin{cor}
\label{cor:qpotentialmeasureUX}
Let $X$ be a spectrally negative L\'evy process and $q>0$. The $q$-potential measure of $(U,X)$ has a density given by
\begin{align}
\label{eq:qpotentialmeasureUX}
\int_0^{\infty} e^{-qr }\P_{u,x}(U_r \in \dd v, X_r \in \dd y) \dd r
&=e^{-q(v-u)}\P_x(X_{v-u} \in \dd y, v-u<\tau_0^-)  \I_{\{v> u \}} \dd v\nonumber\\
&\qquad +\left[e^{\Phi(q) x}\Phi'(q)-W^{(q)}(x) \right]  \frac{y}{v} e^{-qv}\P(X_{v} \in \dd y) \dd v 
\end{align}
for all $(u,x)\in E$ and $v,y>0$. In particular, when $u=x=0$ we have that 
\begin{align*}
\int_0^{\infty} e^{-qr }\P(U_r \in \dd v, X_r \in \dd y) \dd r
&=\Phi'(q)   \frac{y}{v} e^{-qv}\P(X_{v} \in \dd y) \dd v.
\end{align*}
\end{cor}

\begin{proof}
Let $0<u_1<u_2$ and $0<x_1<x_2$ and define the sets $A=(u_1,u_2]$ and $Y=(x_1,x_2]$. Then the function $K(u,x)=\I_{\{ u\in A, x\in Y \}}$ is left-continuous and  bounded from above by $C(x)=\I_{\{x\in Y\}}$. Moreover, we have that for  $q>0$ and $x\in \R$,

\begin{align*}
\E_x\left(\int_0^{\infty}e^{-qr} \I_{\{X_r \in Y \}} \dd r \right)<\infty.
\end{align*}
First, we calculate for all $u,x>0$ such that $u<u_1$, 
\begin{align*}
K^+(u,x)=\E_{u,x}\left(\int_0^{\tau_0^-} e^{-qr}\I_{\{ U_r \in A, X_r  \in Y\}} \dd r \right)=\int_{A} \int_{Y} e^{-q(r-u)}\P_x(X_{r-u} \in  \dd y, r-u<\tau_0^-) \dd r.
\end{align*}
For every $x\leq 0$ we have that 
\begin{align*}
K^-(x)=	\E_{u,x}\left(\int_0^{\tau_0^+} e^{-qr}\I_{\{ U_r \in A, X_r \in Y\}} \dd r \right)=0.
\end{align*}
Hence, for all $(u,x)\in E$ we obtain that
\begin{align*}
\E_{u,x}&\left(\int_0^{\infty} e^{-qr} \I_{\{U_r \in A, X_r \in Y \}} \dd r \right)\\
&=
\int_{A} \int_Y  e^{-q(r-u)}\P_x(X_{r-u} \in \dd y, r-u<\tau_0^-) \dd r\\
&\qquad + e^{\Phi(q) x}\left[1-\psi'(\Phi(q)+) e^{-\Phi(q) x}W^{(q)}(x) \right] \lim_{\varepsilon \downarrow 0 }  \int_{A} \int_Y  \frac{e^{-qr}\P_{\varepsilon}(X_{r} \in \dd y, r<\tau_0^-)  }{ \psi'(\Phi(q)+) W^{(q)}(\varepsilon)} \dd r.
\end{align*}
We calculate the limit on the right-hand side of the equation above. Denote $\P_{\varepsilon}^{\uparrow}$ as the law of $X$ starting from $\varepsilon$ conditioned to stay positive. We have, for all $x\in \R$ and $y>0$, that 
\begin{align*}
\lim_{\varepsilon \downarrow 0 }  \int_{A} \int_Y  \frac{e^{-qr}\P_{\varepsilon}(X_{r} \in \dd y, r<\tau_0^-)  }{ \psi'(\Phi(q)+) W^{(q)}(\varepsilon)} \dd r
&=\lim_{\varepsilon \downarrow 0 } \frac{W(\varepsilon)}{ \psi'(\Phi(q)+)  W^{(q)}(\varepsilon)}   \int_{A} \int_Y  \frac{e^{-qr}\P^{\uparrow}_{\varepsilon}(X_{r} \in \dd y)  }{  W(y)} \dd r \\
&=\frac{1}{ \psi'(\Phi(q)+) }   \int_{A} \int_Y  \frac{e^{-qr}\P^{\uparrow}(X_{r} \in \dd y)  }{  W(y)} \dd r,
\end{align*} 
where the first equality follows from the definition of $\P^{\uparrow}$ (see e.g. \cite{bertoin1998levy} section VII.3 equation (6)) and the last equality follows since $\lim_{\varepsilon \downarrow 0} W(\varepsilon)/W^{(q)}(\varepsilon)=1$ and $\P_{\varepsilon}^{\uparrow}$ converges to $\P^{\uparrow}$ in the sense of finite-dimensional distributions (see Proposition VII.3.14 in \cite{bertoin1998levy}). Moreover, we have for all $y,r>0$ that $\P^{\uparrow}(X_r\in \dd y)= yW(y)\P(X_r \in \dd y)/r$ (see Corollary VII.3.16 in \cite{bertoin1998levy}).
Therefore, we obtain for all $(u,x)\in E$ that
\begin{align*}
\E_{u,x}\left(\int_0^{\infty} e^{-qr} \I_{\{U_r \in A, X_r \in Y \}} \dd r \right)&=
\int_{A} \int_Y e^{-q(r-u)}\P_x(X_{r-u} \in \dd y, r-u<\tau_0^-) \dd r\\
&\qquad + \left[\Phi'(q)  e^{\Phi(q) x}-W^{(q)}(x) \right]  \int_{A} \int_Y   \frac{y}{r}e^{-qr} \P(X_{r} \in \dd y) \dd r,
\end{align*}
where we also used the fact that $\Phi'(q)=1/\psi'(\Phi(q)+)$. The proof is now complete.
% Using the same arguments one can easily see that when $X$ is of finite variation
%\begin{align*}
%\int_{A} \int_Y e^{-q(r-u)}\P(X_{r-u} \in \dd y, r-u<\tau_0^-) \dd r= W^{(q)}(0)  \int_{A} \int_Y   \frac{y}{r}e^{-qr} \P(X_{r} \in \dd y) \dd r.
%\end{align*}

\end{proof}
\begin{rem}
\cite{Binghamfluctuationcontinuous} showed that the $q$-potential measure of $X$ has a density that is absolutely continuous with respect to the Lebesgue measure. This can be demonstrated by moving the killing barrier on the $q$-potential measure killed on entering the set $(-\infty,0]$ (see \eqref{eq:qpotentialdensitytkillingonexiting0}) and taking limits. Alternatively, it can be deduced by taking limits on \eqref{eq:calucationofintegraluptoinfinity2ndformula}. Moreover, Corollary \ref{cor:qpotentialmeasureUX} provides an alternative method for finding the density mentioned above. For this, we use Kendall's identity (see e.g. \cite{bertoin1998levy}, Corollary VII.3) given by
\begin{align}
\label{eq:Kendallsidentity}
r\P(\tau_z^+ \in \dd r) \dd z=z\P(X_r \in \dd z) \dd r
\end{align}
 for all $r,z \geq 0$. Indeed, let $u,y>0$ and $x\in \R$, integrating \eqref{eq:qpotentialmeasureUX} with respect to the variable $v$, we obtain that
\begin{align*}
\int_0^{\infty} e^{-qr} &\P_x(X_r \in \dd y) \dd r\\
&=\int_{(0,\infty)} \int_0^{\infty} e^{-qr }\P_{u,x}(U_r \in \dd v, X_r \in \dd y) \dd r \\
&=\int_0^{\infty} e^{-qv}\P_x(X_{v} \in \dd y, v<\tau_0^-) \dd v   + \int_0^{\infty}\left[e^{\Phi(q) x}\Phi'(q)-W^{(q)}(x) \right]  \frac{y}{v} e^{-qv}\P(X_{v} \in \dd y) \dd v \\
&=[e^{-\Phi(q) y}W^{(q)}(x) -W^{(q)}(x-y)] \dd y +  \left[e^{\Phi(q) x}\Phi'(q)-W^{(q)}(x) \right]  \int_0^{\infty} e^{-q v} \P(\tau_y^+ \in \dd v)\dd y ,
\end{align*}
where the last equality follows from \eqref{eq:qpotentialdensitytkillingonexiting0} and \eqref{eq:Kendallsidentity}. Hence, using the formula for the Laplace transform of $\tau_y^+$ (see equation \eqref{eq:laplacetransformtau0}) we have that 
\begin{align*}
\int_0^{\infty} e^{-qr} \P_x(X_r \in \dd y) \dd r=\left( e^{\Phi(q) (x-y)}\Phi'(q) -W^{(q)}(x-y)\right) \dd y.
\end{align*}
%
%
%On the other hand, from the strong Markov property applied applied to time $\tau_0^+$ and equation \eqref{eq:laplacetransformtau0lzp} we obtain that for any $x<0$, 
%\begin{align*}
%\int_0^{\infty} e^{-qr} \P_x(X_r \in \dd y) \dd r=e^{\Phi(q)x} f_y(q) \dd y,
%\end{align*}
%where for all $q>0$, $f_y(q) \dd y:=\int_0^{\infty}e^{-qr} \P(X_r\in \dd y)$. Then from the two equations above we obtain the integral equation
%\begin{align*}
%f_y(q)=\Phi'(q) y \int_q^{\infty} f_y(p) \dd p.
%\end{align*}
%Solving the equation we obtain that $f_y(q)=\Phi'(q)e^{-\Phi(q)y} C$ for some constant $C>0$. Therefore for all $x\in \R$,
%\begin{align*}
%\int_0^{\infty} e^{-qr} \P_x(X_r \in \dd y) \dd r=\left( e^{\Phi(q) (x-y)}\Phi'(q)C -W^{(q)}(x-y)+e^{-\Phi(q) y}W^{(q)}(x)[C-1] \right) \dd y.
%\end{align*}
%Using the space homogeneity of L\'evy processes we have that for all $y>0$, 
%\begin{align*}
%\lim_{x\uparrow y}\int_0^{\infty} e^{-qr} \P_x(X_r \in \dd y) \dd r= \lim_{y\downarrow 0}\int_0^{\infty} e^{-qr} \P(X_r \in \dd y) \dd r=\lim_{y\downarrow 0} f_y(q) \dd y.
%\end{align*}
%Comparing the latter with the equation above we conclude that $C=1$. Therefore for all $x\in \R$ and $y>0$,
%\begin{align*}
%\int_0^{\infty} e^{-qr} \P_x(X_r \in \dd y) \dd r=\left( e^{\Phi(q) (x-y)}\Phi'(q) -W^{(q)}(x-y)\right) \dd y.
%\end{align*}
%Moreover, from the space homogeneity of L\'evy processes we conclude that the equation above holds for all $x\in \R$ and $y\in \R$.
\end{rem}

\section{Applications to optimal stopping/prediction problems}
\label{sec:applicationtoOSandOPproblems}
\subsection{Optimal stopping problems}
\label{subsec:optimalstoppingproblems}
This section uses the results developed in the previous sections to solve a general optimal stopping problem. For the sake of simplicity, we will assume that $X$ is a spectrally negative process with a Gaussian component. That is, we assume that $\sigma>0$. We take $r\geq 0$ and let $G$ be a continuous function on $E$ such that\begin{align}
\label{eq:OSintegrabilitycontidion}
\E_{u,x}\left(\int_0^{\infty} e^{-rs}|G(U_s,X_s)|\dd s  \right)<\infty
\end{align}
for all $(u,x)\in E$. We further assume that there exists a value $x_G<0$ such that: $G(0,x)<0$ for all $x<x_G$ with $\lim_{x\rightarrow -\infty}G(0,x)<0$, and $G(u,x)\geq 0$ for all $(u,x)\in E$ such that $x\geq x_G$. We also assume that the function 
\begin{align*}
K^+(u,x):=\E_x\left(\int_0^{\tau_0^-} e^{-rs}G(u+s,X_s)\dd s \right)
\end{align*}
is $C^{1,2}$ on $[0,\infty)\times[0,\infty)$.\\

We consider the following optimal stopping problem 
\begin{align}
\label{eq:optimalstoppingproblem}
V(u,x)=\sup_{\tau \in \mathcal{T}} \E_{x,u}\left(\int_0^{\tau}e^{-rs}G(U_s,X_s)\dd s \right),
\end{align}
where $\mathcal{T}$ is the set of all stopping times of $X$. Note that our assumptions suggest that it is never optimal to stop when $X$ is taking positive values, and since $G$ is negative for $x<x_G$, we should stop as soon as $X$ is below a value $z^*<x_G$, for $|z^*|$ sufficiently large. The following theorem confirms that notion.
\begin{thm}
\label{thm:solutiontooptimalstopping}
Under the conditions stated above, we have that an optimal stopping time for \eqref{eq:optimalstoppingproblem} is given by 
\begin{align*}
\tau_{z^*}^-=\inf\{t>0: X_t\leq z^* \},
\end{align*}
where $z^* \in (-\infty,0)$ is characterised as the unique solution for $z$ in $(-\infty,0)$ to the equation 
\begin{align*}
\int_{(0,\infty)}\int_0^{\infty}  G(v,y)  \frac{y}{v}e^{-rv}\P(X_{v} \in \dd y) \dd v+\int_{z}^0 G(0,y) e^{-\Phi(r) y}\dd y=0.
\end{align*}
We have that $z^*<x_G$, where we recall that $x_G=\inf\{x\in \R: G(0,x) \geq 0 \}<0$. Moreover, the value function is given by 
\begin{align*}
V(u,x)= K^+(u,x)  -W^{(r)}(x)\int_{(0,\infty)}\int_0^{\infty}  G(v,y)  \frac{y}{v}e^{-rv}\P(X_{v} \in \dd y) \dd v  
-\int_{z^*}^0 G(0,y)W^{(r)}(x-y) \dd y,
\end{align*}
for all $(u,x)\in E$. Furthermore, there is smooth fit at $z^*$, that is, $\frac{\partial}{\partial x} V(0,z^*+)=\frac{\partial}{\partial x} V(0,z^*-)$.
\end{thm}
We have the following remark regarding some of the assumptions in Theorem \ref{thm:solutiontooptimalstopping}.
\begin{rem}
\begin{enumerate}
\item[i)] Condition \eqref{eq:OSintegrabilitycontidion} ensures that the optimal stopping problem is well posed and can be relaxed without affecting our results. If the function $K$ is a bounded function, we see that \eqref{eq:OSintegrabilitycontidion} is satisfied whenever $r>0$. Moreover, if there exist values $\alpha>0$ and $\beta\geq 0$ such that $|G(u,x)|\leq \alpha e^{\beta x}$, then \eqref{eq:OSintegrabilitycontidion} holds whenever $r>\psi(\beta)$ (see Remark \ref{rem:LaplaceofXeq}).
\item[ii)] The condition imposed on $G$ concerning $x_G$ ensures that an optimal solution is given in terms of the constant barrier $z^*\in (-\infty,0)$. For instance, if $G(u,x)\geq 0$ for all $(u,x)\in E$, the stopping time $\tau\equiv \infty$ is always optimal. On the other hand, if we allow that $G(u,x)<0$, for some values $u>0$ and $x>0$, the optimal solution may be of the form $\tau_D=\inf \{t\geq 0: (U_t, X_t)\in D \}$, where $D\subset E$ and $D\cap [(0,\infty)\times (0,\infty)] \neq  \emptyset$, and a more careful analysis needs to be done (see, e.g., Section \ref{sec:optimalpredictionproblems}).
\end{enumerate}
\end{rem}

The proof of Theorem \ref{thm:solutiontooptimalstopping} relies on finding, with the help of the potential measure of $(U,X)$ given in Corollary \ref{cor:qpotentialmeasureUX}, a semi-explicit expression of the function $V_z(u,x)=\E_{u,x}\left(\int_0^{\tau_z^-} G(U_s,X_s)\dd s \right)$, for each $z\leq 0$. Then, due to the properties of the scale functions $W^{(r)}$ and applying the version of It\^o formula derived in Theorem \ref{thm:ItoformulaforgttXt}, we see that the two conditions given in Lemma \ref{lemma:verificationlemma} are satisfied. The reader should also note that formula \eqref{eq:calucationofintegraluptoinfinity2ndformula} helped prove that condition \eqref{eq:pastingatzeroItoformula} is satisfied for $V$. The proof is deferred to Section \ref{subsec:solutiontooptimalstopping}.\\

Motivated by the example below, we have the following proposition as an application of Theorem \ref{thm:solutiontooptimalstopping}. Its proof is relegated to the Section \ref{subsection:Proofofsolutionbankrupcy}. 
\begin{prop}
\label{prop:SolutiontocorporatebankrupcyOS}
Let $X$ be any spectrally negative L\'evy process with $\sigma>0$. Take $\beta \geq 0$, $K\in (0,1)$ and $r>\psi(1)+\beta$. We consider the optimal stopping problem 
\begin{align*}
V(u,x)
&=\sup_{\tau \in \mathcal{T}}\E_{u,x}\left( \int_0^{\tau } e^{-rs}[e^{X_t +\beta U_s}-K]\dd s  \right),\qquad (u,x)\in E. 
\end{align*} 
Then, the stopping time 
\begin{align*}
\tau_{z^*}^-=\inf\{t>0: X_t\leq z^* \}
\end{align*}
is an optimal stopping time, where $z^*$ is the unique solution on $(-\infty,0)$ to the equation 
\begin{align}
\label{eq:defaultlevel}
\frac{e^{-(\Phi(r)-1)z}}{\Phi(r)-1}-\frac{Ke^{-\Phi(r)z}}{\Phi(r)}+\frac{1}{\Phi(r-\beta)-1}-\frac{1}{\Phi(r)-1}=0.
\end{align}
Moreover, the value function takes the form:
\begin{align*}
V(u,x)&= e^{\beta u} \left[ \frac{W^{(r-\beta)}(x)}{\Phi(r-\beta)-1} -\int_0^x e^y W^{(r-\beta)}(x-y)\dd y \right] +K\int_0^xW^{(r)}(y)\dd y-\frac{K}{\Phi(r)}W^{(r)}(x)\\
&\qquad-W^{(r)}(x)\left[\frac{1}{\Phi(r-\beta)-1}-\frac{K}{\Phi(r)} \right]-\int_{z^*}^0 [e^{y}-K]W^{(r)}(x-y) \dd y
\end{align*}
for every $(u,x)\in E$.

\end{prop}
We have the following remark on the above proposition regarding some particular cases and how our results match the current literature.

\begin{rem}
\label{rem:PropexponentialOStimebankruptcy}
\begin{itemize}
\item[i)] Here, it can be checked that the value function satisfies the condition $\frac{\partial V}{\partial x }(u,0+)=\frac{\partial V}{\partial x} V(u,0-)$ is satisfied only for the case $u=0$.
\item[ii)]
Note that when $\beta=0$,  we see from \eqref{eq:defaultlevel} that the value $z^*$ takes the form 
\begin{align*}
z^*=\log\left( \frac{\Phi(r)-1}{\Phi(r)}K\right).
\end{align*}
Moreover, when $Y$ is a geometric Brownian motion with mean $m$ and volatility $\sigma>0$ (that is, $\mu=m-\sigma^2/2$ and $\Pi\equiv 0$), we recover the value of $z^*$, when $r>m$, found in \cite{leland1994corporate} (see also Section III.C in \cite{quah2013discounting}). Indeed, in this case we have that for any $q\geq 0$,
\begin{align*}
\Phi(q)=\frac{1}{\sigma^2 }\left( \sqrt{\mu^2+2q\sigma^2}-\mu \right).
\end{align*}
An easy calculation shows that 
\begin{align*}
z^*=\log\left( \frac{\xi(r)}{\xi(r)+1}\left(1-\frac{m}{r}\right)K\right),
\end{align*}
where for any $q\geq 0$, 
\begin{align*}
\xi(q)=\frac{1}{\sigma^2 }\left( \sqrt{\mu^2+2q\sigma^2}+\mu \right).
\end{align*}
 
\end{itemize}

\end{rem}
We then present a setting where the result above becomes relevant.

\begin{example}
\label{ex:whentosellastock}
Following the model of corporate bankruptcy in \cite{leland1994corporate} and \cite{manso2010performance} (see also Section III.C in \cite{quah2013discounting}), we consider that equity holders endogenously choose the bankruptcy time. Suppose that the performance of a firm\footnote{This could be any statistic measuring the firm's ability to pay its debt obligations in the future. For example, prices of stocks, financial ratios, or credit ratings.} at time $t\geq 0$, is given by $Y_t=\exp(X_t)$, where $X$ is a spectrally negative L\'evy process such that $\sigma>0$. The performance measure $Y_t$ is normalised such that the values above the level $1$ are considered a good company performance, whereas values below one indicate a negative rating. Then, we consider for $t\geq 0$, $V_t=t-\sup\{0\leq s\leq t: Y_s\leq 1 \}=t-U_t$, the length of time since the last time the company performed poorly. Large values of $V_t$ can be interpreted as the firm's financial stability.

Suppose that, until bankruptcy, the firm must pay a coupon rate $c(v,y)$ to debt holders and receive a payout rate $\delta(v,y)$ in terms of the performance $y$ and $v$, the current excursion above the level $1$. Then, the time of bankruptcy is determined by the optimal stopping problem  
\begin{align}
\label{eq:optimalstoppingproblemcorparebankrupcy}
V(u,x)
%&=\sup_{\tau \in \mathcal{T}}\E_{u,x}\left( \int_0^{\tau } e^{-rs}[\delta(V_s,Y_s)-c(V_s,Y_s)]\dd s  \right)\\
&=\sup_{\tau \in \mathcal{T}}\E_{u,x}\left( \int_0^{\tau } e^{-rs}[\delta(U_s,e^{X_t})-c(U_s,e^{X_t})]\dd s  \right),
\end{align}
where $r\geq 0$ is the risk-free interest rate. Note that if $\delta(U_s,e^{X_s})$ is lower than $c(U_s,e^{X_s})$, equity holders have a negative dividend rate. Then, the firm will keep operating with a negative dividend rate if the firm's prospects are good enough to compensate for the negative losses. Otherwise, the firm will stop operations, and bankruptcy will be declared.

%We assume that the price of a stock at time $t\geq 0$ is given by $Y_t=\exp(X_t)$, where $X$ is a spectrally negative L\'evy process such that $\sigma>0$. We suppose that the stock pays dividends (proportional to the stock price) continuously to the shareholder at rate $\beta \geq 0$ whenever its price is above $1$. We suppose that the shareholder is interested in determining the profitability of such stock in terms of cost-effectiveness, so she uses the following ``rating device'' at time $t\geq 0$, 
%\begin{align*}
%\int_0^t e^{-rs} (Y_se^{\beta V_s}-K) \dd s,
%\end{align*}
%where $r>0$ is the risk free interest rate, $K\in (0,1)$ and $V_t=t-\sup\{0\leq s\leq t: Y_s\leq 1 \}=t-U_t$. The term $K\in(0,1)$ is interpreted as the opportunity cost so that, when the stock price is above $K$, the stock is considered profitable and unprofitable otherwise. Then, when $Y_s<K$, the stock receives a negative rating, and when $Y_s>K$, due to the dividends, the stock receives the positive rating $ Y_se^{\beta V_s}-K$. The agent seeks to keep the stock while it is still profitable. Hence, the best moment of selling is determined by the optimal stopping problem 
%\begin{align*}
%V(u,x)=\sup_{\tau \in \mathcal{T}}\E_{u,x}\left( \int_0^{\tau } e^{-rs}(Y_se^{\beta V_s}-K)\dd s  \right),
%\end{align*}
%for any $(u,x)\in E$.

We then see from Proposition \ref{prop:SolutiontocorporatebankrupcyOS} that, upon taking $\delta(v,y)=ye^{\beta v}$ and $c(v,y)=K$ in \eqref{eq:optimalstoppingproblemcorparebankrupcy}, with $K$, $\beta$ and $r$ as in Proposition \ref{prop:SolutiontocorporatebankrupcyOS}, we conclude that the bankruptcy of the company occurs when $Y_t=\exp(X_t)$ crosses below the level $e^{z^*}$. 
\end{example}

\subsection{Optimal prediction problems}
\label{sec:optimalpredictionproblems}
Let $X$ be a stochastic process with state space in $\R$ and let $\theta$ be a last passage time of $X$, that is, $\theta =\sup\{t\geq 0: X_t\in A \}$, where  $A\subset \R$. The recent literature has solved the problem of finding a stopping time approximating a specific last passage time. There are, for example, various papers in which the approximation is in $L_1$ sense. That is, the following optimal prediction problem is solved:
\begin{align}
\label{eq:L1optimalpredictionproblem}
\inf_{\tau \in \mathcal{T}} \E(|\tau-\theta|).
\end{align}
To mention a few: \cite{du2008predicting} predicted the last zero of a Brownian motion with drift in a finite horizon setting; \cite{duToit2008} predicted the time of the ultimate maximum at time $t=1$ for a Brownian motion with drift is attained; \cite{shiryaev2009} focused on the last time of the attainment of the ultimate maximum of a Brownian motion and proceeded to show that it is equivalent to predicting the last zero of the process in this setting; \cite{glover2013three} predicted the time in which a transient diffusion attains its ultimate minimum; \cite{glover2014optimal} predicted the last passage time of a level $z > 0$ for an arbitrary nonnegative time-homogeneous transient diffusion; \cite{baurdoux2014predicting} predicted the time at which a L\'evy process attains its ultimate supremum and \cite{baurdoux2016optimal} predicted when a positive self-similar Markov process attains its path-wise global supremum or infimum before hitting zero for the first time and \cite{baurdoux2018predicting} predicted the last zero of a spectrally negative L\'evy process.\\

From here onwards, consider $X$ to be a spectrally negative L\'evy process that drifts to infinity and let $g=\sup\{ t\geq 0: X_t\leq 0\}$, the last zero of $X$. The problem \eqref{eq:L1optimalpredictionproblem} can be generalised to any convex function $d:\R_+ \times \R_+ \mapsto \R_+$. That is, under the assumption that $\E(d(0,g))<\infty$, consider the optimal prediction problem:
\begin{align}
\label{eq:doptimalpredictionproblem}
V_d=\inf_{\tau \in \mathcal{T}} \E(d(\tau,g)).
\end{align}
As is in the case for problem \eqref{eq:L1optimalpredictionproblem}, the problem \eqref{eq:doptimalpredictionproblem} cannot be solved using standard techniques of optimal stopping (cf. \cite{peskir2006optimal}) since the random variable $g$ depends on the whole path of the process $X$ and hence is only $\F$ measurable. However, the following Lemma provides an equivalence between the optimal prediction problem above and an optimal stopping problem driven by the process $\{(g_t,t, X_t),t\geq 0\}$.
\begin{lemma}
Let $X$ be a spectrally negative L\'evy process drifting to infinity and $d:\R_+ \times \R_+ \mapsto \R_+$ a convex function such that $\E(d(0,g))<\infty$. We have for each $\tau \in \mathcal{T}$,
\begin{align*}
\E(d(\tau,g))&=\E\left( \int_0^{\tau}  G_d(g_s,s,X_s) \dd s+d(0,g)\right),
\end{align*}
where $G_d(\gamma,s,x)=\frac{\partial }{\partial x} d_+(s,\gamma)\psi'(0+)W(x)+\E_{x}\frac{\partial }{\partial x} d_+(s,g+s) \I_{\{g>0\}})$ and $\frac{\partial }{\partial x} d_+$ is the right derivative with respect to the first argument of $d$.
\end{lemma}

\begin{proof}
Let $\tau \in \mathcal{T}$. Using the integral representation of convex functions, we obtain that

\begin{align*}
\E(d(\tau,g))&=\E\left( \int_0^{\tau}  \frac{\partial }{\partial x} d_+(s,g)\dd s+d(0,g)\right),
\end{align*}
where $\frac{\partial }{\partial x} d_+$ is the right-derivative of $d$ with respect to its first coordinate. Then, using Fubini's theorem and the tower property for conditional expectation, we see that

\begin{align*}
\E\left( \int_0^{\tau}  \frac{\partial }{\partial x} d_+(s,g)\dd s\right) 
%&=\E\left( \int_0^{\infty}  \frac{\partial }{\partial x} d_+(s,g) \I_{\{ s\leq \tau\}} \dd s\right) \\
%&= \int_0^{\infty}  \E\left( \frac{\partial }{\partial x} d_+(s,g) \I_{\{ s\leq \tau\}}\dd s\right) \\
&= \int_0^{\infty}  \E\left[\I_{\{ s\leq \tau\}} \E\left( \frac{\partial }{\partial x} d_+(s,g) \bigg|\F_s\right) \dd s\right] \\
&=\E\left[\int_0^{\tau} \E\left( \frac{\partial }{\partial x} d_+(s,g) \bigg|\F_s\right) \dd s \right]. \\
\end{align*}
Hence, we proceed to find an expression for the conditional expectation inside the last integral. From the strong Markov property of the process $\{(g_t,t,X_t),t\geq 0\}$ we have that
\begin{align*}
\E\left( \frac{\partial }{\partial x} d_+(s,g) \bigg|\F_s\right)
&=\E_{g_s,s,X_s}  \left( \frac{\partial }{\partial x} d_+(s,g)\right).
\end{align*}
From \eqref{eq:initialmeausreP}, we have that for $(\gamma,s,x)\in E_g$,
\begin{align*}
\E_{\gamma,s,x}  \left( \frac{\partial }{\partial x} d_+(s,g)\right)&=\E_{x}  \left( \frac{\partial }{\partial x} d_+(s,\gamma)\I_{\{\tau_0^-=\infty \}}\right)+\E_{x}  \left( \frac{\partial }{\partial x} d_+(s,g+s)\I_{\{\tau_0^-<\infty \}}\right)\\
&=\frac{\partial }{\partial x} d_+(s,\gamma) \psi'(0+)W(x)+\E_x\left(\frac{\partial }{\partial x} d_+(s,g+s)\I_{\{ g>0\}}\right)\\
&=G_d(\gamma,s,x).
\end{align*}
So that, for any $s\geq 0$,
\begin{align*}
\E\left( \frac{\partial }{\partial x} d_+(s,g) \bigg|\F_s\right)=G_d(g_s,s,X_s).
\end{align*}
The result then follows.
\end{proof}
The lemma above directly implies that solving the optimal prediction problem \eqref{eq:doptimalpredictionproblem} is equivalent to solving the optimal stopping problem
\begin{align}
\label{eq:doptimalstoppingproblem}
\inf_{\tau \in \mathcal{T}} \E_{\gamma,t,x}\left[ \int_0^{\tau} G_d(g_{s+t},s+t,X_{s+t}) \dd s \right],
\end{align}
for each $(\gamma,t,x)\in E_g$. In \cite{baurdoux2020lp}, the case when $d(x,y)=|x-y|^p$ with $p>1$ is solved, that is, $g$ is approximated by a stopping time using an $L_p$ distance. In this case, the problem \eqref{eq:doptimalstoppingproblem} reads as
\begin{align}
\label{eq:Lpoptimalstoppingproblem}
V(u,x)=\inf_{\tau \in \mathcal{T}}\E_{u,x}\left( \int_0^{\tau} G(U_s,X_s)\dd s \right),
\end{align}
for $(u,x)\in E$, where $G(u,x)=u^{p-1}\psi'(0+)W(x)-\E_x(g^{p-1})$. Although in \cite{baurdoux2020lp} a rather general spectrally negative L\'evy process is considered (only integrability conditions on the L\'evy measure are imposed), for the sake of simplicity, we include here the main results (see Theorem 3.3 in \cite{baurdoux2020lp}) when $X$ is a Brownian motion with positive drift (with Gaussian coefficient $\sigma>0$). It is shown that an optimal stopping time for \eqref{eq:Lpoptimalstoppingproblem} is given by 
\begin{align*}
\tau_D=\inf\{t\geq 0: X_t\geq b(U_t) \},
\end{align*}
where $b$ is a strictly positive, non-increasing and continuous function such that $\lim_{u\rightarrow\infty }b(u)=0$ and $\lim_{u\downarrow 0}b(u)=\infty$. Moreover, the function $b$ and the value $V(0,0)$ are characterised as the only solution to the non-linear equations
\begin{align*}
0&=V(0,0)\frac{\sigma^2}{2}W'(b(u))+    \E_{b(u)}\left( \int_0^{\tau_0^-} G(u+s,X_s) \I_{\{X_s<b(u+s) \}}\dd s \right),\\
0&=V(0,0)\frac{\sigma^2}{2}W''(0+)+\frac{\partial}{\partial x} \E_{x}\left( \int_0^{\tau_0^-} G(u+s,X_s) \I_{\{X_s<b(u+s) \}}\dd s \right)\bigg|_{x=0,u=0}+\int_{[0,\infty)} \E_{-u}(g^{p-1}) W(\dd u),
\end{align*}
where $b$ is considered in the class of continuous functions bounded by below by $h(u):=\inf\{x\geq 0: G(u,x)\geq 0 \}$ and $V(0,0)<0$.\\

Note that properties of the stochastic process $(U, X)$ were needed to derive the above result. For instance, the Markov property of $(U, X)$ is crucial to solving the optimal stopping problem \eqref{eq:Lpoptimalstoppingproblem} using the standard theory of optimal stopping. Moreover, the explicit version of the infinitesimal generator \eqref{eq:infinitesimalgeneratorofgtX} and formula \eqref{eq:calucationofintegraluptoinfinity2ndformula}  played a crucial role in deriving the non-linear equations presented above. In particular, given the unusual shape of the set $E$, \eqref{eq:calucationofintegraluptoinfinity2ndformula} gives us a method to show that there is smooth pasting at the point $(0,0)$ for the function $V$, which allowed us to propose a characterisation of the value $V(0,0)$.

\section{Main proofs}
\label{sec:proofs}
This section is dedicated to presenting the main proofs of this paper. We start by including the proof of Proposition \ref{prop:Markovproperty}.

\subsection{Proof of Proposition \ref{prop:Markovproperty}}
From the definition of $g_t$, it is easy to note that for all $t\geq 0$, we have $X_t \leq 0$ if and only if $g_t=t$, from which we obtain that $(g_t,t, X_t)$ can take only values in $E_g$. Now, we proceed to show the strong Markov property holds. Consider a measurable positive function $h: E_g \mapsto \R$. Then, we have for any stopping time $\tau$ and $s\geq 0$,
\begin{align*}
\E( h(g_{\tau+s},\tau+s,X_{\tau+s})|\F_{\tau})&=\E(h(g_{\tau} \vee \sup\{ r\in [\tau,s+\tau]: X_{r} \leq 0 \},\tau+s,X_{\tau+s} )|\F_{\tau})\\
&=\E( h(g_{\tau} \vee \sup\{ r\in [\tau,s+\tau]: \widetilde{X}_{r-\tau}+X_{\tau} \leq 0 \},\tau+s,\widetilde{X}_s+X_{\tau} )|\F_{\tau}),
\end{align*}
where $\widetilde{X}_r=X_{r+\tau}-X_{\tau}$ and $a\vee b:=\max\{ a,b \}$ for any $a,b\in \R$. Using the strong Markov property for L\'evy processes and the fact that $g_{\tau}$ and $X_{\tau}$ are $\F_{\tau}$ measurable we obtain that
\begin{align*}
\E( h(g_{\tau+s},\tau+s,X_{\tau+s})|\F_{\tau})=f_s(g_{\tau},\tau,X_{\tau}),
\end{align*}
where for any $x\in \R$ and $0\leq \gamma \leq t$, the function $f_s$ is given by
\begin{align*}
f_s(\gamma,t ,x)&=\E_x( h(\gamma \vee \sup\{ r\in [t,s+t]: X_{r-t} \leq 0 \},t+s,X_s )).
\end{align*}
Note that, on the event $\{\sigma_{0}^->s\}$, the set $\{ r\in [t,s+t]: X_{r-t} \leq 0 \}=\emptyset$. Then, $\gamma \vee \sup\{ r\in [t,s+t]: X_{r-t} \leq 0 \}=\gamma$, where we used the convention that $\sup \emptyset =0$. Otherwise, in the event $\{\sigma_{0}^-\leq s\}$, we have that $\{ r\in [t,s+t]: X_{r-t} \leq 0 \}\neq \emptyset$ and then $\sup\{ r\in [t,s+t]: X_{r-t} \leq 0 \} \geq t\geq \gamma$. Hence, we have that, in the event $\{\sigma_{0}^-\leq s\}$,
\begin{align*}
\gamma \vee \sup\{ r\in [t,s+t]: X_{r-t} \leq 0 \}&=\sup\{ r\in [t,s+t]: X_{r-t} \leq 0 \}\\
&=t+\sup\{ r\in [0,s]: X_{r} \leq 0 \}\\
&=t+g_s.
\end{align*}
 Therefore, for any $x\in \R$ and $0\leq \gamma \leq t$, the function $f_s$ takes the form
\begin{align*}
f_s(\gamma,t ,x)&=\E_x(h(\gamma,t+s, X_s)\I_{\{ \sigma_0^->s\}})+\E_x(h(g_s+t,t+s, X_s)\I_{\{ \sigma_0^-\leq s\}}).
\end{align*}
On the other hand, similar calculations lead us to
\begin{align*}
\E( h(g_{\tau+s},\tau+s,X_{\tau+s})|\sigma(g_{\tau},\tau,X_{\tau}))=f_s(g_{\tau},\tau,X_{\tau}).
\end{align*}
Hence, for any measurable positive function $h$, we obtain 
\begin{align*}
\E( h(g_{\tau+s},\tau+s,X_{\tau+s})|\F_{ \tau})&=\E( h(g_{\tau+s},\tau+s,X_{\tau+s})|\sigma(g_{\tau},\tau,X_{\tau})).
\end{align*}
Therefore, we conclude that the process $\{ (g_t,t,X_t),t\geq 0 \}$ is a strong Markov process.

\subsection{Perturbed L\'evy process}
Suppose that $X$ is a spectrally negative L\'evy process of finite variation. Then, with probability one, it takes a positive amount of time to cross below $0$, that is, $\tau_0^->0$ $\P$-a.s. Hence, stopping at the consecutive times at which $X$ is below zero and together with the ideas mentioned in Remark \ref{rem:behaviourofgt}, we can fully describe the behaviour of $g_t$ and then derive the results in Theorems \ref{thm:ItoformulaforgttXt} and \ref{thm:integralofgr}. However,  when $X$ is of infinite variation, it is well known that the closure of the set of zeroes of $X$ is perfect and nowhere dense, and the mentioned approach is no longer useful (since we have that $\tau_0^-=0$ $\P$-a.s). Therefore, we use a perturbation method to exploit the idea applicable to finite variation processes. This method, which is mainly based on the work of \cite{dassios2011} and \cite{revuz2004continuous} (see Theorem VI.1.10), consists of constructing a new ``perturbed'' process $X^{(\varepsilon)}$ (for $\varepsilon$ sufficiently small) that approximates $X$, with the property that $X^{(\varepsilon)}$ visits the level zero a finite number of times before any time $t\geq 0$. Then we approximate $g_t$ by the corresponding last zero process of $X^{(\varepsilon)}$.\\

We formally describe the construction of the ``perturbed'' process $X^{(\varepsilon)}$. Let $\varepsilon>0$, define the stopping times $\rho_{1,\varepsilon}^-=0$ and for any $k \geq 1$,
\begin{align*}
\rho_{k,\varepsilon}^+&:=\inf\{ t>\rho_{k,\varepsilon}^-: X_t \geq \varepsilon \},\\
\rho_{k+1,\varepsilon}^-&:=\inf\{ t>\rho_{k,\varepsilon}^+: X_t< 0\},
\end{align*}
where we use the usual convention that $\inf \emptyset=\infty$. We define the auxiliary process $X^{(\varepsilon)}=\{X_t^{(\varepsilon)} ,t\geq 0 \}$, where for $t\geq 0$,

\begin{align}
\label{eq:definitionofXepsilon}
X_t^{(\varepsilon)}=\left\{
\begin{array}{ll}
X_t-\varepsilon, &  \rho_{k,\varepsilon}^-\leq t< \rho_{k,\varepsilon }^+, \\
X_t, & \rho_{k,\varepsilon}^+\leq t< \rho_{k+1,\varepsilon }^- . \\
\end{array}
\right.
\end{align}
In Figure \ref{fig:samplepathofperturbedlevy} we include a sample path of the process $X^{(\varepsilon)}$ compared with the original process $X$. \\
\begin{figure}[hbtp]
\centering
\includegraphics[scale=0.41]{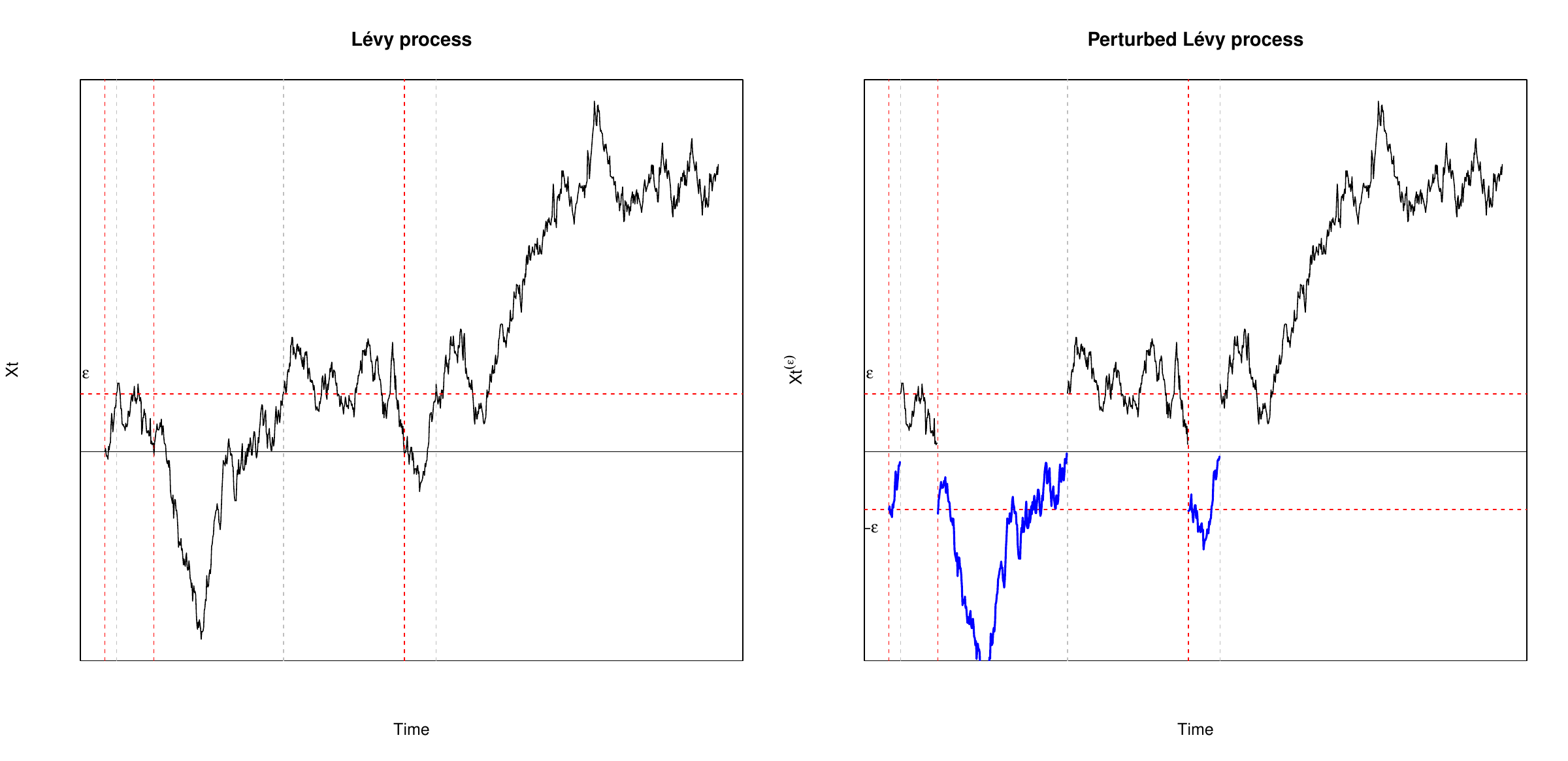}
\caption{Left: Sample path of $X$. Right: Sample path of the perturbed process $X^{(\varepsilon)}$. The red vertical lines correspond to the sequence of stopping times $\{\rho_{k,\varepsilon}^-,k\geq 1 \}$, whereas the grey vertical lines correspond to $\{\rho_{k,\varepsilon}^+,k\geq 1 \}$.}
\label{fig:samplepathofperturbedlevy}
\end{figure}

It is straightforward from the definition of $X^{(\varepsilon)}$ that $X_t -\varepsilon \leq X_t^{(\varepsilon)} \leq X_t$, and that $X^{(\varepsilon)} \uparrow X $ uniformly when $\varepsilon \downarrow 0$, i.e.,
\begin{align*}
%\label{eq:uniformconvergenceperturbedlevy}
\lim_{\varepsilon \downarrow 0} \sup_{t\geq 0} |X_t^{(\varepsilon)}-X_t|=0.
\end{align*}
In addition, we define the last zero process $g_{\varepsilon,t}$ associated to the process $X^{(\varepsilon)}$, that is, 
\begin{align*}
g_{\varepsilon,t}=\sup\{0\leq s\leq t: X_s^{(\varepsilon)}\leq 0 \}
\end{align*}
for $\varepsilon>0$ and $t\geq 0$. The inequality $g_t\leq g_{\varepsilon,t} \leq g_t^{(\varepsilon)}$ holds for all $t\geq 0$. Taking $\varepsilon\downarrow 0$, and by the right continuity of $x\mapsto g_t^{(x)}$, we obtain that $ g_{\varepsilon,t} \downarrow  g_t$ when $\varepsilon \downarrow 0$ for all $t\geq 0$. Moreover, we have that $t-g_{\varepsilon,t}=:U_{\varepsilon,t} \uparrow U_t$ when $\varepsilon \downarrow 0$ for all $t\geq 0$. \\

Recall that the local time at the level $a\in \R$, $L^a=\{L_t^a,t\geq 0 \}$, is a continuous process defined in terms of the It\^o--Tanaka formula (see, e.g., Theorem IV.68 in \cite{protter2005} on p. 216 and \eqref{eq:definitionofLocaltimebtanaka}) and its measure $\dd L_t^a$ is carried by the set $\{s\geq 0: X_{s-}=X_s=a \}$. For ease of notation, we denote $L_t=L_t^0$ for every $t\geq 0$, the local time at the level zero. In this case (see e.g. Corollary IV.1 in \cite{protter2005} on p. 219), we have that for any bounded and measurable function $g$,
\begin{align*}
\int_{-\infty}^{\infty} L_t^a g(a)\dd a=\sigma \int_0^t g(X_s) \dd s, \qquad \text{a.s.} 
\end{align*}

For each $\varepsilon>0$ and $t\geq 0$, we define 
\begin{align*}
M_t^{(\varepsilon)}=\sum_{k=1}^{\infty} \I_{\{\rho_{k,\varepsilon}^-<t \}}.
\end{align*}
Note that $M_t^{(\varepsilon)}-1$ is the number of downcrossings of the level zero at time $t> 0$ of the process $X_t^{(\varepsilon)}$. It turns out that $M_t^{(\varepsilon)}$ works as an approximation of the local time at zero in some sense. We have the following lemma. The proof follows an argument similar to the one in \cite{revuz2004continuous} (see Theorem VI.1.10 on p. 227).

\begin{lemma}
\label{lemma:Mtlocaltimeconvergence}
Suppose that $X$ is a spectrally negative L\'evy process. Then for all $t\geq 0$,
\begin{align*}
\lim_{\varepsilon \downarrow 0}\varepsilon M_t^{(\varepsilon)} =\frac{1}{2} L_t \qquad \text{in probability}.
\end{align*}
\end{lemma}
\begin{proof}
From the Meyer--It\^o formula (see, e.g., \cite{protter2005}, Theorem IV.68 and Theorem IV.70 in pp. 216,218) we know that
\begin{align*}
X_t^+=X_0^++\int_{(0,t]} \I_{\{X_{s-}>0\}}\dd X_s +\int_{(0,t]} \int_{(-\infty,0)} (X_{s-}+y)^- \I_{\{X_{s-}>0\}} N(\dd s \times \dd y)+\frac{1}{2} L_t,
\end{align*}
where $x^+$ and $x^-$ are the positive and negative part, respectively, of $x$ defined by $x^+=\max\{ x,0\}$ and $x^-=-\min\{x,0\}$.
Hence, for $t\geq 0$ and $1 \leq k\leq M_t^{(\varepsilon)}$ we get that 

\begin{align*}
X_{\rho_{k,\varepsilon}^+ \wedge t}^+ -X_{\rho_{k,\varepsilon}^-}^+
&=\int_{(\rho_{k,\varepsilon}^-,\rho_{k,\varepsilon}^+\wedge t]} \I_{\{X_{s-}>0\}}\dd X_s +\int_{(\rho_{k,\varepsilon}^-,\rho_{k,\varepsilon}^+ \wedge t]} \int_{(-\infty,0)} (X_{s-}+y)^- \I_{\{X_{s-}>0\}} N(\dd s \times \dd y)\\
&\qquad +\frac{1}{2} (L_{\rho_{k,\varepsilon}^+ \wedge t}-L_{\rho_{k,\varepsilon}^-}).
\end{align*}
From the definition of the stopping times $\rho_{k,\varepsilon}^-$, we have that $X_r>0$ when $r\in [\rho_{k,\varepsilon}^+, \rho_{k+1,\varepsilon}^-)$ for some $k \geq 1$, and, since $L$ is continuous and only charge points in the set of zeros of $X$, we have that $L_{\rho_{k,\varepsilon}^+}=L_{\rho_{k+1,\varepsilon}^-}$ and $L_{t \wedge \rho_{M_t^{(\varepsilon)},\varepsilon}^+}=L_t$. Hence, using a telescopic sum and the fact that $g_{\varepsilon,r-}=r$ if and only if $r \in (\rho_{k,\varepsilon}^-,\rho_{k,\varepsilon}^+]$, for some $k\geq 1$, we have that

\begin{align*}
X_{t \wedge \rho_{M_t^{(\varepsilon)},\varepsilon}^+}^+ &-X_{\rho_{M_t^{(\varepsilon)},\varepsilon}^-}^+ +\sum_{k=1}^{M_t^{(\varepsilon)}-1} (X_{\rho_{k,\varepsilon}^+}^+ -X_{\rho_{k,\varepsilon}^-}^+ ) \\
& =\int_{(0,t]} \I_{\{ g_{\varepsilon,s-}=s\}}\I_{\{X_{s-}>0\}}\dd X_s +\int_{(0,t]} \int_{(-\infty,0)} (X_{s-}+y)^- \I_{\{g_{\varepsilon,s-}=s\}} \I_{\{X_{s-}>0\}} N(\dd s \times \dd y)+\frac{1}{2} L_t.
\end{align*}
Thus, since $X_{\rho_{k,\varepsilon}^-}\leq 0$ and $X_{\rho_{k,\varepsilon}^+}=\varepsilon$ on the event $\{\rho_{k,\varepsilon}^+<\infty \}$ for all $k\geq 1$, we obtain that for any $t\geq 0$ and $\varepsilon>0$, 

\begin{align*}
X_{t \wedge \rho_{M_t^{(\varepsilon)},\varepsilon}^+}^+ & +\varepsilon (M_t^{(\varepsilon)}-1) \\
& =\int_{(0,t]} \I_{\{ g_{\varepsilon,s-}=s\}}\I_{\{X_{s-}>0\}}\dd X_s +\int_{(0,t]} \int_{(-\infty,0)} (X_{s-}+y)^- \I_{\{ g_{\varepsilon,s-}=s\}} \I_{\{X_{s-}>0\}} N(\dd s \times \dd y)+\frac{1}{2} L_t.
\end{align*}
Note that $0\leq X_{t \wedge \rho_{M_t^{(\varepsilon)},\varepsilon}^+}^+ \leq \varepsilon$ and then $\lim_{\varepsilon \downarrow 0} X_{t \wedge \rho_{M_t^{(\varepsilon)},\varepsilon}^+}^+=0$. Moreover, from the inequality $g_t\leq g_{\varepsilon,t} \leq g_t^{(\varepsilon)}$ we see that 
\begin{align*}
\I_{\{g_{t-}^{(\varepsilon)}<t \}}\leq \I_{\{g_{\varepsilon,t-}<t \}}\leq  \I_{\{g_{t-}<t \}}
\end{align*}
for any $t>0$ and $\varepsilon>0$. Thus, since $g_{t-}^{(\varepsilon)} \downarrow g_{t-}$ when $\varepsilon \downarrow 0$, and the mapping $x\mapsto \I_{\{ x<t \}}$ is right-continuous, for fixed $t>0$, we conclude that $\lim_{\varepsilon \downarrow 0} \I_{\{g_{\varepsilon,t-}<t \}}=\I_{\{g_{t-}<t \}}$ for all $t>0$. Therefore, from the dominated convergence theorem for stochastic integrals (see, for example, Theorem IV.15 and Theorem IV.32 in \cite{protter2005} on pp. 166,176, respectively), we have that the first term in the right-hand side of the equation above converges to $0$ uniformly on compacts in probability, that is, for all $t> 0$, 
\begin{align*}
\sup_{0\leq s\leq t} \left|\int_{(0,s]}  \I_{\{ g_{\varepsilon,r-}=r\}}\I_{\{X_{r-}>0\}}\dd X_r  \right| 
\end{align*}
converges to $0$ in probability when $\varepsilon \downarrow 0$. Note that, for all $s\geq 0$, we have that $(X_{s-}+y)^- \I_{\{ g_{\varepsilon,s-}=s\}} \I_{\{X_{s-}>0\}}  \leq (X_{s-}+y)^-  \I_{\{X_{s-}>0\}}$ and that

\begin{align*}
 \int_{(0,t]} \int_{(-\infty,0)} (X_{s-}+y)^-  \I_{\{X_{s-}>0\}} N(\dd s \times \dd y) <\infty,
\end{align*}
for all $t\geq 0$. Then, by the dominated convergence theorem

\begin{align*}
\lim_{\varepsilon \downarrow 0}\int_{(0,t]} \int_{(-\infty,0)} (X_{s-}+y)^- \I_{\{g_{\varepsilon,s-}=s\}} \I_{\{X_{s-}>0\}} N(\dd s \times \dd y)=0,
\end{align*}
for any $t\geq 0$. Thus, for fixed $t\geq 0$, we have that $\varepsilon M_t^{(\varepsilon)}$ converges to $L_t/2$ in probability when $\varepsilon \downarrow 0$.
%We know that there exists a decreasing sub-sequence $\{ \varepsilon_{n}, n\geq 1\}$ converging to $0$ such that $ \lim _{n\rightarrow \infty}\varepsilon_n  M_t^{(\varepsilon_n)}=L_t/2$, $\P$-a.s. From the fact that $M_t^{(\varepsilon)}$ increases when $\varepsilon$ decreases, we  have that for each $\varepsilon \in [\varepsilon_{n+1}, \varepsilon_n]$,
%
%\begin{align*}
%\varepsilon_{n+1} M_t^{(\varepsilon_n)} \leq \varepsilon M_t^{(\varepsilon)} \leq \varepsilon_{n} M_t^{(\varepsilon_{n+1})}.
%\end{align*}
%Hence, we conclude that $\lim_{\varepsilon \downarrow 0} \varepsilon M_t^{(\varepsilon)}=L_t/2$ $\P$-a.s as claimed.
\end{proof}

\begin{rem}
\label{rem:behaviourofget}
For all $\varepsilon>0$ fixed, we can describe the paths of the process $\{ g_{\varepsilon,t},t\geq 0\}$ in terms of the stopping times $\{ (\rho_{k,\varepsilon}^-,\rho_{k,\varepsilon}^+),k\geq 1\}$. When $X_{t}^{(\varepsilon)}\leq 0$ we have that $\rho_{k,\varepsilon}^- \leq t <\rho_{k,\varepsilon}^+$, for some $k\geq 1$, and then $g_{\varepsilon,t}=t$. Similarly, when $X_t^{(\varepsilon)} > 0$, there exists $k\geq 1$ such that $\rho_{k,\varepsilon}^+ \leq t <\rho_{k+1,\varepsilon}^-$, and hence, $g_{\varepsilon,t}=\rho_{k,\varepsilon}^+$.  The reader can refer to Figure \ref{fig:samplepathofperturbedlevy} for a graphical representation of this fact.
\end{rem}

We conclude this section by stating the following Proposition for stopped processes that will be of use in the proof of Theorem \ref{thm:ItoformulaforgttXt}. We introduce some additional notation, for any $-\infty\leq b<a\leq \infty$ and any stopping time $\tau$, we define the following stopping time 
\begin{align*}
T_{a,b,\tau}=\inf\{t\geq \tau: X_t\notin [b,a) \}.
\end{align*}
We also recall that $i=1$ if $X$ is of finite variation and $i=2$ when $X$ is of infinite variation.
\begin{prop}
\label{prop:stoppedItoformula} 
Suppose that $X$ is any spectrally negative process. Take $-\infty\leq b<a\leq \infty$ and define $C=[0,\infty)\times (b,a)$. Let $F:\R_+\times \R \mapsto \R$ be a function such that $F$ is $C^{1,i}$ on $\bar{C}$ and, if $a<\infty$, we assume that $F$ satisfies $\lim_{h \downarrow 0} F(t,a-h)=F(t,a)$ for all $t\geq 0$. Further, if $\sigma>0$ and $b>-\infty$, assume that $\lim_{h\downarrow 0}F(t,b+h)=F(t,b)$ for all $t\geq 0$.Then, for any stopping time $\tau$ and $t\geq 0$ we have
\begin{align}
F(t\wedge T_{a,b,\tau}&,X_{t\wedge T_{a,b,\tau}})\nonumber\\
&= F(t\wedge \tau ,X_{t\wedge \tau})+\int_{t\wedge \tau}^{t\wedge T_{a,b,\tau}} \frac{\partial F}{\partial t}(s,X_{s-})\dd s+\int_{(t\wedge \tau, t\wedge T_{a,b,\tau}]} \frac{\partial F}{\partial x} (s,X_{s-})\dd X_s \nonumber\\
&\qquad +\frac{1}{2}\sigma^2 \int_{t\wedge \tau}^{t\wedge T_{a,b,\tau}} \frac{\partial^2 F}{\partial x^2} (s,X_{s-})\dd s \nonumber\\
\label{eq:Itononstoppedprocesses}
&\qquad+\int_{(t\wedge \tau,t\wedge T_{a,b,\tau}]} \int_{(-\infty,0)}[F(s,X_{s-}+y)-F(s,X_{s-})-y \frac{\partial F}{\partial x} (s,X_{s-}) ]N(\dd s\times \dd y).
\end{align}
\end{prop}

\begin{proof}
We first note that $t\wedge \tau \leq t\wedge T_{a,b,\tau} $ and that the statement trivially holds when $t\leq \tau$ or when $X_{\tau} \notin [b,a) $, as $T_{a,b,\tau}=\tau$ in the latter case. Moreover, if $X$ is of infinite variation and $X_{\tau}=b$, it can also be easily seen that $T_{a,b,\tau}=\tau$. Further, from the Markov property we see that if $a<\infty$,
\begin{align*}
\P(X_{ T_{a,b,\tau} }>a,X_{\tau}\in [b,a), T_{a,b,\tau}<\infty)=\E\left[\I_{\{X_{\tau}\in [b,a), T_{a,b,\tau}<\infty \}}  \P_{X_{\tau} }( X_{ \tau_a^+ \wedge \tau_b^-}  >a )\right]=0,
\end{align*}
where we used that $X_{\tau_a^+}=a$ with probability one on the event $\{ 0<\tau_a^+<\infty\}$ and $X_{\tau_b^-}\leq b$ on the event $\{0<\tau_b^-<\infty \}$, whenever $b>-\infty$. Moreover, if $\sigma=0$ and $b>-\infty$, we deduce by the strong Markov property that
\begin{align*}
\P(X_{ T_{a,b,\tau} }=b,X_{\tau}\in (b,a), T_{a,b,\tau}<\infty)
%=\E\left[\I_{\{X_{\tau}\in (b,a) , T_{a,b,\tau}<\infty\}}  \P_{X_{\tau} }( X_{ \tau_a^+ \wedge \tau_b^-} =b )\right]
&=\E\left[ \I_{\{X_{\tau}\in (b,a) \}} \P_{X_{\tau} }( X_{ \tau_b^-} =b, \tau_b^-<\tau_a^+ )\right]\\
&\leq \E\left[ \I_{\{X_{\tau}\in (b,a) \}}  \P_{X_{\tau} }( X_{ \tau_b^-} =b, \tau_b^-<\infty )\right]\\
&=0,
\end{align*}
where we used that $X_{\tau_b^-}=b$ with positive probability, on the event $\{0< \tau_b^-<\infty\}$, only when $\sigma>0$ (see \eqref{eq:probabilityofcreepingdownwards}). Similarly, when $X$ is of finite variation, we can see that $\P(X_{ T_{a,b,\tau} }=b, T_{a,b,\tau}<\infty,X_{\tau}=b)=0$.
%Note that the assertion trivially follows if $X_0\notin(a,b)$, since in this case $\tau_a^+\wedge\tau_b^-=0$. Hence, we assume that $X_0\in (b,a)$.
We define the auxiliary function 
\begin{align*}
\widehat{F}(t,x)=\left\{
\begin{array}{ll}
F(t,a-), & t\geq 0 \text{ and } x\geq a,\\
F(t,x), & t\geq 0 \text{ and } b<x< a,\\
F(t,b+), & t\geq 0 \text{ and } x\leq b.
\end{array}
\right.
\end{align*}
Hence, we have that $\widehat{F}$ is continuous on $\R_+\times \R$ and is such that $\widehat{F}$ is $C^{1,i}$ on $\bar{C}$, $\widehat{F}$ is $C^{1,i}$ on $\bar{D}_1$ and $\widehat{F}$ is $C^{1,i}$ on $\bar{D}_2$, where $D_1=[0,\infty)\times (a,\infty)$ and $D_2=[0,\infty)\times (-\infty,b)$. Using that $F(t,a)=F(t,a-)$ and $F(t,b+)=F(t,b)$, when $\sigma>0$, we see that 
\begin{align*}
F(t\wedge T_{a,b,\tau},X_{t\wedge T_{a,b,\tau}})
&=F(t,X_{t})\I_{\{t<T_{a,b,\tau} \}} +F(T_{a,b,\tau},a)\I_{\{T_{a,b,\tau} \leq t \}}\I_{\{X_{T_{a,b,\tau}} =a  \}}\\
&\qquad +F( T_{a,b,\tau},b)\I_{\{T_{a,b,\tau}\leq t  \}}\I_{\{X_{T_{a,b,\tau}}=b \}}+F( T_{a,b,\tau},X_{ T_{a,b,\tau}})\I_{\{T_{a,b,\tau}\leq t  \}}\I_{\{X_{T_{a,b,\tau}}<b \}}\\
%&=F(t,X_{t})\I_{\{t<\tau_b^- \}}+F( \tau_b^-,b+)\I_{\{\tau_b^-\leq t \}}+[F( \tau_b^-,X_{ \tau_b^-})-F(\tau_b^-,b+)]\I_{\{\tau_b^-\leq t \}}\\
&=\widehat{F}(t\wedge T_{a,b,\tau},X_{t\wedge T_{a,b,\tau}})+[F( T_{a,b,\tau},X_{T_{a,b,\tau}})-F(T_{a,b,\tau},b+)]\I_{\{T_{a,b,\tau}\leq t \}}\I_{\{X_{T_{a,b,\tau}}<b \}}\\
&=\widehat{F}(t\wedge T_{a,b,\tau},X_{t\wedge T_{a,b,\tau}})\\
&\qquad+\int_{(t\wedge \tau ,t\wedge T_{a,b,\tau}]}\int_{(-\infty,0)}[F( s,X_{ s-}+y)-F(s,b+)]\I_{\{X_{s-}+y \leq b\}} N(\dd s \times \dd y),
\end{align*}
Thus, applying the version for multiple curves of formula \eqref{eq:semimartingalelocaltimeItoformula} (see Remark 2.2 and Remark 3.3 on \cite{peskir2007change}) at the times $t\wedge T_{a,b,\tau}$ and $t\wedge \tau$, respectively, upon using that $L_{t\wedge T_{a,b,\tau}}^a=L_{t\wedge \tau}^a$ and $L_{t\wedge T_{a,b,\tau}}^b=L_{t\wedge \tau}^b$ as the local time is continuous and $X_{s-}\in (b,a)$ for $s\in ( \tau,t\wedge T_{a,b,\tau})$ when $\tau< t$ and $\tau<T_{a,b,\tau}$, we obtain that
\begin{align*}
\widehat{F}&(t\wedge T_{a,b,\tau},X_{t\wedge T_{a,b,\tau}})\\
&= \widehat{F}(t\wedge \tau,X_{t\wedge \tau})+\int_{t\wedge \tau}^{t\wedge  T_{a,b,\tau}} \frac{\partial \widehat{F}}{\partial t}(s,X_{s-})\dd s+\int_{(t\wedge \tau, t\wedge T_{a,b,\tau}]} \frac{\partial \widehat{F}}{\partial x} (s,X_{s-})\dd X_s+\frac{1}{2}\sigma^2 \int_{t\wedge \tau}^{t\wedge  T_{a,b,\tau}} \frac{\partial^2 \widehat{F}}{\partial x^2} (s, X_{s-}) \dd s\\
&\qquad+\int_{(t\wedge \tau,t\wedge   T_{a,b,\tau}]} \int_{(-\infty,0)}[\widehat{F}(s,X_{s-}+y)-\widehat{F}(s,X_{s-})-y \frac{\partial \widehat{F}}{\partial x} (s,X_{s-}) ]N(\dd s\times \dd y)\\
&= F(t\wedge \tau,X_{t\wedge\tau})+\int_{t\wedge \tau}^{t\wedge  T_{a,b,\tau}} \frac{\partial F}{\partial t}(s,X_{s-})\dd s+\int_{(t\wedge \tau, t\wedge  T_{a,b,\tau}]} \frac{\partial F}{\partial x} (s,X_{s-})\dd X_s+\frac{1}{2}\sigma^2 \int_{t\wedge \tau}^{t\wedge  T_{a,b,\tau}} \frac{\partial^2 F}{\partial x^2} (s, X_{s-}) \dd s\\
&\qquad+\int_{(t\wedge\tau,t\wedge  T_{a,b,\tau}]} \int_{(-\infty,0)}[F(s,X_{s-}+y)-F(s,X_{s-})-y \frac{\partial F}{\partial x} (s,X_{s-}) ]\I_{\{X_{s-}+y> b \}}N(\dd s\times \dd y)\\
&\qquad+\int_{(t\wedge\tau,t\wedge  T_{a,b,\tau}]} \int_{(-\infty,0)}[F(s,b+)-F(s,X_{s-})-y \frac{\partial F}{\partial x} (s,X_{s-}) ]\I_{\{ X_{s-}+y\leq b\}}N(\dd s\times \dd y).
\end{align*}
Therefore, we obtain the desired result by combining the two equations above.  
\end{proof}

\subsubsection{Proof of Theorem \ref{thm:ItoformulaforgttXt}}
\label{sub:proofofItoforumula}
We first assume that $F$ and its derivatives are bounded. It is easy to see that for each $\varepsilon>0$ we have $\sup_{n\geq 0} \rho_{k,\varepsilon}^-=\infty$. Then, for each $t\geq 0$ and $\varepsilon>0$, there exists a value $n\geq 1$ such that $\rho_{n,\varepsilon}^-\leq t< \rho_{n+1,\varepsilon}^-$, so that $M_{t}^{(\varepsilon)}=n$. We consider the case when $X_t>0$ (we omit the proof of the case when $X_t\leq 0$ as it is similar). Then, for any $\varepsilon>0$ sufficiently small, there exists a value $n\geq 1$ (depending on the value of $\varepsilon$) such that  $\rho_{n,\varepsilon}^+ \leq t <\rho_{n+1,\varepsilon}^-$. Note that in this case we have that $X_{t}^{(\varepsilon)}>0$ and $g_{\varepsilon,t}=\rho_{n,\varepsilon}^+$. Using a telescopic sum we obtain that
%
% We consider the case when $\rho_{n,\varepsilon}^+ \leq t <\rho_{n+1,\varepsilon}^-$ (we omit the proof of the case when $\rho_{n,\varepsilon}^-\leq t< \rho_{n,\varepsilon}^+$ as it is similar). Note that in this case we have that $X_{t}^{(\varepsilon)}>0$ and $g_{\varepsilon,t}=\sigma_{n,\varepsilon}^+$. Using a telescopic sum we obtain that
\begin{align*}
F(g_{\varepsilon,t},t,X_t^{(\varepsilon)})
&=F(g_{\varepsilon,0},0,X^{(\varepsilon)}_0+\varepsilon)+\sum_{k=1}^{M_t^{(\varepsilon)}}[ F(g_{\varepsilon,\rho_{k,\varepsilon}^-},\rho_{k,\varepsilon}^-,X^{(\varepsilon)}_{\rho_{k,\varepsilon}^-})-F(g_{\varepsilon,\rho_{k,\varepsilon}^-},\rho_{k,\varepsilon}^-,X^{(\varepsilon)}_{\rho_{k,\varepsilon}^-}+\varepsilon)]\\
&\qquad +\sum_{k=1}^{M_t^{(\varepsilon)}}[ F(g_{\varepsilon,\rho_{k,\varepsilon}^+},\rho_{k,\varepsilon}^+,X^{(\varepsilon)}_{\rho_{k,\varepsilon}^+}-\varepsilon)-F(g_{\varepsilon,\rho_{k,\varepsilon}^-},\rho_{k,\varepsilon}^-,X^{(\varepsilon)}_{\rho_{k,\varepsilon}^-})]\\
&\qquad +\sum_{k=1}^{M_t^{(\varepsilon)}}[ F(g_{\varepsilon,\rho_{k,\varepsilon}^+},\rho_{k,\varepsilon}^+,X^{(\varepsilon)}_{\rho_{k,\varepsilon}^+})-F(g_{\varepsilon,\rho_{k,\varepsilon}^+},\rho_{k,\varepsilon}^+,X^{(\varepsilon)}_{\rho_{k,\varepsilon}^+}-\varepsilon)]\\
&\qquad +\sum_{k=1}^{M_t^{(\varepsilon)}-1}[ F(g_{\varepsilon,\rho_{k+1,\varepsilon}^-},\rho_{k+1,\varepsilon}^-,X^{(\varepsilon)}_{\rho_{k+1,\varepsilon}^-}+\varepsilon)-F(g_{\varepsilon,\rho_{k,\varepsilon}^+},\rho_{k,\varepsilon}^+,X^{(\varepsilon)}_{\rho_{k,\varepsilon}^+})]\\
&\qquad+[F(g_{\varepsilon,t},t,X^{(\varepsilon)}_t)-F(g_{\varepsilon,\rho_{M_{t}^{(\varepsilon)},\varepsilon}^+},\rho_{M_{t}^{(\varepsilon)},\varepsilon}^+,X^{(\varepsilon)}_{\rho_{M_{t}^{(\varepsilon)},\varepsilon}^+})].
\end{align*}
From Remark \ref{rem:behaviourofget} we know that $g_{\varepsilon,\rho_{k,\varepsilon}^-}=\rho_{k,\varepsilon}^- $ and $g_{\varepsilon,\rho_{k,\varepsilon}^+}=\rho_{k,\varepsilon}^+$ for all $k\geq 1$. Thus, from the definition of $X^{(\varepsilon)}$ (see \eqref{eq:definitionofXepsilon}) we see that  
\begin{align}
F(g_{\varepsilon,t},t,X_t^{(\varepsilon)})
&=F(g_{\varepsilon,0},0,X^{(\varepsilon)}_0+\varepsilon)+\sum_{k=1}^{M_t^{(\varepsilon)}}[ F(\rho_{k,\varepsilon}^-,\rho_{k,\varepsilon}^-,X_{\rho_{k,\varepsilon}^-}-\varepsilon)-F(\rho_{k,\varepsilon}^-,\rho_{k,\varepsilon}^-,X_{\rho_{k,\varepsilon}^-})]\nonumber\\
&\qquad +\sum_{k=1}^{M_t^{(\varepsilon)}}[ F(\rho_{k,\varepsilon}^+,\rho_{k,\varepsilon}^+,X_{\rho_{k,\varepsilon}^+}-\varepsilon)-F(\rho_{k,\varepsilon}^-,\rho_{k,\varepsilon}^-,X_{\rho_{k,\varepsilon}^-}-\varepsilon)]\nonumber\\
&\qquad +\sum_{k=1}^{M_t^{(\varepsilon)}}[ F(\rho_{k,\varepsilon}^+,\rho_{k,\varepsilon}^+,X_{\rho_{k,\varepsilon}^+})-F(\rho_{k,\varepsilon}^+,\rho_{k,\varepsilon}^+,X_{\rho_{k,\varepsilon}^+}-\varepsilon)]\nonumber\\
&\qquad +\sum_{k=1}^{M_t^{(\varepsilon)}-1}[ F(\rho_{k+1,\varepsilon}^-,\rho_{k+1,\varepsilon}^-,X_{\rho_{k+1,\varepsilon}^-})-F(\rho_{k,\varepsilon}^+,\rho_{k,\varepsilon}^+,X_{\rho_{k,\varepsilon}^+})]\nonumber\\
\label{eq:telescopicsumforFgttXt}
&\qquad+[F(\rho^+_{M_{t}^{(\varepsilon)},\varepsilon},t,X_t)-F(\rho_{M_{t}^{(\varepsilon)},\varepsilon}^+,\rho_{M_{t}^{(\varepsilon)},\varepsilon}^+,X_{\rho_{M_{t}^{(\varepsilon)},\varepsilon}^+})].
\end{align}
We first show that the sum of the terms on the second, fourth and last line on the right-hand side of the equation above converge in probability to the terms on the right-hand side of \eqref{eq:ItoformulaforgttXt}. 
\begin{lemma}
\label{lemma:convergenceofItoformulagttXt}
Under the assumptions of Theorem \ref{thm:ItoformulaforgttXt} and further assuming that $F$ and its derivatives are bounded, we have that  
\begin{align}
\lim_{\varepsilon \downarrow 0}&\bigg[\sum_{k=1}^{M_t^{(\varepsilon)}}[ F(\rho_{k,\varepsilon}^+,\rho_{k,\varepsilon}^+,X_{\rho_{k,\varepsilon}^+}-\varepsilon)-F(\rho_{k,\varepsilon}^-,\rho_{k,\varepsilon}^-,X_{\rho_{k,\varepsilon}^-}-\varepsilon)]\nonumber\\
&\qquad+\sum_{k=1}^{M_t^{(\varepsilon)}-1}[ F(\rho_{k+1,\varepsilon}^-,\rho_{k+1,\varepsilon}^-,X_{\rho_{k+1,\varepsilon}^-})-F(\rho_{k,\varepsilon}^+,\rho_{k,\varepsilon}^+,X_{\rho_{k,\varepsilon}^+})]\nonumber\\
&\qquad+[F(\rho^+_{M_{t}^{(\varepsilon)},\varepsilon},t,X_t)-F(\rho_{M_{t}^{(\varepsilon)},\varepsilon}^+,\rho_{M_{t}^{(\varepsilon)},\varepsilon}^+,X_{\rho_{M_{t}^{(\varepsilon)},\varepsilon}^+})]\bigg]\nonumber\\
&= \int_{0}^{t}   \frac{\partial F_g }{\partial t } (s,X_{s-})\I_{\{g_{s-}=s \}}\dd s+ \int_{0}^{t} \frac{\partial F }{\partial t} (g_{s-},s,X_{s-})\I_{\{ g_{s-}<s\}} \dd s+\int_{0}^{t}\frac{\partial F }{\partial x} (g_{s-},s,X_{s-})\dd X_s \nonumber\\
&\qquad+\frac{1}{2} \sigma^2 \int_{0}^{t} \frac{\partial^2  F}{\partial x^2} (g_{s-},s,X_{s-}) \dd s\nonumber\\
& \qquad +\int_{[0,t] } \int_{(-\infty,0)} \left[ F(s,s,X_{s-}+y) -F(s,s,X_{s-})-y\frac{\partial F }{\partial x} (s,s,X_{s-}) \right]\I_{\{g_{s-}=s \}}N(\dd s\times \dd y)\nonumber\\
&\qquad+\int_{[0,t]}\int_{(-\infty,0)} \left[F(g_{s-},s,X_{s-}+y)-F(g_{s-},s,X_{s-})-y\frac{\partial F }{\partial x} (g_{s-},s,X_{s-})\right]\nonumber\\
&\qquad\qquad\qquad \times \I_{\{X_{s-}+y >0 \}}\I_{\{ g_{s-}<s\}} N(\dd s\times \dd y) \nonumber\\
&\qquad+ \int_{[0,t]}\int_{(-\infty,0)} \left[F(s,s,X_{s-}+y)-F(g_{s-},s,X_{s-})-y\frac{\partial F }{\partial x} (g_{s-},s,X_{s-}) \right]\nonumber\\
\label{eq:Itoconvergenceinprobabilityofstochintegrals}
&\qquad\qquad\qquad \times \I_{\{X_{s-}+y \leq 0 \}} \I_{\{ g_{s-}<s\}}N(\dd s\times \dd y)
\end{align}
in probability.
\end{lemma}
 \begin{proof}
We recall that $E_g=\{ (\gamma,t,x)  : 0\leq \gamma < t \text{ and } x>0\}\cup \{ (\gamma,t,x) : 0 \leq \gamma=t \text{ and } x\leq 0\}$. We define the auxiliary function $\widetilde{F}:\R_2^+\times \R\mapsto \R$ via 
\begin{align*}
\widetilde{F}(\gamma,t,x)=\left\{\begin{array}{ll}
F(\gamma,t,x), & (\gamma,t,x)\in \bar{E}_g,\\
F(t,t,x), & (\gamma,t,x)\in (\R_2^+\times \R) \setminus \bar{E}_g.
\end{array}  \right.
\end{align*}
Note that in particular we have that $\widetilde{F}(\gamma,t,x)=F(t,t,x)$ for any $\gamma,t\geq 0$ and $x\leq 0$. Hence, since $X_{\rho_{k+1,\varepsilon}^-}\leq 0$ for all $k\leq M_t^{(\varepsilon)}-1$, we can write 
\begin{align*}
\sum_{k=1}^{M_t^{(\varepsilon)}-1}&[ F(\rho_{k+1,\varepsilon}^-,\rho_{k+1,\varepsilon}^-,X_{\rho_{k+1,\varepsilon}^-})-F(\rho_{k,\varepsilon}^+,\rho_{k,\varepsilon}^+,X_{\rho_{k,\varepsilon}^+})]\\
&\qquad+[F(\rho^+_{M_{t}^{(\varepsilon)},\varepsilon},t,X_t)-F(\rho_{M_{t}^{(\varepsilon)},\varepsilon}^+,\rho_{M_{t}^{(\varepsilon)},\varepsilon}^+,X_{\rho_{M_{t}^{(\varepsilon)},\varepsilon}^+})]\\
&=\sum_{k=1}^{M_t^{(\varepsilon)}}[ \widetilde{F}(\rho_{k,\varepsilon}^+,\rho_{k+1,\varepsilon}^- \wedge t,X_{\rho_{k+1,\varepsilon}^- \wedge t})-\widetilde{F}(\rho_{k,\varepsilon}^+,\rho_{k,\varepsilon}^+,X_{\rho_{k,\varepsilon}^+})].
\end{align*}
From the assumptions on $F$, we see that for fixed $\gamma\geq 0$, the mapping $(t,x)\mapsto \widetilde{F}(\gamma,t,x)$ is $C^{1,i}$ on $[\gamma,\infty)\times [0,\infty)$. Moreover, if $\sigma>0$, since $\lim_{h\downarrow 0} F(\gamma,t,h)=F(t,t,0)$ for all $0\leq \gamma \leq t$, we see that $ \lim_{h\downarrow 0} \widetilde{F}(\gamma,t,h)=\widetilde{F}(\gamma,t,0)$ for all $\gamma\geq 0$ and $t\geq 0$. Applying It\^o formula (see Proposition \ref{prop:stoppedItoformula}) to  $s\mapsto \widetilde{F}(\rho_{k,\varepsilon}^+,s , X_{s}) $ on time intervals of the form $(\rho_{k,\varepsilon}^+ \wedge t,\rho_{k+1,\varepsilon}^-\wedge t]$, for each $k\leq M_t^{(\varepsilon)}$, we deduce that
%upon noticing that for each $k\geq 1$, due to the strong Markov property for L\'evy processes, the stochastic process $\widetilde{X}_{t}=X_{t+\rho_{k,\varepsilon}^{\varepsilon}}$ is a L\'evy process with the same L\'evy triplet as $X$ starting from $\varepsilon>0$, so that $\rho_{k+1,\varepsilon}^{\varepsilon}$ corresponds to the first passage time to $(-\infty,0)$
\begin{align*}
&\sum_{k=1}^{M_t^{(\varepsilon)}}[ \widetilde{F}(\rho_{k,\varepsilon}^+,\rho_{k+1,\varepsilon}^- \wedge t,X_{\rho_{k+1,\varepsilon}^- \wedge t})-\widetilde{F}(\rho_{k,\varepsilon}^+,\rho_{k,\varepsilon}^+,X_{\rho_{k,\varepsilon}^+})]\\
&=\sum_{k=1}^{M_t^{(\varepsilon)}} \left[ \int_{\rho_{k,\varepsilon}^+}^{\rho_{k+1,\varepsilon}^- \wedge t} \frac{\partial \widetilde{F} }{\partial t} (\rho_{k,\varepsilon}^+,s,X_{s-}) \dd s +\int_{(\rho_{k,\varepsilon}^+,\rho_{k+1,\varepsilon}^- \wedge t]} \frac{\partial \widetilde{F} }{\partial x} (\rho_{k,\varepsilon}^+,s,X_{s-}) \dd X_s \right]\\
&\qquad +\sum_{k=1}^{M_t^{(\varepsilon)}}  \frac{1}{2}\sigma^2\int_{\rho_{k,\varepsilon}^+}^{\rho_{k+1,\varepsilon}^- \wedge t} \frac{\partial^2 \widetilde{F} }{\partial x^2} (\rho_{k,\varepsilon}^+,s,X_{s-}) \dd s \\
&\qquad+\sum_{k=1}^{M_t^{(\varepsilon)}} \int_{(\rho_{k,\varepsilon}^+,\rho_{k+1,\varepsilon}^- \wedge t]} \int_{(-\infty,0)} \left[\widetilde{F}(\rho_{k,\varepsilon}^+,s,X_{s-}+y)-\widetilde{F}(\rho_{k,\varepsilon}^+,s,X_{s-}) -y\frac{\partial \widetilde{F}}{\partial x} ( \rho_{k,\varepsilon}^+,s,X_{s-}) \right]\\
&\qquad\qquad\qquad\times N(\dd s\times \dd y).
\end{align*}
From the definition of $\widetilde{F}$ and since $g_{\varepsilon,s-}<s$ if and only if $s\in (\rho_{k,\varepsilon}^+, \rho_{k+1,\varepsilon}^-]$ for some $k\geq 1$ (in this case $g_{\varepsilon,s-}=\rho_{k,\varepsilon}^+$ and $X_{s-}^{(\varepsilon)}=X_{s-}$), we have that
\begin{align*}
&\sum_{k=1}^{M_t^{(\varepsilon)}}[ \widetilde{F}(\rho_{k,\varepsilon}^+,\rho_{k+1,\varepsilon}^- \wedge t,X_{\rho_{k+1,\varepsilon}^- \wedge t})-\widetilde{F}(\rho_{k,\varepsilon}^+,\rho_{k,\varepsilon}^+,X_{\rho_{k,\varepsilon}^+})]\\
&= \int_{0}^{t} \frac{\partial F }{\partial t} (g_{\varepsilon,s-},s,X_{s-}^{(\varepsilon)})\I_{\{ g_{\varepsilon,s-}<s\}} \dd s +\int_{0}^{t} \frac{\partial F }{\partial x} (g_{\varepsilon,s-},s,X_{s-}^{(\varepsilon)})\I_{\{ g_{\varepsilon,s-}<s \}} \dd X_s \\
&\qquad + \frac{1}{2}\sigma^2\int_{0}^{t} \frac{\partial^2 F }{\partial x^2} (g_{\varepsilon,s-},s,X_{s-}^{(\varepsilon)})\I_{\{ g_{\varepsilon,s-}<s \}} \dd s \\
&\qquad+\int_{[0,t]}\int_{(-\infty,0)} \left[F(g_{\varepsilon,s-},s,X_{s-}^{(\varepsilon)}+y)-F(g_{\varepsilon,s-},s,X_{s-}^{(\varepsilon)})-y\frac{\partial F }{\partial x} (g_{\varepsilon,s-},s,X_{s-}^{(\varepsilon)})\right]\\
&\qquad\qquad\qquad \times \I_{\{X_{s-}^{(\varepsilon)}+y >0 \}}\I_{\{ g_{\varepsilon,s-}<s\}} N(\dd s\times \dd y) \\
&\qquad+ \int_{[0,t]}\int_{(-\infty,0)} \left[F(s,s,X_{s-}^{(\varepsilon)}+y)-F(g_{\varepsilon,s-},s,X_{s-}^{(\varepsilon)})-y\frac{\partial F }{\partial x} (g_{\varepsilon,s-},s,X_{s-}^{(\varepsilon)}) \right]\\
&\qquad\qquad\qquad \times \I_{\{X_{s-}^{(\varepsilon)}+y \leq 0 \}} \I_{\{ g_{\varepsilon,s-}<s\}}N(\dd s\times \dd y).
\end{align*}
Similarly, since $(t,x)\mapsto F(t,t,x)$ is $C^{1,i}$ on $[0,\infty)\times (-\infty,0]$ and $F$ is continuous on $E_g$, we can apply Proposition \ref{prop:stoppedItoformula} on intervals of the form $(\rho_{k,\varepsilon}^-,\rho_{k,\varepsilon}^+\wedge t]=(\rho_{k,\varepsilon}^-,\rho_{k,\varepsilon}^+]$, upon noticing that $X_s-\varepsilon\leq 0$ for all $s\in (\rho_{k,\varepsilon}^-,\rho_{k,\varepsilon}^+]$, to see that 

\begin{align*}
&\sum_{k=1}^{M_t^{(\varepsilon)}}[ F(\rho_{k,\varepsilon}^+,\rho_{k,\varepsilon}^+,X_{\rho_{k,\varepsilon}^+}-\varepsilon)-F(\rho_{k,\varepsilon}^-,\rho_{k,\varepsilon}^-,X_{\rho_{k,\varepsilon}^-}-\varepsilon)]\\
&=\sum_{k=1}^{M_t^{(\varepsilon)}}\left[ \int_{\rho_{k,\varepsilon}^-}^{\rho_{k,\varepsilon}^+ }   \frac{\partial F_g }{\partial t } (s,X_{s-}-\varepsilon)\dd s+\int_{(\rho_{k,\varepsilon}^-,\rho_{k,\varepsilon}^+] }\frac{\partial F }{\partial x} (s,s,X_{s-}-\varepsilon)\dd X_s +\frac{1}{2} \sigma^2 \int_{\rho_{k,\varepsilon}^-}^{\rho_{k,\varepsilon}^+ } \frac{\partial^2  F}{\partial x^2} (s,s,X_{s-}-\varepsilon)\dd s\right]\\
& \qquad +\sum_{k=1}^{M_t^{(\varepsilon)}}\int_{(\rho_{k,\varepsilon}^-,\rho_{k,\varepsilon}^+] } \int_{(-\infty,0)} \left[ F(s,s,X_{s-}+y-\varepsilon) -F(s,s,X_{s-}-\varepsilon)-y\frac{\partial F }{\partial x} (s,s,X_{s-}-\varepsilon) \right]N(\dd s\times \dd y)\\
&= \int_{0}^{t}   \frac{\partial F_g }{\partial t } (s,X_{s-}^{(\varepsilon)})\I_{\{g_{\varepsilon,s-}=s \}}\dd s+\int_{0}^{t}\frac{\partial F }{\partial x} (s,s,X_{s-}^{(\varepsilon)})\I_{\{g_{\varepsilon,s-}=s \}}\dd X_s \\
&\qquad+\frac{1}{2} \sigma^2 \int_{0}^{t} \frac{\partial^2  F}{\partial x^2} (s,s,X_{s-}^{(\varepsilon)}) \I_{\{g_{\varepsilon,s-}=s \}}\dd s\\
& \qquad +\int_{[0,t] } \int_{(-\infty,0)} \left[ F(s,s,X_{s-}^{(\varepsilon)}+y) -F(s,s,X_{s-}^{(\varepsilon)})-y\frac{\partial F }{\partial x} (s,s,X_{s-}^{(\varepsilon)}) \right] \I_{\{g_{\varepsilon,s-}=s \}}N(\dd s\times \dd y),
\end{align*}
where we recall that $F_g(t,x)=F(t,t,x)$ for $t\geq 0$ and $x\leq 0$. Therefore, we conclude that 
\begin{align*}
&\sum_{k=1}^{M_t^{(\varepsilon)}}[ F(\rho_{k,\varepsilon}^+,\rho_{k,\varepsilon}^+,X_{\rho_{k,\varepsilon}^+}-\varepsilon)-F(\rho_{k,\varepsilon}^-,\rho_{k,\varepsilon}^-,X_{\rho_{k,\varepsilon}^-}-\varepsilon)]\\
&\qquad+\sum_{k=1}^{M_t^{(\varepsilon)}-1}[ F(\rho_{k+1,\varepsilon}^-,\rho_{k+1,\varepsilon}^-,X_{\rho_{k+1,\varepsilon}^-})-F(\rho_{k,\varepsilon}^+,\rho_{k,\varepsilon}^+,X_{\rho_{k,\varepsilon}^+})]\\
&\qquad+[F(\rho_{M_{t}^{(\varepsilon)}},t,X_t)-F(\rho_{M_{t}^{(\varepsilon)},\varepsilon}^+,\rho_{M_{t}^{(\varepsilon)},\varepsilon}^+,X_{\rho_{M_{t}^{(\varepsilon)},\varepsilon}^+})]\\
&= \int_{0}^{t}   \frac{\partial F_g }{\partial t } (s,X_{s-}^{(\varepsilon)})\I_{\{g_{\varepsilon,s-}=s \}}\dd s+ \int_{0}^{t} \frac{\partial F }{\partial t} (g_{\varepsilon,s-},s,X_{s-}^{(\varepsilon)})\I_{\{ g_{\varepsilon,s-}<s\}} \dd s+\int_{0}^{t}\frac{\partial F }{\partial x} (g_{\varepsilon,s-},s,X_{s-}^{(\varepsilon)})\dd X_s \\
&\qquad+\frac{1}{2} \sigma^2 \int_{0}^{t} \frac{\partial^2  F}{\partial x^2} (g_{\varepsilon,s-},s,X_{s-}^{(\varepsilon)}) \dd s\\
& \qquad +\int_{[0,t] } \int_{(-\infty,0)} \left[ F(s,s,X_{s-}^{(\varepsilon)}+y) -F(s,s,X_{s-}^{(\varepsilon)})-y\frac{\partial F }{\partial x} (s,s,X_{s-}^{(\varepsilon)}) \right]\I_{\{g_{\varepsilon,s-}=s \}}N(\dd s\times \dd y)\\
&\qquad+\int_{[0,t]}\int_{(-\infty,0)} \left[F(g_{\varepsilon,s-},s,X_{s-}^{(\varepsilon)}+y)-F(g_{\varepsilon,s-},s,X_{s-}^{(\varepsilon)})-y\frac{\partial F }{\partial x} (g_{\varepsilon,s-},s,X_{s-}^{(\varepsilon)})\right]\\
&\qquad\qquad\qquad \times \I_{\{X_{s-}^{(\varepsilon)}+y >0 \}}\I_{\{ g_{\varepsilon,s-}<s\}} N(\dd s\times \dd y) \\
&\qquad+ \int_{[0,t]}\int_{(-\infty,0)} \left[F(s,s,X_{s-}^{(\varepsilon)}+y)-F(g_{\varepsilon,s-},s,X_{s-}^{(\varepsilon)})-y\frac{\partial F }{\partial x} (g_{\varepsilon,s-},s,X_{s-}^{(\varepsilon)}) \right]\\
&\qquad\qquad\qquad \times \I_{\{X_{s-}^{(\varepsilon)}+y \leq 0 \}} \I_{\{ g_{\varepsilon,s-}<s\}}N(\dd s\times \dd y).
\end{align*}
Hence, since we are assuming that $F$ and its derivatives are continuous and bounded, by using the mean value theorem when $X$ is of finite variation, Taylor's theorem when $X$ is of infinite variation, and the dominated convergence theorem for stochastic integrals (see, for example, Theorem IV.15 and Theorem IV.32 in \cite{protter2005} on pp. 166,176, respectively) we conclude that \eqref{eq:Itoconvergenceinprobabilityofstochintegrals} indeed holds.
\end{proof}
Next, to complete the proof of Theorem \ref{thm:ItoformulaforgttXt}, we show that the existence of a sequence such that the remaining terms on the right-hand side of \eqref{eq:telescopicsumforFgttXt} converge to zero almost surely.  
\begin{lemma}
\label{lemma:convergenceofdowncrossingslocaltime}
Under the assumptions of Theorem \ref{thm:ItoformulaforgttXt} and further assuming that $F$ and its derivatives are bounded and that $\Pi(-\infty,0)<\infty$, we have that there exists a sequence $\{ \varepsilon_n, n\geq 0\}$ converging to zero such that 
\begin{align*}
\lim_{n\rightarrow\infty}\bigg[\sum_{k=1}^{M_t^{(\varepsilon_n)}}&[ F(\rho_{k,\varepsilon_n}^-,\rho_{k,\varepsilon_n}^-,X_{\rho_{k,\varepsilon_n}^-}-\varepsilon_n)-F(\rho_{k,\varepsilon_n}^-,\rho_{k,\varepsilon_n}^-,X_{\rho_{k,\varepsilon_n}^-})]\\
&\qquad +\sum_{k=1}^{M_t^{(\varepsilon_n)}}[ F(\rho_{k,\varepsilon_n}^+,\rho_{k,\varepsilon_n}^+,X_{\rho_{k,\varepsilon_n}^+})-F(\rho_{k,\varepsilon_n}^+,\rho_{k,\varepsilon_n}^+,X_{\rho_{k,\varepsilon_n}^+}-\varepsilon_n)]\bigg] =0
\end{align*}
almost surely.
\end{lemma}

\begin{proof}
For any $s>0$ and $\varepsilon>0$, we define the stopping time
\begin{align*}
\tau_{\varepsilon,s}^+:=\inf\{ t\geq s: X_t\geq \varepsilon\}.
\end{align*}
By using that $\rho_{1,\varepsilon}^-=0$ (by definition), that $N$ is the Poisson random measure of the jumps of $X$ and by rearranging some of the terms, we see that
\begin{align*}
\sum_{k=1}^{M_t^{(\varepsilon)}}&[ F(\rho_{k,\varepsilon}^-,\rho_{k,\varepsilon}^-,X_{\rho_{k,\varepsilon}^-}-\varepsilon)-F(\rho_{k,\varepsilon}^-,\rho_{k,\varepsilon}^-,X_{\rho_{k,\varepsilon}^-})] +\sum_{k=1}^{M_t^{(\varepsilon)}}[ F(\rho_{k,\varepsilon}^+,\rho_{k,\varepsilon}^+,X_{\rho_{k,\varepsilon}^+})-F(\rho_{k,\varepsilon}^+,\rho_{k,\varepsilon}^+,X_{\rho_{k,\varepsilon}^+}-\varepsilon)] \\
&= F(0,0,X_0-\varepsilon)-F(0,0,X_0)+F(\rho_{1,\varepsilon}^+,\rho_{1,\varepsilon}^+,\varepsilon)-F(\rho_{1,\varepsilon}^+,\rho_{1,\varepsilon}^+,0) \\
&\qquad+\sum_{k=2}^{M_t^{(\varepsilon)}}[ F(\rho_{k,\varepsilon}^-,\rho_{k,\varepsilon}^-,X_{\rho_{k,\varepsilon}^-}-\varepsilon)-F(\rho_{k,\varepsilon}^-,\rho_{k,\varepsilon}^-,X_{\rho_{k,\varepsilon}^-})] \I_{\{ X_{\rho_{k,\varepsilon}^-}<0\}}\\
&\qquad +\sum_{k=2}^{M_t^{(\varepsilon)}}[ F(\rho_{k,\varepsilon}^-,\rho_{k,\varepsilon}^-,-\varepsilon)-F(\rho_{k,\varepsilon}^-,\rho_{k,\varepsilon}^-,0)] \I_{\{ X_{\rho_{k,\varepsilon}^-}=0\}}\\
&\qquad +\sum_{k=2}^{M_t^{(\varepsilon)}}[ F(\rho_{k,\varepsilon}^+,\rho_{k,\varepsilon}^+,\varepsilon)-F(\rho_{k,\varepsilon}^+,\rho_{k,\varepsilon}^+,0)] \I_{\{ X_{\rho_{k,\varepsilon}^-}<0\}}+[ F(\rho_{k,\varepsilon}^+,\rho_{k,\varepsilon}^+,\varepsilon)-F(\rho_{k,\varepsilon}^+,\rho_{k,\varepsilon}^+,0)]\I_{\{ X_{\rho_{k,\varepsilon}^-}=0\}}\\
&=F(0,0,X_0-\varepsilon)-F(0,0,X_0)+F(\rho_{1,\varepsilon}^+,\rho_{1,\varepsilon}^+,\varepsilon)-F(\rho_{1,\varepsilon}^+,\rho_{1,\varepsilon}^+,0)\\
&\qquad +\int_{[0,t]} \int_{(-\infty,0)}[ F(s,s,X_{s-}+y-\varepsilon)-F(s,s,X_{s-}+y)] \I_{\{ X^{(\varepsilon)}_{s-}+y<0\}} \I_{\{ X_{s-}^{(\varepsilon)}>0\}} N(\dd s \times \dd y)\\
&\qquad +\int_{[0,t]} \int_{(-\infty,0)}[ F(\tau_{\varepsilon,s}^+,\tau_{\varepsilon,s}^+,\varepsilon)-F(\tau_{\varepsilon,s}^+,\tau_{\varepsilon,s}^+,0)] \I_{\{ X^{(\varepsilon)}_{s-}+y<0\}} \I_{\{ X_{s-}^{(\varepsilon)}>0\}} N(\dd s \times \dd y)\\
&\qquad +\int_0^t[ F(s,s,-\varepsilon)-F(s,s,0)+  F(\tau_{\varepsilon,s}^+,\tau_{\varepsilon,s}^+,\varepsilon)-F(\tau_{\varepsilon,s}^+,\tau_{\varepsilon,s}^+,0)] \I_{\{ X_{s}=0\}}\dd M_s^{(\varepsilon)}.
\end{align*}
We then show that the integrals with respect to the Poisson random measure $N$ converge to zero. Since $X_s^{(\varepsilon)}\leq X_{s}$ for all $s>0$, and on the event $\{ X_{s-}^{(\varepsilon)}>0 \}$, we have that $X_{s-}^{(\varepsilon)}=X_{s-}$ we see that
\begin{align*}
&\left|\int_{[0,t]} \int_{(-\infty,0)}[ F(s,s,X_{s-}+y-\varepsilon)-F(s,s,X_{s-}+y)] \I_{\{ X^{(\varepsilon)}_{s-}+y<0\}} \I_{\{ X_{s-}^{(\varepsilon)}>0\}} N(\dd s \times \dd y)\right|\\
&\qquad\leq  \int_{[0,t]} \int_{(-\infty,0)}\left| F(s,s,X_{s-}+y-\varepsilon)-F(s,s,X_{s-}+y)\right|\I_{\{ X^{(\varepsilon)}_{s-}+y<0\}} \I_{\{ X_{s-}^{(\varepsilon)}>0\}} N(\dd s \times \dd y)\\
&\qquad\leq  2M\int_{[0,t]} \int_{(-\infty,0)}\I_{\{ X_{s-}+y<0\}} \I_{\{ X_{s-}>0\}} N(\dd s \times \dd y)\\
&\qquad<\infty,
\end{align*}
where we used that $F$ is bounded, say, by a constant $M>0$ and that $\Pi(-\infty,0)<\infty$ so the last integral is finite. Hence, by the dominated convergence theorem and the continuity of $F$, we have that 
\begin{align*}
\lim_{\varepsilon\downarrow 0}\int_{[0,t]} \int_{(-\infty,0)}[ F(s,s,X_{s-}+y-\varepsilon)-F(s,s,X_{s-}+y)] \I_{\{ X^{(\varepsilon)}_{s-}+y<0\}} \I_{\{ X_{s-}^{(\varepsilon)}>0\}} N(\dd s \times \dd y)=0
\end{align*}
almost surely. Similarly, we can see that 
\begin{align*}
\lim_{\varepsilon \downarrow 0}\int_{[0,t]} \int_{(-\infty,0)}[ F(\tau_{\varepsilon,s}^+,\tau_{\varepsilon,s}^+,\varepsilon)-F(\tau_{\varepsilon,s}^+,\tau_{\varepsilon,s}^+,0)] \I_{\{ X^{(\varepsilon)}_{s-}+y<0\}} \I_{\{ X_{s-}^{(\varepsilon)}>0\}} N(\dd s \times \dd y)=0
\end{align*}
almost surely. On the other hand, since we are assuming that $x\mapsto F(t,t,x)$ is differentiable on $[0,\infty)$ and on $(-\infty,0]$ we have that 
\begin{align*}
&\left|\int_0^t[ F(s,s,-\varepsilon)-F(s,s,0)+  F(\tau_{\varepsilon,s}^+,\tau_{\varepsilon,s}^+,\varepsilon)-F(\tau_{\varepsilon,s}^+,\tau_{\varepsilon,s}^+,0)] \I_{\{ X_{s}=0\}}\dd M_s^{(\varepsilon)} \right|\\
%&\qquad=\left| \int_0^t\int_0^{\varepsilon} \left[\frac{\partial F}{\partial x} (\tau_{\varepsilon,s}^+,\tau_{\varepsilon,s}^+,y)-\frac{\partial F}{\partial x}(s,s,-y) \right]\I_{\{ X_{s}=0\}} \dd y  \dd M_s^{(\varepsilon)} \right| \\
&\qquad\leq \int_0^t\int_0^{\varepsilon} \left|\frac{\partial F}{\partial x} (\tau_{\varepsilon,s}^+,\tau_{\varepsilon,s}^+,y)-\frac{\partial F}{\partial x}(s,s,-y) \right|\I_{\{ X_{s}=0\}} \dd y  \dd M_s^{(\varepsilon)} \\
& \qquad \leq \varepsilon ( M_{t}^{(\varepsilon)}-1) \sup_{(s,y)\in [0,t]\times [0,\varepsilon]}\left\{ \left|\frac{\partial F}{\partial x} (\tau_{\varepsilon,s}^+,\tau_{\varepsilon,s}^+,y)-\frac{\partial F}{\partial x}(s,s,-y) \right|\I_{\{ X_{s}=0\}} \right\}.
\end{align*}
Note that $X$ creeps downwards only when $\sigma>0$, so then, the first term above vanishes when $\sigma=0$. Henceforth, we now assume that $\sigma>0$. From Lemma \ref{lemma:Mtlocaltimeconvergence} we know that $\varepsilon M_t^{(\varepsilon)}$ converges to $L_t/2$ in probability when $\varepsilon\downarrow 0$. Hence, there exists a subsequence $\{\varepsilon_n,n\geq 0 \}$ convergent to zero such that $\varepsilon_n M_t^{(\varepsilon_n)}$ converges to $L_t/2$, almost surely. Thus, we see that 
\begin{align*}
\lim_{n\rightarrow \infty} &\left|\int_0^t[ F(s,s,-\varepsilon_n)-F(s,s,0)+  F(\tau_{\varepsilon_n,s}^+,\tau_{\varepsilon_n,s}^+,\varepsilon_n)-F(\tau_{\varepsilon_n,s}^+,\tau_{\varepsilon_n,s}^+,0)] \I_{\{ X_{s}=0\}}\dd M_s^{(\varepsilon)} \right|\\
& \leq \lim_{n\rightarrow \infty} \varepsilon_n ( M_{t}^{(\varepsilon_n)}-1) \sup_{(s,y)\in [0,t]\times [0,\varepsilon_n]}\left\{ \left|\frac{\partial F}{\partial x} (\tau_{\varepsilon_n,s}^+,\tau_{\varepsilon_n,s}^+,y)-\frac{\partial F}{\partial x}(s,s,-y) \right|\I_{\{ X_{s}=0\}}\right\}\\
&=0
\end{align*}
almost surely, where we used that for each $s>0$, on the event $\{X_s=0 \}$, we have $\lim_{\varepsilon \downarrow 0} \tau_{\varepsilon,s}^+=s$ almost surely, and that $(s,y)\mapsto \frac{\partial F}{\partial x}(s,s,y)$ is continuous in $[0,\infty)\times \R$ when $\sigma>0$. Indeed, assumptions $i)$ and $ii)$ imply that $F(t,t,x)=F_1(t,x)$ for every $t\geq 0$ and $x> 0$ and $F(t,t,x)=F_2(t,x)$ for $t\geq 0$ and $x<0$, where $F_1$ and $F_2$ are some $C^{1,1}$ functions on $[0,\infty)\times \R$. Moreover, from the above facts and the mean value theorem we have that for any $h>0$, there exists a value $c_h\in (0,h)$ such that
\begin{align*}
\lim_{h\downarrow 0}\frac{F(t,t,h)-F(t,t,0)}{h}=\lim_{h\downarrow 0}\frac{\partial F}{\partial x}(t,t,c_h)=\frac{\partial F}{\partial x} F(t,t,0+)
\end{align*}
for any $t\geq 0$. Similarly, 
\begin{align*}
\lim_{h\downarrow 0}\frac{F(t,t,-h)-F(t,t,0)}{-h}=\frac{\partial F}{\partial x} F(t,t,0-).
\end{align*}
Thus, from assumption \eqref{eq:pastingatzeroItoformula},  we have that $\frac{\partial F}{\partial x}(t,t,0)$ exists for all $t\geq 0$. It follows that
\begin{align*}
\frac{\partial F_1}{\partial x}(t,0)=\frac{\partial F}{\partial x} (t,t,0+) =\frac{\partial F}{\partial x} (t,t,0)=\frac{\partial F}{\partial x} (t,t,0-)=\frac{\partial F_2}{\partial x}(t,0)
\end{align*}
for all $t\geq 0$. Then, we can write 
\begin{align*}
\frac{\partial F}{\partial x}F(t,t,x)&=\I_{\{x\geq 0 \}}\frac{\partial F_1}{\partial x}(t,x)+\I_{\{x< 0 \}}\frac{\partial F_2}{\partial x}(t,x)\\
&=\I_{\{x> 0 \}}\frac{\partial F_1}{\partial x}(t,x)+\I_{\{x\leq 0 \}}\frac{\partial F_2}{\partial x}(t,x)
\end{align*}
for any $t\geq 0$ and $x\in \R$. Therefore, from the equation above it can be easily seen that $(s,y)\mapsto \frac{\partial F}{\partial x}(s,s,y)$ is continuous in $[0,\infty)\times \R$ as claimed. The conclusion holds. 
\end{proof}

Hence, when $\Pi(-\infty,0)<\infty$, we deduce that there exists a subsequence $\{\varepsilon'_{n}, n\geq 0 \}$, convergent to zero such that the limits in Lemmas \ref{lemma:convergenceofItoformulagttXt}-\ref{lemma:convergenceofdowncrossingslocaltime} hold almost surely. Thus, following a similar argument as in \cite{applebaum_2009} (see proof of Theorem 4.4.7 on p. 226), from \eqref{eq:telescopicsumforFgttXt} along the subsequence $\{\varepsilon'_{n}, n\geq 0 \}$ and taking $n\rightarrow \infty$ we deduce, when $F$ and its derivatives are bounded, that
\begin{align*}
F(g_{t}, t,X_t)
&=F(g_0,0,X_0)+ \int_{0}^{t}   \frac{\partial F_g }{\partial t } (s,X_{s-})\I_{\{g_{s-}=s \}}\dd s+ \int_{0}^{t} \frac{\partial F }{\partial t} (g_{s-},s,X_{s-})\I_{\{ g_{s-}<s\}} \dd s \\
&\qquad+\int_{0}^{t}\frac{\partial F }{\partial x} (g_{s-},s,X_{s-})\dd X_s+\frac{1}{2} \sigma^2 \int_{0}^{t} \frac{\partial^2  F}{\partial x^2} (g_{s-},s,X_{s-}) \dd s\\
& \qquad +\int_{[0,t] } \int_{(-\infty,0)} \left[ F(s,s,X_{s-}+y) -F(s,s,X_{s-})-y\frac{\partial F }{\partial x} (s,s,X_{s-}) \right]\I_{\{g_{s-}=s \}}N(\dd s\times \dd y)\\
&\qquad+\int_{[0,t]}\int_{(-\infty,0)} \left[F(g_{s-},s,X_{s-}+y)-F(g_{s-},s,X_{s-})-y\frac{\partial F }{\partial x} (g_{s-},s,X_{s-})\right]\\
&\qquad\qquad\qquad \times \I_{\{X_{s-}+y >0 \}}\I_{\{ g_{s-}<s\}} N(\dd s\times \dd y) \\
&\qquad+ \int_{[0,t]}\int_{(-\infty,0)} \left[F(s,s,X_{s-}+y)-F(g_{s-},s,X_{s-})-y\frac{\partial F }{\partial x} (g_{s-},s,X_{s-}) \right]\\
&\qquad\qquad\qquad \times \I_{\{X_{s-}+y \leq 0 \}} \I_{\{ g_{s-}<s\}}N(\dd s\times \dd y),
\end{align*}
with probability one. For the case when $\Pi$ is any L\'evy measure we know from Corollary 4.3.10 in \cite{applebaum_2009} that there exists a sequence $\{ A_n,n\geq 1\}$ with $\Pi(A_n)<\infty$, for each $n\geq 1$, and $A_n \uparrow (-1,0)$
when $n\rightarrow \infty$, such that $\lim_{n\rightarrow } X_{t}^{(n)}=X_t$ uniformly on compacts in probability, where
\begin{align*}
X_t^{(n)}=\sigma B_t-\mu t+\int_{[0,t]}\int_{(-\infty,-1)} x N(\dd s\times \dd x)+\int_{[0,t]}\int_{(-1,0)\cap A_n} x( N(\dd s\times \dd x)-\dd s\Pi(\dd x)).
\end{align*}
Hence, the result above is valid for the process $X^{(n)}=\{X_t^{(n)},t\geq 0\} $ and, by taking $n\rightarrow \infty$ along a subsequence for which $X_t^{(n)}$ converges to $X_t$ almost surely, the result follows for the process $X$. The general case, when $F$ and its derivatives are not necessarily bounded, follows by a standard stopping-time argument. Moreover,
From the fact that $g_t$ is continuous on the set $\{ t\geq 0: g_{t-}=t \text{ or } g_{t-}<t \text{ and } X_t>0\}$, we obtain the first equality in \eqref{eq:ItoformulaforgttXt}. Lastly, the case when $X_t\leq 0$, is similar and the proof is omitted.

\subsubsection{Proof of Theorem \ref{thm:integralofgr}}
We recall that for any $\varepsilon>0$ and $r\geq 0$,
\begin{align*}
U_{\varepsilon,r}=r-g_{\varepsilon,r}=r-\sup\{0\leq s\leq r: X^{(\varepsilon)}_s\leq 0 \}.
\end{align*}
First, note that, since $|K(U_s,X_s)|\leq C(U_s,X_s)$ for all $s\geq 0$ and $\E_{u,x}\left(\int_0^{\infty} e^{-qr} C(U_r,X_r+y) \dd r\right)<\infty$ for all $(u,x)\in E$ and $y\in \R$, we have that $K^+$ and $K^-$ are finite. Moreover, since $u\mapsto C(u,x)$ is monotone for all $x\in \R$ and non-negative, we have that for all $r\geq 0$ and $\varepsilon>0$,
\begin{align*}
 |K(U_{\varepsilon,r},X_r^{(\varepsilon)})|\leq C(U_{\varepsilon,r},X_r^{(\varepsilon)})\leq C(U_r,X_r)+C(U_r,X_r-\varepsilon)+C(U_r^{(\varepsilon)},X_r)+C(U_r^{(\varepsilon)},X_r-\varepsilon),
\end{align*}
where  $U_r^{(\varepsilon)}=r-g_t^{(\varepsilon)}=r-\sup\{0\leq s\leq r: X_s \leq \varepsilon \}$ and we used that $U_t\geq U_{\varepsilon,t} \geq U_t^{(\varepsilon)}$, for all $t\geq 0$. It follows from integrability of $e^{-qr}C(U_r,X_r+y)$ with respect to the product measure $\P_{u,x} \times \dd r$, for all $(u,x)\in E$, by dominated convergence theorem and left-continuity in each argument of $K$ that for $x\leq 0$,
\begin{align*}
\E_{x}\left(\int_0^{\infty}e^{-qr} K(U_{r},X_r) \dd r \right)=\lim_{\varepsilon \downarrow 0} \E_{x}\left(\int_0^{\infty}e^{-qr} K(U_{\varepsilon,r},X_r^{(\varepsilon)}) \dd r \right),
\end{align*}
where we note that for $x\leq \varepsilon$, $X^{(\varepsilon)}_0=X_0-\varepsilon=x-\varepsilon$ and $U_{\varepsilon,0}=0$, under $\P_x$. Then we calculate the right-hand side of the equation above. We define the auxiliary function 
\begin{align*}
J^{(\varepsilon)}(x):=\E_x\left(\int_0^{\infty}e^{-qr} K(U_{\varepsilon,r},X_r^{(\varepsilon)}) \dd r \right), \qquad x\leq \varepsilon.
\end{align*}
Fix $\varepsilon>0$ and take any $x\leq \varepsilon$. Then, we can write 
\begin{align}
J^{(\varepsilon)}(x)
&=\E_x\left(\int_0^{\rho_{1,\varepsilon}^+} e^{-qr} K(0,X_r-\varepsilon)\dd r \right)+\E_x\left(\I_{\{\rho_{1,\varepsilon}^+<\infty \}}\int_{\rho_{1,\varepsilon}^+}^{\infty}e^{-qr} K(U_{\varepsilon,r},X_r^{(\varepsilon)}) \dd r \right)\nonumber\\
\label{eq:findingJ(x)}
&=K^-(x-\varepsilon)+\int_{0}^{\infty} \E_x\left(\I_{\{\rho_{1,\varepsilon}^+<\infty \}}e^{-q(r+\rho_{1,\varepsilon}^+)} K(U_{\varepsilon,r+\rho_{1,\varepsilon}^+},X_{r+\rho_{1,\varepsilon}^+}^{(\varepsilon)})\right)  \dd r ,
\end{align}
where we recall that $K^-$ is given in \eqref{eq:definitionofK-}, and the last equality follows by Fubini's theorem. By using that $g_{\varepsilon,t}=t$, when $t\in [\rho_{k,\varepsilon}^-,\rho_{k,\varepsilon}^+]$, and $g_{\varepsilon,t}=\rho_{k,\varepsilon}^+$, when $t\in [\rho_{k,\varepsilon}^+,\rho_{k+1,\varepsilon}^-)$, for some $k\geq 1$, we obtain that for any $r> 0$,
\begin{align*}
\E_x&\left(\I_{\{\rho_{1,\varepsilon}^+<\infty \}}e^{-q(r+\rho_{1,\varepsilon}^+)} K(U_{\varepsilon,r+\rho_{1,\varepsilon}^+},X_{r+\rho_{1,\varepsilon}^+}^{(\varepsilon)})\right)\\
&=\sum_{k=1}^{\infty}\E_x\left(\I_{\{\rho_{1,\varepsilon}^+<\infty \}}e^{-q(r+\rho_{1,\varepsilon}^+)} K(r+\rho_{1,\varepsilon}^+-\rho_{k,\varepsilon}^+,X_{r+\rho_{1,\varepsilon}^+}) \I_{\{\rho_{k,\varepsilon}^+ \leq r+\rho_{1,\varepsilon}^+ < \rho_{k+1,\varepsilon}^-\}}\right)\\
&\qquad+\sum_{k=1}^{\infty}\E_x\left(\I_{\{\rho_{1,\varepsilon}^+<\infty \}} e^{-q(r+\rho_{1,\varepsilon}^+)} K(0,X_{r+\rho_{1,\varepsilon}^+}-\varepsilon) \I_{\{\rho_{k+1,\varepsilon}^- \leq r+\rho_{1,\varepsilon}^+ < \rho_{k+1,\varepsilon}^+\}}\right)\\
&=\E_x\left(\I_{\{\rho_{1,\varepsilon}^+<\infty \}}e^{-q \rho_{1,\varepsilon}^+} \right) \sum_{k=1}^{\infty} \E_{\varepsilon}\left( e^{-qr}   K(r-\rho_{k,\varepsilon}^+,X_{r}) \I_{\{\rho_{k,\varepsilon}^+ \leq r < \rho_{k+1,\varepsilon}^-\}}\right)\\
&\qquad+\E_x\left(\I_{\{\rho_{1,\varepsilon}^+<\infty \}} e^{-q\rho_{1,\varepsilon}^+} \right) \sum_{k=1}^{\infty} \E_{\varepsilon}\left( e^{-qr} K(0,X_{r}-\varepsilon) \I_{\{\rho_{k+1,\varepsilon}^- \leq r < \rho_{k+1,\varepsilon}^+\}}\right)\\
&=e^{-\Phi(q) (\varepsilon-x)}  \E_{\varepsilon}\left( e^{-qr} K(U_{\varepsilon,r},X_{r}^{(\varepsilon)}) \right),
\end{align*}
where the second equality follows by applying the strong Markov property at time $\rho_{1,\varepsilon}^+$ and the last from \eqref{eq:laplacetransformtau0} and the definition of $U_{\varepsilon,r}$ and $X^{(\varepsilon)}_r$. Hence, substituting the expression above into \eqref{eq:findingJ(x)} we deduce that for any $x\leq \varepsilon$,
\begin{align}
J^{(\varepsilon)}(x)&=K^-(x-\varepsilon)+e^{-\Phi(q) (\varepsilon-x)} \int_{0}^{\infty} \E_{\varepsilon}\left( e^{-qr} K(U_{\varepsilon,r},X_{r}^{(\varepsilon)}) \right) \dd r\nonumber\\
\label{eq:renewalequationforJ(x)}
&=K^-(x-\varepsilon)+e^{-\Phi(q) (\varepsilon-x)} J^{(\varepsilon)}(\varepsilon).
\end{align}
On the other hand, using a similar argument and that $\rho_{2,\varepsilon}^-=\tau_0^-$ under $\P_{\varepsilon}$, we obtain that 
\begin{align*}
J^{(\varepsilon)}(\varepsilon)
&=K^+(0,\varepsilon)+\E_{\varepsilon}\left( \I_{\{ \tau_0^-<\infty\}}e^{-q\tau_0^-} J^{(\varepsilon)}( X_{\tau_0^-})  \right)\\
&=K^+(0,\varepsilon)+\E_{\varepsilon}\left( \I_{\{ \tau_0^-<\infty\}}e^{-q\tau_0^-} K^-(X_{\tau_0^-}-\varepsilon)\right)+\E_{\varepsilon}\left( \I_{\{ \tau_0^-<\infty\}}e^{-q\tau_0^-}e^{-\Phi(q) (\varepsilon-X_{\tau_0^-})}   \right)J^{(\varepsilon)}(\varepsilon)\\
&=K^+(0,\varepsilon)+\E_{\varepsilon}\left( \I_{\{ \tau_0^-<\infty\}}e^{-q\tau_0^-} K^-(X_{\tau_0^-}-\varepsilon)\right)+\mathcal{I}^{(q,\Phi(q))}(\varepsilon)J^{(\varepsilon)}(\varepsilon),
\end{align*}
where the second equality follows by \eqref{eq:renewalequationforJ(x)} and that $X_{\tau_0^-}\leq 0$, on the event $\{\tau_0^-<\infty \}$, and the last equality follows from \eqref{eq:jointlaplacetransformtau0-Xtau0-}. Hence, solving for $J^{(\varepsilon)}(\varepsilon)$ in the expression above we obtain that 
\begin{align*}
J^{(\varepsilon)}(\varepsilon) 
&=\frac{1}{1-\mathcal{I}^{(q,\Phi(q))}(\varepsilon) }\left[ K^+(0,\varepsilon)+\E_{\varepsilon}\left( \I_{\{ \tau_0^-<\infty\}}e^{-q\tau_0^-} K^-(X_{\tau_0^-}-\varepsilon)\right)\right].
\end{align*}
If we further substitute the expression found for $J^{(\varepsilon)}(\varepsilon)$ into \eqref{eq:renewalequationforJ(x)}, we deduce that for any $x\leq \varepsilon$,
\begin{align*}
J^{(\varepsilon)}(x)= K^-(x-\varepsilon)+ \frac{e^{-\Phi(q) (\varepsilon-x)}}{1-\mathcal{I}^{(q,\Phi(q))}(\varepsilon) }\left[ K^+(0,\varepsilon)+\E_{\varepsilon}\left( \I_{\{ \tau_0^-<\infty\}}e^{-q\tau_0^-} K^-(X_{\tau_0^-}-\varepsilon)\right)\right].
\end{align*}
Therefore, by the dominated convergence theorem we have that for all $x\leq 0$,
\begin{align*}
\E_x&\left(\int_0^{\infty}e^{-qr} K(U_r,X_r) \dd r \right)\\
&=  \lim_{\varepsilon \downarrow 0 }\left\{  K^-(x-\varepsilon)+\frac{e^{-\Phi(q) (\varepsilon-x)}}{ 1-\mathcal{I}^{(q,\Phi(q))}(\varepsilon)} \left[ \E_{\varepsilon}\left(  \I_{\{ \tau_0^-<\infty\}} e^{-q \tau_0^-} K^-(X_{\tau_0^-}-\varepsilon) \right)+K^+(0,\varepsilon) \right] \right\}  .
\end{align*}
Using Fubini's theorem and equation \eqref{eq:qpotentialdensitytkillingonexitinga} we have that for all $x<0$,
\begin{align}
\label{eq:expressionforK-}
K^-(x)=\int_{(-\infty,0)}K(0,y) \int_0^{\infty} e^{-qr} \P_x(X_r\in \dd y, r<\tau_0^+) \dd r=\int_{-\infty}^0 K(0,y) [e^{\Phi(q)x}W^{(q)}(-y) -W^{(q)}(x-y)] \dd y.
\end{align}
Recall that we are assuming that $K$ is bounded by a non-negative integrable function $C$. Then, we have that for any $y<0$ and $x<0$, 
\begin{align*}
|K(0,y) [e^{\Phi(q)x}W^{(q)}(-y) -W^{(q)}(x-y)]|\leq C(0,y) [e^{\Phi(q)x}W^{(q)}(-y) -W^{(q)}(x-y)].
\end{align*}
Moreover, we have that the quantity on the right-hand side of the equation above is integrable since for any $x<0$,
\begin{align*}
C^-(x):= \E_x\left(\int_0^{\tau_0^+} e^{-qr}C(0,X_r)\dd r \right)<\E_x\left(\int_0^{\infty} e^{-qr}C(U_r,X_r)\dd r \right)<\infty
\end{align*}
and 
\begin{align*}
C^-(x)=\int_{-\infty}^0 C(0,y) [e^{\Phi(q)x}W^{(q)}(-y) -W^{(q)}(x-y)] \dd y,
\end{align*}
where we used Fubini's theorem and equation \eqref{eq:qpotentialdensitytkillingonexitinga}. Furthermore, it can be seen (see, for example, the proof of Theorem 8.1 in \cite{kyprianou2014fluctuations}) that for any $q>0$ and $y<0$, the mapping $x\mapsto e^{-\Phi(q)x}W^{(q)}(x-y)$ is non-decreasing. Thus, for any $x<0$ and any $x_1<x_2<0$ such that $x\in [x_1,x_2]$,
\begin{align*}
C^-(x)&= e^{\Phi(q) x}\int_{-\infty}^0 C(0,y) [W^{(q)}(-y) -e^{-\Phi(q)x}W^{(q)}(x-y)] \dd y\\
&\leq e^{\Phi(q) x}\int_{-\infty}^0 C(0,y) [W^{(q)}(-y) -e^{-\Phi(q)x_1}W^{(q)}(x_1-y)] \dd y\\
%&= e^{\Phi(q) (x-x_1)}C^{-}(x_1)\\
&\leq  e^{\Phi(q) (x_2-x_1)}C^{-}(x_1).
\end{align*}
Hence, since $W$ is continuous on $(0,\infty)$ and by the dominated convergence theorem, we deduce that $K^-$ is a continuous function. Thus, we have that for any $x\leq 0$,
\begin{align}
\E_x&\left(\int_0^{\infty}e^{-qr} K(U_r,X_r) \dd r \right)\nonumber\\
\label{eq:limitexpressionofintKforxnegative}
&=   K^-(x)+\lim_{\varepsilon \downarrow 0 } \frac{e^{-\Phi(q) (\varepsilon-x)}}{ 1-\mathcal{I}^{(q,\Phi(q))}(\varepsilon)} \left[ \E_{\varepsilon}\left(  \I_{\{ \tau_0^-<\infty\}} e^{-q \tau_0^-} K^-(X_{\tau_0^-}-\varepsilon) \right)+K^+(0,\varepsilon) \right]   .
\end{align}
We then proceed to find the limit on the right-hand side of the equation above. For any $x>0$ and $\varepsilon>0$, we deduce from Fubini's theorem, equation \eqref{eq:jointlaplacetransformtau0-Xtau0-} and equation \eqref{eq:expressionforK-} that
\begin{align*}
&\E_{x}\left(  \I_{\{ \tau_0^-<\infty\}} e^{-q \tau_0^-} K^-(X_{\tau_0^-}-\varepsilon) \right)\\
&= e^{\Phi(q)(x-\varepsilon)}\mathcal{I}^{(q,\Phi(q))}(x) \int_{-\varepsilon}^0 K(0,y) W^{(q)}(-y)\dd y \\
&\qquad +\int_{-\infty}^{-\varepsilon}K(0,y)\left[e^{\Phi(q)(x-\varepsilon)}\mathcal{I}^{(q,\Phi(q))}(x) W^{(q)}(-y) - \E_x\left(\I_{\{\tau_0^-<\infty \}} e^{-q\tau_0^-}W^{(q)}(X_{\tau_0^-}-\varepsilon-y)\right) \right] \dd y.
\end{align*}
Let $x,\varepsilon>0$ and $y<-\varepsilon$. From the monotone convergence theorem and \eqref{eq:laplacetransformoftaua+beforetau0-} we have that 
\begin{align}
\E_x&\left(\I_{\{\tau_0^-<\infty \}} e^{-q\tau_0^-}W^{(q)}(X_{\tau_0^-}-\varepsilon-y)\right)\nonumber\\
&=\lim_{a\rightarrow \infty} \E_x\left(\I_{\{\tau_0^-<\tau_a^+ \}} e^{-q\tau_0^-}W^{(q)}(X_{\tau_0^-}-\varepsilon-y)\right)\nonumber\\
&=\lim_{a\rightarrow \infty} \E_x\left( e^{-q\tau_0^-\wedge \tau_a^+ } W^{(q)}(X_{\tau_0^-\wedge \tau_a^+}-\varepsilon-y)\right)-\lim_{a\rightarrow \infty} \E_x\left( \I_{\{\tau_a	^+< \tau_0^- \}} e^{-q \tau_a^+ } W^{(q)}(a-\varepsilon-y)\right)\nonumber\\
&=\lim_{a\rightarrow \infty} \E_{x-\varepsilon -y}\left( e^{-q\tau_{-\varepsilon -y}^-\wedge \tau_{a-\varepsilon-y}^+ } W^{(q)}(X_{\tau_{-\varepsilon-y}^-\wedge \tau_{a-\varepsilon-y}^+})\right)-\lim_{a\rightarrow \infty}  W^{(q)}(a-\varepsilon-y)\frac{W^{(q)}(x)}{W^{(q)}(a)}\nonumber\\
\label{eq:expectationofW(Xtau)}
&=W^{(q)}(x-\varepsilon-y)-e^{-\Phi(q)(\varepsilon+y)} W^{(q)}(x),
\end{align}
where the last equality follows since, for any $a\geq 0$, the process $e^{-q(t\wedge \tau_0^-\wedge \tau_a^+)}W^{(q)}(X_{t\wedge \tau_0^-\wedge \tau_a^+})$ is a martingale, the optional sampling theorem (note that $\tau_{-\varepsilon-y}^-\leq \tau_0^-$) and since $\lim_{a \rightarrow \infty} W^{(q)}(a-z)/W^{(q)}(a)=e^{-\Phi(q)z}$ for $z\leq a$ and $a\geq 0$ (see Exercises 8.5 and 8.12 in \cite{kyprianou2014fluctuations}). 
%
%
%
%Let $x,\varepsilon>0$ and $y<-\varepsilon$, using a change of measure (see equation \eqref{eq:exponentialchangeofmeasure}) we obtain that
%\begin{align*}
%\E_x\left(\I_{\{\tau_0^-=\infty \}} e^{-q\tau_0^-}W^{(q)}(X_{\tau_0^-}-\varepsilon-y)\right)&=e^{\Phi(q)(x-\varepsilon-y)}\E_x^{\Phi(q)}\left(\I_{\{\tau_0^-=\infty \}}    e^{-\Phi(q) (X_{\tau_0^-}-\varepsilon-y)}W^{(q)}(X_{\tau_0^-}-\varepsilon-y)\right)\\
%&=e^{\Phi(q)(x-\varepsilon-y)} \Phi'(q)\P_x^{\Phi(q)}(\tau_0^-=\infty)\\
%&=e^{-\Phi(q)(\varepsilon+y)} W^{(q)}(x),
%\end{align*}
%where in the second equality we used that fact that $X$ drifts to infinity under the measure $\P^{\Phi(q)}$ and the last follows from equation \eqref{eq:jointlaplacetransformtau0-Xtau0-}. Then from the fact that $e^{-q(t\wedge \tau_0^-)}W^{(q)}(X_{t\wedge \tau_0^-})$ is a martingale and since $\tau_{-\varepsilon-y}^-<\tau_0^-$ we have that 
%\begin{align*}
%\E_x&\left(\I_{\{\tau_0^-<\infty \}} e^{-q\tau_0^-}W^{(q)}(X_{\tau_0^-}-\varepsilon-y)\right)\\
%&=\E_{x}\left( e^{-q\tau_{0}^-}W^{(q)}(X_{\tau_{0}^-}-\varepsilon-y)\right)-\E_x\left(\I_{\{\tau_0^-=\infty \}} e^{-q\tau_0^-}W^{(q)}(X_{\tau_0^-}-\varepsilon-y)\right)\\
%&=W^{(q)}(x-\varepsilon-y)-e^{-\Phi(q)(\varepsilon+y)} W^{(q)}(x).
%\end{align*}
Hence, we obtain that for any $x>0$ and $\varepsilon>0$,
\begin{align*}
&\E_{x}\left(  \I_{\{ \tau_0^-<\infty\}} e^{-q \tau_0^-} K^-(X_{\tau_0^-}-\varepsilon) \right)\\
&= e^{\Phi(q)(x-\varepsilon)}\mathcal{I}^{(q,\Phi(q))}(x) \int_{-\varepsilon}^0 K(0,y) W^{(q)}(-y)\dd y \\
&\qquad +\int_{-\infty}^{-\varepsilon}K(0,y)\left[e^{\Phi(q)(x-\varepsilon)}\mathcal{I}^{(q,\Phi(q))}(x) W^{(q)}(-y) -W^{(q)}(x-\varepsilon-y)+e^{-\Phi(q)(\varepsilon+y)} W^{(q)}(x) \right] \dd y.
\end{align*}
In particular, when $x=\varepsilon$ we obtain that 
\begin{align*}
&\E_{\varepsilon}\left(  \I_{\{ \tau_0^-<\infty\}} e^{-q \tau_0^-} K^-(X_{\tau_0^-}-\varepsilon) \right)\\
&= \mathcal{I}^{(q,\Phi(q))}(\varepsilon) \int_{-\varepsilon}^0 K(0,y) W^{(q)}(-y)\dd y \\
&\qquad +\int_{-\infty}^{-\varepsilon}K(0,y)\left( [\mathcal{I}^{(q,\Phi(q))}(\varepsilon)-1] W^{(q)}(-y) +e^{-\Phi(q)(\varepsilon+y)} W^{(q)}(\varepsilon)   \right) \dd y.
\end{align*}
Thus, for any $\varepsilon>0$ and $x\leq 0$,
\begin{align*}
 & \frac{e^{-\Phi(q) (\varepsilon-x)}}{ 1-\mathcal{I}^{(q,\Phi(q))}(\varepsilon)}  \E_{\varepsilon}\left(  \I_{\{ \tau_0^-<\infty\}} e^{-q \tau_0^-} K^-(X_{\tau_0^-}-\varepsilon) \right)\\
&\qquad=\frac{e^{-\Phi(q) (\varepsilon-x)}}{ 1-\mathcal{I}^{(q,\Phi(q))}(\varepsilon)}  \mathcal{I}^{(q,\Phi(q))}(\varepsilon) \int_{-\varepsilon}^0 K(0,y) W^{(q)}(-y)\dd y \\
&\qquad\qquad +\frac{e^{-\Phi(q) (\varepsilon-x)}}{ 1-\mathcal{I}^{(q,\Phi(q))}(\varepsilon)} \int_{-\infty}^{-\varepsilon}K(0,y)\left( [\mathcal{I}^{(q,\Phi(q))}(\varepsilon)-1] W^{(q)}(-y) +e^{-\Phi(q)(\varepsilon+y)} W^{(q)}(\varepsilon)   \right) \dd y\\
&\qquad=  \Phi'(q)\frac{e^{\Phi(q) x}}{ W^{(q)}(\varepsilon)}   \int_{-\varepsilon}^0 K(0,y) W^{(q)}(-y)\dd y-e^{-\Phi(q) (\varepsilon-x)}  \int_{-\varepsilon}^0 K(0,y) W^{(q)}(-y)\dd y \\
&\qquad \qquad- e^{-\Phi(q) (\varepsilon-x)}\int_{-\infty}^{-\varepsilon}K(0,y)\left(  W^{(q)}(-y) -\Phi'(q)e^{-\Phi(q)y}   \right) \dd y,
\end{align*}
where we used that $ 1-\mathcal{I}^{(q,\Phi(q))}(\varepsilon)= [\Phi'(q)]^{-1} e^{-\Phi(q) \varepsilon}W^{(q)}(\varepsilon) $. Moreover, by using that $W^{(q)}$ is increasing on $(0,\infty)$ we see that 
\begin{align*}
0\leq  \left|\lim_{\varepsilon \downarrow 0} \Phi'(q)\frac{e^{\Phi(q) x}}{ W^{(q)}(\varepsilon)}   \int_{-\varepsilon}^0 K(0,y) W^{(q)}(-y)\dd y \right|\leq  \lim_{\varepsilon \downarrow 0}\Phi'(q)\frac{e^{\Phi(q) x}}{ W^{(q)}(\varepsilon)}   W^{(q)}(\varepsilon) \int_{-\varepsilon}^0 |K(0,y)| \dd y=0.
\end{align*}
Thus, we conclude that for any $x\leq 0$, 
\begin{align*}
\lim_{\varepsilon \downarrow 0} & \frac{e^{-\Phi(q) (\varepsilon-x)}}{ 1-\mathcal{I}^{(q,\Phi(q))}(\varepsilon)}  \E_{\varepsilon}\left(  \I_{\{ \tau_0^-<\infty\}} e^{-q \tau_0^-} K^-(X_{\tau_0^-}-\varepsilon) \right)\\
&=- e^{\Phi(q) x}\int_{-\infty}^{0}K(0,y)\left(  W^{(q)}(-y) -\Phi'(q)e^{-\Phi(q)y}   \right) \dd y.
\end{align*}
Substituting the expression above into \eqref{eq:limitexpressionofintKforxnegative} and using the expression obtained for $K^-(x)$ in \eqref{eq:expressionforK-}, we deduce that for any $x\leq 0$,
\begin{align}
\E_x&\left(\int_0^{\infty}e^{-qr} K(U_r,X_r) \dd r \right)\nonumber\\
&=   K^-(x)+\lim_{\varepsilon \downarrow 0 } \frac{e^{-\Phi(q) (\varepsilon-x)}}{ 1-\mathcal{I}^{(q,\Phi(q))}(\varepsilon)} \left[ \E_{\varepsilon}\left(  \I_{\{ \tau_0^-<\infty\}} e^{-q \tau_0^-} K^-(X_{\tau_0^-}-\varepsilon) \right)+K^+(0,\varepsilon) \right]\nonumber \\
\label{eq:expressionforintKUXwhenxnegative}
&=\int_{-\infty}^0 K(0,y) [\Phi'(q)e^{-\Phi(q)(y-x)} -W^{(q)}(x-y)] \dd y+ \Phi'(q)e^{\Phi(q) x}\lim_{\varepsilon \downarrow 0 } \frac{K^+(0,\varepsilon)}{ W^{(q)}(\varepsilon)}.
\end{align}
For the case when $u>0$ and $x>0$, using the expression above and the strong Markov property, we obtain that 
\begin{align*}
\E_{u,x}&\left(\int_0^{\infty}e^{-qr} K(U_r,X_r) \dd r \right)\\
&=\E_x\left(\int_0^{\tau_0^-}e^{-qr} K(u+r,X_r) \dd r \right)+\E_{x}\left(\I_{\{\tau_0^-<\infty \}}\int_{\tau_0^-}^{\infty}e^{-qr} K(U_r,X_r) \dd r \right)\\
&=\E_x\left(\int_0^{\tau_0^-}e^{-qr} K(u+r,X_r) \dd r \right)+\E_{x}\left(\I_{\{\tau_0^-<\infty \}} e^{q \tau_0^-} \E_{X_{\tau_0^-}}\left[\int_{0}^{\infty}e^{-qr} K(U_r,X_r) \dd r \right] \right)\\
&=K^+(u,x)\\
&\qquad+\int_{-\infty}^0 K(0,y) \left[\Phi'(q)e^{-\Phi(q)y} \E_{x}\left(\I_{\{\tau_0^-<\infty \}}e^{-q\tau_0^-} e^{\Phi(q) X_{\tau_0^-}} \right) - \E_{x}\left(\I_{\{\tau_0^-<\infty \}}e^{-q\tau_0^-}  W^{(q)}(X_{\tau_0^-}-y) \right)\right] \dd y\\
&\qquad+ \Phi'(q)\E_{x}\left(\I_{\{\tau_0^-<\infty \}}e^{-q\tau_0^-} e^{\Phi(q) X_{\tau_0^-}} \right)\lim_{\varepsilon \downarrow 0 } \frac{K^+(0,\varepsilon)}{ W^{(q)}(\varepsilon)}\\
&=K^+(u,x)\\
&\qquad+\int_{-\infty}^0 K(0,y) \left[\Phi'(q)e^{-\Phi(q)y} e^{\Phi(q)x}\mathcal{I}^{(q,\Phi(q))}(x) -W^{(q)}(x-y)+e^{-\Phi(q)y} W^{(q)}(x)\right] \dd y\\
&\qquad+ \Phi'(q)e^{\Phi(q)x}\mathcal{I}^{(q,\Phi(q))}(x)\lim_{\varepsilon \downarrow 0 } \frac{K^+(0,\varepsilon)}{ W^{(q)}(\varepsilon)}\\
&=K^+(u,x)+\int_{-\infty}^0 K(0,y) \left[\Phi'(q)e^{-\Phi(q)(y-x)} -W^{(q)}(x-y)\right] \dd y\\
&\qquad+ [\Phi'(q)e^{\Phi(q)x}-W^{(q)}(x)]\lim_{\varepsilon \downarrow 0 } \frac{K^+(0,\varepsilon)}{ W^{(q)}(\varepsilon)},
\end{align*}
where the second last equality follows from \eqref{eq:jointlaplacetransformtau0-Xtau0-}  and \eqref{eq:expectationofW(Xtau)} and the last by substituting the value of $\mathcal{I}^{(q,\Phi(q)}(x)$ (see \eqref{eq:functionI}). Lastly, since $W^{(q)}(x)=0$ for $x<0$, $W^{(q)}(0)=0$ when $X$ has paths of infinite variation and $W^{(q)}(0)>0$ when $X$ has paths of finite variation, we note that the expression above coincides with \eqref{eq:expressionforintKUXwhenxnegative} when $x\leq 0$. In other words, the expression above is valid for any $(u,x)\in E$. The proof is now complete.

\subsection{Proof of Theorem \ref{thm:solutiontooptimalstopping}}
\label{subsec:solutiontooptimalstopping}

We first state a verification Lemma that provides sufficient conditions for the optimality of a given candidate solution $\tau^*$.
\begin{lemma}
\label{lemma:verificationlemma}
Suppose that $\tau^*$ is candidate solution to the optimal stopping problem and let $V^*$ its corresponding value function, i.e., $V^*(u,x)= \E_{x,u}\left(\int_0^{\tau^*}e^{-rs}G(U_s,X_s)\dd s \right)$. Assume that 
\begin{enumerate}
\item[i)] $V^*(u,x)\geq 0$ for all $(u,x)\in E$.
\item[ii)] For each $(u,x)\in E$ and $N>0$, the stochastic process $\{Z_{t\wedge \tau_N^+}, t\geq 0 \}$ is a supermartingale under the measure $\P_{u,x}$, where
\begin{align*}
Z_t=e^{-rt}V^*(U_t,X_t)+\int_0^t e^{-rs}G(U_s,X_s)\dd s.
\end{align*}
\end{enumerate}
Then $V=V^*$ and the stopping time $\tau^*$ is an optimal stopping time for \eqref{eq:optimalstoppingproblem}.
\end{lemma}
\begin{proof}
From the definition of $V$, we deduce that $V\geq V^*$. On the other hand, due to the optimal sampling theorem we have that, for any $t\geq 0$, $N>0$ and any stopping time $\tau \in \mathcal{T}$, the stopped process $Z_{t\wedge \tau \wedge \tau_N^+ }$ is a supermartingale. This implies that for any $t\geq 0$, $N>0$  and $\tau \in \mathcal{T}$,
\begin{align*}
V^*(u,x) \geq \E_{u,x}\left(e^{-r(T\wedge \tau)} V^*(U_{T\wedge \tau },X_{T\wedge \tau})+\int_0^{T\wedge \tau } e^{-rs}G(U_s,X_s)\dd s\right)\geq \E_{u,x}\left( \int_0^{T\wedge \tau}e^{-rs} G(U_s,X_s)\dd s\right),
\end{align*}
where $T=t\wedge \tau_N^+$ and the last inequality follows since $V^*\geq 0$, by assumption. From the dominated convergence theorem we conclude (see \eqref{eq:OSintegrabilitycontidion}), by taking $t,N\rightarrow \infty$ in the equation above, that
\begin{align*}
V^*(u,x) \geq \E_{u,x}\left( \int_0^{\tau}e^{-rs} G(U_s,X_s)\dd s\right)
\end{align*}
for all $(u,x)\in E$ and $\tau \in \mathcal{T}$. Hence, we have that $V\leq V^*$, implying that $V=V^*$. Therefore, the supremum in \eqref{eq:optimalstoppingproblem} is attained by $\tau^*$ as claimed.
\end{proof}
For $z\leq 0$ fixed, we define the function
\begin{align*}
V_z(u,x):=\E_{x,u}\left(\int_0^{\tau_z^-}e^{-rs}G(U_s,X_s)\dd s \right),
\end{align*}
for $(u,x)\in E$. The following lemma gives a semi-explicit expression for $V_z$ in terms of the scale functions.
	
\begin{lemma}
For any $z<	0$ and $(u,x)\in E$ such that $x\geq z$ we have that
\begin{align}
V_z(u,x)&=K^+(u,x) + \left[e^{\Phi(r)z}W^{(r)}(x-z)-W^{(r)}(x) \right]\int_{(0,\infty)}\int_0^{\infty}  G(v,y)  \frac{y}{v}e^{-rv}\P(X_{v} \in \dd y) \dd v\nonumber\\
\label{eq:expressionforVz}
&\qquad+e^{\Phi(r) z}W^{(r)}(x-z)\int_{z}^{0}  G(0,y) e^{-\Phi(r) y}\dd y -\int_{z}^{0}  G(0,y)W^{(r)}(x-y)\dd y.
\end{align}
\end{lemma}

\begin{proof}
Note that for any $(u,x)\in E$,
\begin{align*}
V_z(u,x)=\E_{x,u}\left(\int_0^{\tau_z^-}e^{-rs}G(U_s,X_s)\I_{\{X_s>0 \}}\dd s \right)+\E_{x}\left(\int_0^{\tau_z^-}e^{-rs}G(0,X_s)\I_{\{X_s\leq 0 \}}\dd s \right),
\end{align*}
where the two terms on the right-hand side above are finite due to equation \eqref{eq:OSintegrabilitycontidion}. Using equation \eqref{eq:qpotentialdensitytkillingonexiting0} and Fubini's theorem we deduce that for any $x\geq z$,
\begin{align*}
\E_{x}\left(\int_0^{\tau_z^-}e^{-rs}G(0,X_s)\I_{\{X_s\leq 0 \}}\dd s \right)&=\E_{x-z}\left(\int_0^{\tau_0^-}e^{-rs}G(0,X_s+z)\I_{\{X_s+z\leq 0 \}}\dd s \right)\\
&=\int_{(0,-z]}  G(0,y+z)  \int_0^{\infty} e^{-rs}\P_{x-z}\left(X_s\in \dd y, s<\tau_0^- \right)\dd s\\
&=\int_{0}^{-z}  G(0,y+z) \left[ e^{-\Phi(r) y}W^{(r)}(x-z) -W^{(r)}(x-z-y) \right]\dd y\\
&=e^{\Phi(r) z}W^{(r)}(x-z)\int_{z}^{0}  G(0,y) e^{-\Phi(r) y}\dd y -\int_{z}^{0}  G(0,y)W^{(r)}(x-y)\dd y.
\end{align*}
On the other hand, from the strong Markov property, we have that for $(u,x)\in E$ such that $x\geq z$,  
\begin{align*}
\E_{x,u}\left(\int_0^{\tau_z^-}e^{-rs}G(U_s,X_s)\I_{\{X_s>0 \}}\dd s \right)
&=H(u,x)-\E_{x}(e^{-r \tau_z^-}\I_{\{\tau_z^- <\infty \}} H(0,X_{\tau_z^-})),
\end{align*}
where 
\begin{align}
\label{eq:definitionofH}
H(u,x)&:=\E_{u,x}\left(\int_0^{\infty} e^{-rs}G(U_s,X_s)\I_{\{X_s>0 \}}\dd s \right).
\end{align}
It follows from \eqref{eq:OSintegrabilitycontidion} that $|H(u,x)|<\infty$ for all $(u,x)\in E$. Hence, by using the potential measure of $(U,X)$ given in Corollary \ref{cor:qpotentialmeasureUX} we see that for any $(u,x)\in E$, 
\begin{align}
H(u,x)
&=\int_{(0,\infty)}\int_0^{\infty} G(v,y) \int_0^{\infty}e^{-rs}\P_{u,x}(X_s\in \dd y, U_s\in \dd v) \dd s \nonumber\\
&=\int_{(0,\infty)}\int_u^{\infty} G(v,y)e^{-r(v-u)}\P_x(X_{v-u}\in \dd y,v-u<\tau_0^-)\dd v\nonumber\\
&\qquad + \left[e^{\Phi(r) x}\Phi'(r)-W^{(r)}(x) \right]\int_{(0,\infty)}\int_0^{\infty}  G(v,y)  \frac{y}{v}e^{-rv}
\P(X_{v} \in \dd y) \dd v  \nonumber\\
\label{eq:expressionforH}
&=K^+(u,x) + \left[e^{\Phi(r) x}\Phi'(r)-W^{(r)}(x) \right]\int_{(0,\infty)}\int_0^{\infty}  G(v,y)  \frac{y}{v}e^{-rv}\P(X_{v} \in \dd y) \dd v .		
\end{align}
So that, by using \eqref{eq:jointlaplacetransformtau0-Xtau0-} and since $K^+(0,x)=0=W^{(q)}(x)$ when $x<0$, we deduce that for any $x\in \R$, 
\begin{align}
\E_{x}(e^{-r \tau_z^-}\I_{\{\tau_z^- <\infty \}} H(0,X_{\tau_z^-}))&=\Phi'(r)\E_x\left(e^{-r\tau_z^-+\Phi(r) X_{\tau_z^-}} \I_{\{\tau_z^-<\infty \}}\right) \int_{(0,\infty)}\int_0^{\infty}  G(v,y)  \frac{y}{v}e^{-rv}\P(X_{v} \in \dd y) \dd v  \nonumber\\
&=\Phi'(r)\E_{x-z}\left(e^{-r \tau_0^-+ \Phi(r) (X_{\tau_0^-}+z)}\I_{\{\tau_0^-<\infty \}} \right)\nonumber\\
&\qquad\times\left[ \int_{(0,\infty)}\int_0^{\infty}  G(v,y) \frac{y}{v}e^{-rv}\P(X_{v} \in \dd y) \dd v \right]\nonumber\\
%&=\Phi'(r)e^{\Phi(r)z} e^{\Phi(r)(x-z)}\left(1-\psi'(\Phi(r)+)e^{-\Phi(r)(x-z)}W^{(r)}(x-z) \right)\left[ \int_{(0,\infty)}\int_0^{\infty}  G(v,y)  \frac{y}{v}e^{-rv}\P(X_{v} \in \dd y) \dd v + \int_{z}^0 G(0,y) e^{-\Phi(0) y}\dd y \right]\\
\label{eq:HXtauz}
&=\left(\Phi'(r) e^{\Phi(r)x} -e^{\Phi(r)z}W^{(r)}(x-z) \right) \int_{(0,\infty)}\int_0^{\infty}  G(v,y)  \frac{y}{v}e^{-rv}\P(X_{v} \in \dd y) \dd v.
\end{align}
Therefore, we get that 
\begin{align*}
\E_{x,u}&\left(\int_0^{\tau_z^-}e^{-rs}G(U_s,X_s)\I_{\{X_s>0 \}}\dd s \right)\\
&=H(u,x)-\E_{x}(e^{-r \tau_z^-}\I_{\{\tau_z^- <\infty \}} H(0,X_{\tau_z^-}))\\
&=K^+(u,x) + \left[e^{\Phi(r)z}W^{(r)}(x-z)-W^{(r)}(x) \right]\int_{(0,\infty)}\int_0^{\infty}  G(v,y)  \frac{y}{v}e^{-rv}\P(X_{v} \in \dd y) \dd v 
\end{align*}
for any $(u,x)\in E$. The result follows. 
\end{proof}
For optimal stopping problems it is common to choose candidate solutions to satisfy the principle of smooth fit. Recall that we are assuming that $\sigma>0$ so that, in this case, $W^{(r)}$ is $C^2$ on $(0,\infty)$ with $W^{(r)\prime}(0+)=2/\sigma^2$ (see e.g. Theorem 3.10 and Lemma 3.2  \cite{kyprianou2011theory}). Then, by differentiating $V_z(u,x)$ with respect to $x$, we obtain for $z<x<0$ that,
\begin{align*}
\frac{\partial}{\partial x} V_z(0,x)&=e^{\Phi(r)z}W^{(r)\prime}(x-z)\left[\int_{(0,\infty)}\int_0^{\infty}  G(v,y)  \frac{y}{v}e^{-rv}\P(X_{v} \in \dd y) \dd v  +\int_{z}^0 G(0,y) e^{-\Phi(r) y}\dd y\right]\\
&\qquad -\int_{z}^x G(0,y)W^{(r)\prime}(x-y) \dd y.
\end{align*}
Then, by letting $x\downarrow z$, we see that the equation 
\begin{align*}
\frac{\partial}{\partial x} V_z(0,z+)=0
\end{align*}
is satisfied if and only if $z$ is solution to the equation
\begin{align*}
\int_{(0,\infty)}\int_0^{\infty}  G(v,y)  \frac{y}{v}e^{-rv}\P(X_{v} \in \dd y) \dd v  +\int_{z}^0 G(0,y) e^{-\Phi(r) y}=0.
\end{align*}
That is, if $z=z^*$. In the following lemma, we verify that the characterisation of $z^*$ given in the statement of Theorem \ref{thm:solutiontooptimalstopping} indeed holds and that condition i) given in Lemma \ref{lemma:verificationlemma} holds when $z\geq z^*$. 

\begin{lemma}
\label{eq:uniquenessofeqautionandVispositive}
For $z\leq 0$, we define the function
\begin{align*}
f(z)=\int_{(0,\infty)}\int_0^{\infty}  G(v,y)  \frac{y}{v}e^{-rv}\P(X_{v} \in \dd y) \dd v+\int_{z}^0 G(0,y) e^{-\Phi(r	) y}\dd y.
\end{align*}
Then the equation $f(z)=0$ has a unique solution $z^*$ on $(-\infty,0)$ such that $z^*\leq y_0$. Moreover, we have that $V_z(u,x)\geq 0$ for all $z\geq z^*$ and $(u,x)\in E$.
\end{lemma}
\begin{proof}
From Corollary \ref{cor:qpotentialmeasureUX} and by assumption \eqref{eq:OSintegrabilitycontidion} we know that  
\begin{align*}
0\leq f(0)
=\int_{(0,\infty)}\int_0^{\infty}  G(v,y)  \frac{y}{v}e^{-rv}\P(X_{v} \in \dd y) \dd v=\frac{1}{\Phi'(r)}\E\left(\int_0^{\infty} e^{-rs}G(U_s,X_s)\I_{\{X_s>0 \}}\dd s \right)<\infty.
\end{align*}
On the other hand, since $G(0,y)$ is non positive on $(-\infty,y_0)$, we have that $f(z)$ is increasing on $(-\infty,y_0)$ with $f(z)>0$ for all $ y_0\leq z\leq 0$ and $\lim_{y\rightarrow -\infty}f(z)=-\infty$, where the latter follows due to the assumption $\lim_{y\rightarrow -\infty }G(0,y)<0$. Then, due to the continuity of $f$, we see that the equation $f(z)=0$ has a unique solution $z^*$ on $(-\infty,y_0)$.\\

Next, we proceed to show the statement on $V_z$. Since $G(u,x)$ is non negative for all $(u,x)\in E$ such that $x\geq y_0$, we see that $V_z(x,u)\geq 0$ for all $(u,x)\in E$ and $z\geq y_0$. Take $z<0$ and $(u,x)\in E$ such that $x>z$, we see from \eqref{eq:expressionforVz} that 

\begin{align*}
\frac{\partial }{\partial z} V_z(u,x)=f(z)\frac{\partial }{\partial z} (e^{\Phi(r)z}W^{(r)}(x-z)).
\end{align*}
Note that we can write $e^{\Phi(r)z}W^{(r)}(x-z)=e^{\Phi(r)x}W_{\Phi(r)}(x-z)$, where $W_{\Phi(r)}$ is the $r$-scale function under the measure $\P^{\Phi(r)}$ (see e.g. the proof of Theorem 8.1 in \cite{kyprianou2014fluctuations}). Then we see that the mapping $z \mapsto e^{\Phi(r)z}W^{(r)}(x-z)$ is non increasing on $\R$, and then, $\frac{\partial }{\partial z} V_z(u,x)\leq 0$ for all $(u,x) \in E$ and $ z^*\leq z < 0$ such that $x>z$. We conclude that, for $(u,x)\in E$ fixed such that $x>z^*$, the mapping $z\mapsto V_z(x,u)$ is non increasing on $[z^*,x\wedge 0)$. Hence, 
\begin{align*}
 V_z(x,u)\geq  \lim_{z\uparrow x\wedge 0}V_z(u,x)\geq 0
\end{align*}
for any $z^*\leq z\leq 0$ and $(u,x)\in E$ such that $x>z$. The proof is now complete.
\end{proof}
For ease of notation, we denote $V^*=V_{z^*}$. Note that for any $(u,x)\in E$,
\begin{align}
\label{eq:expressionforVstar}
V^*(u,x)= K^+(u,x)  -W^{(r)}(x)\int_{(0,\infty)}\int_0^{\infty}  G(v,y)  \frac{y}{v}e^{-rv}\P(X_{v} \in \dd y) \dd v  
-\int_{z^*}^0 G(0,y)W^{(r)}(x-y) \dd y.
\end{align}
Next, we show that the supermartingale property holds for $V^*$.

\begin{lemma}
\label{lemma:supermartingalepropertyforV*}
For any $N>0$ we have that the process $\{ Z_{t\wedge \tau_N^+}^*, t\geq 0\}$ is a supermartingale under $\P_{u,x}$, for each $(u,x)\in E$, where
\begin{align*}
Z_t^{*}=e^{-rt}V^*(U_t,X_t)+\int_0^t e^{-rs} G(U_s,X_s)\dd s.
\end{align*}
\end{lemma}
\begin{proof}
Due to the fact that $X$ is of infinite variation, we have that $\P(\tau_0^-=0)=1$ and $W^{(r)}$ is continuous on $\R$. Thus, $V_z$ is continuous on $E$ and $\lim_{h\downarrow 0} V^*(u,h)=V^*(0,0)$ for any $u\geq 0$. Moreover, since we are assuming that $\sigma>0$, we have that $W^{(r)}\in C^2(0,\infty)$ with $W^{(r)\prime}(0+)=2/\sigma^2$ (see Lemma 3.2 and Theorem 3.10 in \cite{kyprianou2011theory}). Hence, we have that $V_z(u,x)$ is  $C^{1,1}$ function on $[0,\infty)\times [0,\infty)$ and the second derivative $\frac{\partial^2}{\partial x^2} V^*(u,x)$ exists and is continuous on $(0,\infty)$ for all $u\geq 0$ (recall that we are assuming that $K^+$ is $C^{1,2}$ function on $[0,\infty)\times [0,\infty)$). On the other hand, for $z^*<x<0$ we have that		 
\begin{align}
\label{eq:derivativeofV*withrespecttox}
\frac{\partial}{\partial x} V^*(0,x)&=\int_{z^*}^x G(0,y)W^{(r)\prime}(x-y) \dd y,\\
\frac{\partial^2}{\partial x^2} V^*(0,x)&=\int_{z^*}^x G(0,y)W^{(r)\prime \prime}(x-y) \dd y+G(0,x)W^{(r)\prime}(0+).
\end{align}
Hence, we see that $V_z$ is $C^1$ function on $(-\infty,0]$ and its second derivative exists and is continuous on $(-\infty,0]\setminus \{z^*\}$. Furthermore, by applying formula \eqref{eq:calucationofintegraluptoinfinity2ndformula} to $H(0,0)$ (see equation \eqref{eq:definitionofH}) and from \eqref{eq:expressionforH} we see that 
\begin{align}
\label{eq:smoothpastingatzero}
\Phi'(r)\int_{(0,\infty)}\int_0^{\infty}  G(v,y) e^{-rv} \frac{y}{v}\P(X_{v} \in \dd y) \dd v 	 =H(0,0)=\lim_{\varepsilon\downarrow 0} \frac{K^+(0,\varepsilon)}{\psi'(\Phi(r)) W^{(r)}(\varepsilon)}=\Phi'(r)\frac{\sigma^2}{2}\frac{\partial }{\partial x} K^+(0,0+).
\end{align}
Hence, from the equality above and \eqref{eq:expressionforVstar} we deduce that 
\begin{align*}
\frac{\partial}{\partial x} V^*(0,0+)=\frac{\partial}{\partial x} V^*(0,0-).
\end{align*}

It can be easily seen that the process $\{ Z_{t\wedge \tau_{z^*}^-}^*,t\geq 0 \}$ is a martingale. Hence, by using standard arguments (cf. \cite{peskir2006optimal}, Section III.7.2 or \cite{lamberton2008critical}, Proposition 2.4), we deduce that 
\begin{align}
\label{eq:pdeforVstarandG}
\mathcal{A}_{U,X}(V^*)(u,x)+G=rV^*
\end{align}
for all $(u,x)\in E$ such that $x\geq z^*$, where from Corollary \ref{cor:infinitesimalgeneratorofgtX} we obtain that
\begin{align*}
\mathcal{A}_{(U,X)}(V^*)(u,x)&=\frac{\partial}{\partial u}V^*(u,x)\I_{\{x>0 \}}-\mu \frac{\partial}{\partial x} V^*(u,x)+\frac{1}{2}\sigma^2 \frac{\partial^2}{\partial x^2} V^*(u,x) \\
&\qquad +\int_{(-\infty,0)} \left( V^*(u,x+y) -V^*(u,x)-y\I_{\{y>-1\}}\frac{\partial }{\partial x} V^*(u,x) \right)\I_{\{x+y >0 \}}\Pi(\dd y)\\
& \qquad + \int_{(-\infty,0)} \left( V^*(0,x+y) -V^*(0,x)-y\I_{\{y>-1\}}\frac{\partial }{\partial x} V^*(0,x) \right)\I_{\{x\leq 0 \}}\Pi(\dd y)\\
& \qquad +  \int_{(-\infty,0)} \left( V^*(0,x+y)-V^*(0,x)-y\I_{\{y>-1\}} \frac{\partial }{\partial x}V^*(u,x) \right)\I_{\{x>0 \}}\I_{\{x+y <0 \}}\Pi(\dd y).
\end{align*}
Hence, for any $(u,x)\in E$, $t\geq 0$ and $N>0$, by applying the version of It\^o formula derived in Theorem \ref{thm:ItoformulaforgttXt} and letting $T=t \wedge \tau_N^+$, we deduce that, under $\P_{u,x}$,

\begin{align*}
e^{-rT}&V^*(U_T,X_T)\\
&=V^*(u,x)-\int_0^T e^{-rs}rV^*(U_s,X_s)\dd s + \int_{0}^{T } e^{-rs}\frac{\partial }{\partial u} V^*(U_s,X_{s})\I_{\{ X_s>0\}} \dd s\\
&\qquad+\int_{0}^{T}e^{-rs} \frac{\partial }{\partial x} V^*(U_{s-},X_{s-})\dd X_s+\frac{1}{2} \sigma^2 \int_{0}^{T } e^{-rs}\frac{\partial^2 }{\partial x^2} V^*(U_{s},X_{s})\dd s\\
& \qquad +\int_{[0,T] }e^{-rs} \int_{(-\infty,0)} \left( V^*(U_{s},X_{s^-}+y) -V^*(U_{s-},X_{s-})-y\frac{\partial }{\partial x} V^*(U_{s-},X_{s-}) \right)N(\dd s\times \dd y)\\
&=M_{T}-\int_0^T e^{-rs}[\mathcal{A}_{(U,X)}(V^*)(U_s,X_s)-rV^*(U_s,X_s)]\dd s \\
&=M_{T}-\int_0^T e^{-rs}G(U_s,X_s)\I_{\{X_s>z^* \}}\dd s,
\end{align*}
where $\{ M_{t\wedge \tau_N^+}, t\geq 0\}$ is a martingale and the last equality follows since $V^*(0,x)=0$ for all $x\leq z^*$ and then $\mathcal{A}_{U,X}(V^*)(0,x)=0$ for all $x\leq z^*$. Hence, we deduce that, for each $t\geq 0$ and $N>0$,
\begin{align*}
Z^*_{t\wedge \tau_N^+}=e^{-r(t\wedge \tau_N^+)}V^*(U_{t\wedge \tau_N^+},X_{t\wedge \tau_N^+})+\int_0^{t\wedge \tau_N^+}e^{-rs} G(U_s,X_s)	\dd s=M_{t\wedge \tau_N^+}+\int_0^{t\wedge \tau_N^+} e^{-rs}G(0,X_s)\I_{\{X_s\leq z^* \}}\dd s.
\end{align*}
Hence, since $G(0,x)\leq 0$ for all $x\leq z^*\leq  y_0$, we conclude that $\{Z_{t\wedge \tau_N^+}^*,t\geq 0\}$ is a supermartingale as claimed.
\end{proof}
Then the statements in Theorem \ref{thm:solutiontooptimalstopping} follow from Lemmas \ref{lemma:verificationlemma}, \ref{eq:uniquenessofeqautionandVispositive} and \ref{lemma:supermartingalepropertyforV*}. Finally, note from \eqref{eq:derivativeofV*withrespecttox} that the smooth fit property holds in this case.

\subsection{Proof of Proposition \ref{prop:SolutiontocorporatebankrupcyOS}}
\label{subsection:Proofofsolutionbankrupcy}
%For $v\geq 0$ and $y\geq 0$, we let $\delta(v,y)=ye^{\beta v}$ and $c(v,y)=K$ with $\beta\geq 0$ and $K\in (0,1)$.
Note that the optimal stopping problem is of the form \eqref{eq:optimalstoppingproblem}, with  $G(u,x)=e^{x+\beta u}-K$. 
%To be able to apply Theorem \ref{thm:solutiontooptimalstopping}, we assume that $r>\psi(1)+\beta$. Indeed, 
From the definition of $\psi$ and since $U_s\leq u+s$ under $\P_{u,x}$, for any $(u,x)\in E$, we have that the assumption $r>\psi(1)+\beta$ implies that
\begin{align*}
\E_{u,x}\left(\int_0^{\infty} e^{-rs}|G(U_s,X_s)|\dd s \right)
&\leq e^x\E\left(\int_0^{\infty} e^{-rs}e^{X_s+\beta(u+s)}\dd s \right)+\frac{K}{r}
%&=e^xe^{\beta u} \int_0^{\infty} e^{-(r-\psi(1)-\beta) s}\dd s +\frac{K}{r}
=\frac{e^{x+\beta u}}{r-\psi(1)-\beta}  +\frac{K}{r}<\infty.
\end{align*}
Due to \eqref{eq:qpotentialdensitytkillingonexiting0} we see that for any $x>0$ and $u>0$,
\begin{align*}
K^+(u,x)&=\E_{x}\left(\int_0^{\tau_0^-} e^{-rs}[e^{X_s+\beta(u+s)	}-K]\dd s\right)\\
&=\int_{(0,\infty)} e^{\beta u}e^{y} \int_{0}^{\infty}e^{-(r-\beta) s} \P_x(X_s \in \dd y,s<\tau_0^-)\dd s-\frac{K}{r}[1-\E_x(e^{-r\tau_0^-}\I_{\{\tau_0^-<\infty\}})]\\
%&=e^{\beta u} \int_{0}^{\infty} e^{y} \left[ e^{-\Phi(r-\beta) y}W^{(r-\beta)}(x) -W^{(r-\beta)}(x-y) \right]\dd z\\
%&\qquad +K\int_0^xW^{(r)}(y)\dd y-\frac{K}{\Phi(r)}W^{(r)}(x)\\
&=e^{\beta u} \left[ \frac{W^{(r-\beta)}(x)}{\Phi(r-\beta)-1} -\int_0^x e^y W^{(r-\beta)}(x-y)\dd y \right] +K\int_0^xW^{(r)}(y)\dd y-\frac{K}{\Phi(r)}W^{(r)}(x),
\end{align*}
where we used that $\Phi(r-\beta)>1$ due to the assumption $r>\psi(1)+\beta$ and since $\Phi$ is the right-inverse of $\psi$. On the other hand,
\begin{align*}
\int_z^{0}G(0,y) e^{-\Phi(r)y}\dd y=	\int_z^{0}[e^y-K] e^{-\Phi(r)y}\dd y=\frac{e^{-(\Phi(r)-1)z}-1}{\Phi(r)-1}-K\frac{e^{-\Phi(r)z}-1}{\Phi(r)}.
\end{align*} 
By differentiating $K^+$ (see \eqref{eq:smoothpastingatzero}) or by using Kendall's identity (see \eqref{eq:Kendallsidentity}), we can easily see that
\begin{align*}
\int_{(0,\infty)}\int_0^{\infty}  G(v,y)  \frac{y}{v}e^{-rv}\P(X_{v} \in \dd y) \dd v 
%&=\int_0^{\infty}  \int_{(0,\infty)}[e^y e^{-(r-\beta) v}-Ke^{-rv}]  \P(\tau_y^+ \in \dd v) \dd y\\
%&=\int_0^{\infty} \E\left( [e^y e^{-(r-\beta) \tau_y^+}-Ke^{-r\tau_y^+}] \I_{\{\tau_y^+<\infty \}}\right) \dd y\\
%&=\int_0^{\infty}  [e^y e^{-\Phi(r-\beta)y}-Ke^{-\Phi(r)y}] \dd z\\
&=\frac{1}{\Phi(r-\beta)-1}-\frac{K}{\Phi(r)}.
\end{align*}
Then, from Theorem \ref{thm:solutiontooptimalstopping} we know that $\tau_{z^*}^-$ is optimal, where in this case $z^*$ is the unique solution on $(-\infty,0)$ to the equation 
\begin{align*}
\frac{e^{-(\Phi(r)-1)z}}{\Phi(r)-1}-\frac{Ke^{-\Phi(r)z}}{\Phi(r)}+\frac{1}{\Phi(r-\beta)-1}-\frac{1}{\Phi(r)-1}=0.
\end{align*}
Lastly, from Theorem \ref{thm:solutiontooptimalstopping}, we see that the value function is given by
\begin{align*}
V(u,x)&= K^+(u,x)  -W^{(r)}(x)\int_{(0,\infty)}\int_0^{\infty}  G(v,y)  \frac{y}{v}e^{-rv}\P(X_{v} \in \dd y) \dd v  
-\int_{z^*}^0 G(0,y)W^{(r)}(x-y) \dd y\\
&=e^{\beta u} \left[ \frac{W^{(r-\beta)}(x)}{\Phi(r-\beta)-1} -\int_0^x e^y W^{(r-\beta)}(x-y)\dd y \right] +K\int_0^xW^{(r)}(y)\dd y-\frac{K}{\Phi(r)}W^{(r)}(x)\\
&\qquad-W^{(r)}(x)\left[\frac{1}{\Phi(r-\beta)-1}-\frac{K}{\Phi(r)} \right]-\int_{z^*}^0 [e^{y}-K]W^{(r)}(x-y) \dd y,
\end{align*}
for any $(u,x)\in E$. The proof is now complete.

\bibliographystyle{apalike}
\bibliography{lastzeroprocess}
%\nocite{*}

\end{document}